\documentclass[a4paper,11pt,DIV=11,%
abstract=on%
]{scrartcl}                  %
\usepackage{graphicx}
\usepackage{tabularx}
\usepackage{amsmath, amsthm, amssymb}
\usepackage{mathtools}
\usepackage[utf8]{inputenc}
\usepackage{todonotes}
\usepackage{subcaption}
\usepackage[ruled,vlined]{algorithm2e}
\usepackage{array,multirow,multicol}
\usepackage{hyperref}
\usepackage[
  backend=biber,
  style=alphabetic,
]{biblatex}
\bibliography{biblio}

\newtheorem{alg}{Algorithm}
\newtheorem{rmk}[alg]{Remark}
\newtheorem{prop}[alg]{Proposition}
\newtheorem{defi}[alg]{Definition}
\newtheorem{lem}[alg]{Lemma}

\newcommand{\argmin}{\operatornamewithlimits{argmin}}

\newcommand{\proj}{\operatorname{proj}}
\newcommand{\prox}{\operatorname{prox}}
\newcommand{\1}{\vec{1}}
\newcommand{\sign}{\operatorname{sgn}}

\newcommand{\TV}{\text{TV}}
\newcommand{\Hil}{\mathcal{H}}

\usepackage{xcolor}
\newcommand\fg[1]{{\color{black}#1}}
\newcommand\dt[1]{{\color{black}#1}}

\newcolumntype{C}[1]{>{\centering\arraybackslash}m{#1}}
\newcolumntype{L}[1]{>{\arraybackslash}m{#1}}
\newcolumntype{R}[1]{>{\raggedleft\arraybackslash}m{#1}}

\begin{document}

\title{Variational Graph Methods for\\ Efficient Point Cloud Sparsification%
}

\author{
	Daniel Tenbrinck%
	\thanks{Department Mathematik,
	Friedrich-Alexander-Universität Erlangen-Nürnberg\newline
	\texttt{daniel.tenbrinck@fau.de, martin.burger@fau.de}}%
	\and
	Fjedor Gaede%
	\thanks{Institut für Analysis und Numerik,
	Westfälische Wilhelms-Universität Münster\newline
	\texttt{fjedor.gaede@uni-muenster.de}}%
	\and
	Martin Burger%
	\footnotemark[1]%
}
\date{March 7th, 2019}

\maketitle

\begin{abstract}
In recent years new application areas have emerged in which one aims to capture the geometry of objects by means of three-dimensional point clouds. Often the obtained data consist of a dense sampling of the object's surface, containing many redundant 3D points. These unnecessary data samples lead to high computational effort in subsequent processing steps. Thus, point cloud sparsification or compression is often applied as a preprocessing step. The two standard methods to compress dense 3D point clouds are random subsampling and approximation schemes based on hierarchical tree structures, e.g., octree representations. However, both approaches give little flexibility for adjusting point cloud compression based on a-priori knowledge on the geometry of the scanned object. Furthermore, these methods lead to suboptimal approximations if the 3D point cloud data is prone to noise.
In this paper we propose a variational method defined on finite weighted graphs, which allows to sparsify a given 3D point cloud while giving the flexibility to control the appearance of the resulting approximation based on the chosen regularization functional. The main contribution in this paper is a novel coarse-to-fine optimization scheme for point cloud sparsification, inspired by the efficiency of the recently proposed Cut Pursuit algorithm for total variation denoising. This strategy gives a substantial speed up in computing sparse point clouds compared to a direct application on all points as done in previous works and renders variational methods now applicable for this task. We compare different settings for our point cloud sparsification method both on unperturbed as well as noisy 3D point cloud data.
\end{abstract}

\section{Introduction}
\label{intro}
Due to recent technological advances 3D depth sensors have become affordable for the broad public in the last years. Nowadays we are able to scan 3D objects by relatively cheap data acquisition devices, such as the Microsoft Kinect, or simply by using the cameras of our cell phones together with an elaborated reconstruction software \cite{Kolev2014}. Additionally, we benefit from the ever increasing computational power of general purpose computing hardware on smaller scales leading to a higher mobility of computing devices. This technological trend led to the rise of new application areas in which one aims to capture the geometry of scanned objects as 3D point clouds. Processing of raw point clouds is rather challenging as the points are unorganized and one has no clue on the underlying data topology a-priori. On the other hand, using a meshing algorithm as a preprocessing step on the point cloud often leads to artifacts and holes for non-uniformly distributed points, and thus should be avoided in these cases.

Based on the application one has to discriminate between two different types of 3D point clouds. First, there exist point cloud data of time-varying objects, i.e., the object to be captured is dynamic. This situation typically appears in the augmented reality entertainment environment, e.g., in 3D tele-immersive video \cite{Mekuria2017} or motion-controlled computer gaming as the Microsoft Kinect system. On the other hand, in science related areas one has to deal with static point clouds of single  objects or even whole landscapes. Especially the use of small aircrafts and drones together with 3D sensor technology, such as LiDAR, makes it possible to capture vast regions as point cloud data for geographic information systems. One well-known project that openly publishes the acquired point cloud data is OpenTopography \cite{opentopography}. It hosts datasets with currently approximately up to $284$ billion total LiDAR returns covering an area of roughly $26,000$ km$^2$. Processing and analysis of such massive point clouds is a major challenge due to the high computational costs. In this paper we will concentrate on the latter type of point cloud, i.e., static unorganized 3D point clouds.

As becomes apparent processing of massive 3D point clouds is very time consuming and hence there is a strong need for point cloud sparsification or compression. One possible strategy is to exploit redundancies within the sampling and reducing unnecessary 3D points only to the required level of detail. Ideally, one wants to find an approximation of a given point cloud, such that flat regions are described only by very few points, while feature-rich surface regions contain a higher density of 3D points and hence a better resolution of small details. It is feasible to first approximate the dense point cloud by polygonal meshes and subsequently apply mesh coarsening strategies, e.g., cf. \cite{Gooch2003}. However, triangulation is in general too computationally expensive to be used for massive 3D point cloud sparsification. Hence, other methods for compression directly work on the raw data of unorganized 3D point clouds. Typically, there are two standard methods, which both can be found, e.g., in the open source Point Cloud Library (PCL) \cite{PCL}. 
The first approach performs a random subsampling of a given point cloud based on a user-controlled fraction parameter assuming a uniform point distribution. It gets clear that one has little control and flexibility for point cloud sparsification in this simple method. Additionally, results are in general not reproducible as they are based on the actual seeding of the applied pseudo-random generators. The second standard strategy is based on the idea of partitioning the data into 3D cells of a fixed size, which can be controlled by the user. Methods such as an octree \cite{octree1980} data representation start by finding a 3D bounding box of the scanned object that contains all acquired 3D points (after an optional outlier removal). Then the bounding box is successively divided into equally-sized cells up to a level in which a subpartition becomes empty. A sparse version of the original dense point cloud can be obtained by choosing one level of the octree data representation. The disadvantage of these methods is that the orientation of the coordinate system containing the 3D point cloud has impact on the octree approximation results. Furthermore, one has no immediate influence on the distribution of the resulting point cloud sparsification and thus cannot control the density of 3D points in feature-rich surface regions.
 
The two standard methods for point cloud sparsification described above, i.e., random subsampling and octree data representation, are on the one hand able to provide compressed 3D point clouds relatively fast without the need to reconstruct the scanned object's surface by a polygon mesh or levelset function. On the other hand, they give the user little control about the level-of-detail of the resulting approximation. Furthermore, these methods are not suitable for point cloud sparsification of fine features in the presence of geometric noise perturbations as we will show in Section \ref{s:numerics}.

Since many applied problems can be cast into a variational model they play a key role in data sciences nowadays, e.g., in image processing or machine learning. Calculus of variations has a long history within the field of mathematical analysis and evolved an elaborated theory with many useful tools. In this setting one formulates a task as an optimization problem of functionals and then exploits the solid theory of variational methods to investigate the existence and uniqueness of optimal solutions, as well as to deduce algorithms to numerically compute the latter. Additionally, they provide more flexibility in controlling the appearance of solutions, e.g., by modeling a-priori knowledge with the help of properly chosen regularization functionals. For this reason the application of variational methods would be beneficial for point cloud compression. However, since 3D point clouds are unorganized and have very little structure in general a translation of traditional variational methods is not directly possible as they are formulated for data with a structured topology, e.g., images or voxel grids.

One way to tackle this problem is to model the data by a finite weighted graph and then translate variational methods and partial differential equations to the abstract structure of the graph. This has been initially proposed and investigated in the seminal works in \cite{elmoataz2008,gilboa2008}. Yet, variational graph methods are computationally infeasible for 3D point cloud data. Applying a variational denoising model on a dense point cloud using convex, non-smooth regularization functionals will lead to a sparse approximation as reported in previous works discussed below. However, the process of numerically solving the involved equations is computationally very intense as we show in this work. Depending on the number of samples in the original point cloud users may have to wait for hours in order to get a sparse approximation using variational methods for this task. This is our motivation for proposing a more efficient strategy to solve variational graph problems on large multi-dimensional data sets. 

\subsection{Related work}
In order to tackle variational problems on finite weighted graphs the basic graph operators were introduced independently by Elmoataz, Lezoray and Bougleux in \cite{elmoataz2008} and by Gilboa and Osher in \cite{gilboa2008}. These definitions were used to introduce the notion of a graph $p$-Laplacian as a one-dimensional vertex function, which has been applied for solving imaging problems on graphs, such as denoising, segmentation, and simplification (cf. \cite{elmoataz2015} and reference therein). Subsequently, the anisotropic graph $p$-Laplacian, i.e., each coordinate is treated independently, has been translated by Lozes et al. to three dimensional meshes, polygonal curves and 3D point clouds represented by graphs \cite{lozes_2014_imagestuff,lozesdiss}. Using this approach the authors were able to tackle imaging problems such as morphological inpainting, restoration, and denoising for surfaces and point clouds. 
Particularly, they showed preliminary results of using a non-convex variational model for 3D point cloud sparsification, i.e., the graph $p$-Laplacian for $p<1$.
In \cite{danielundronnie2017, danielundronnie2018} Bergmann and Tenbrinck extended the graph framework to manifold-valued data and showed results for denoising and inpainting of semi implicitly given surfaces, surface normals and phase-valued data. 
Since the method proposed in this paper contains a denoising step we mention in the following related work on point cloud and mesh denoising. From a large amount of proposed denoising methods we will list only a few important representatives. In \cite{bilateralfiltering} Fleishman et al. introduced a bilateral filtering method which filters vertices in the normal direction by using the respective local neighborhoods. Due to its simplicity, efficiency, and a good feature preservation it was basis for many later works. Mattei et al. introduced in \cite{denoisingrpca} a point cloud denoising method with a moving robust principal component analysis, which does not require oriented normals and minds local and nonlocal features. Sharp edges are preserved by minimizing a weighted \dt{total variation} regularization. Recently, Yadav et al. proposed a normal voting tensor and binary optimization in \cite{meshdenoisingnormalvotingtensor}. They also provide a rich quantitative comparison with other denoising methods.
In \cite{l0denoising2015} Sun et al. present a denoising method based on $\ell_0$ regularization. This is done by computing the normals of the surface and then denoising the point cloud by allowing movement only in the normal direction. Both steps are done with a $\ell_0$ regularization.
 Zhong et al. \cite{denoisingl1} provide an algorithm that decouples noise and features from the data. For this sake they use a discrete Laplace regularization to get the underlying smooth surface and then recover the sharp features by a compressed sensing approach. 
 
A research field known as `stippling` is closely related to the task of point cloud sparsification in which one aims to approximate arbitrary density functions by point distributions. There exists a heuristic method known as Lloyd's algorithm \cite{Lloyd} that aims to find barycenters of partitions based on $k$-means clustering and the related Voronoi cells. More sophisticated methods extend this approach via a variational formulation based on optimal transport and Laguerre cells \cite{blueOptimalTransport, OptimalTransportPointCloud}.

In this paper we are inspired by the general framework of the Cut Pursuit algorithm first proposed in \cite{cutloic}. Landrieu and Obozinski introduced two algorithms with Cut Pursuit methods to solve minimization problems regularized with total variation and $\ell_0$ regularization for the Mumford-Shah penalization of the boundary length. Additionally, Raguet and Landrieu present in \cite{raguetcut} an extension of the Cut Pursuit method for an additional non-differentiable term. This term is given by a vertex function which is said to be non-differentiable, but for which every directional derivative exists. To solve the resulting model, they introduced a ternary cut and proved convergence of this algorithm. Tests on brain source identification in electroencephalography and 3D point cloud labeling demonstrate an enormous speed up compared to the well-known preconditioned primal-dual algorithm \cite{primaldual,diagonalPPD} and the preconditioned forward Douglas-Rachford splitting \cite{raguet2013generalized,raguet2017note} on graphs. This speed up motivates our work on efficient methods for 3D point cloud sparsification.

\subsection{Own contributions}
In this paper we overcome the problems discussed above by proposing an optimization technique that follows a coarse-to-fine strategy as sometimes used in other imaging tasks, e.g., multiscale methods for optical flow computation \cite{opticalflow_1,opticalflow_2}. Our method is based on an alternating iterative scheme that is inspired by the recently proposed Cut Pursuit algorithm discussed above. In contrast to the seminal work by Landrieu et al. in \cite{cutloic} we decouple the graph cut partitioning step and the denoising step of Cut Pursuit even further by introducing two different regularization parameters. This allows for additional flexibility in the control of the appearance of the sparse 3D point cloud \dt{, i.e., we are able to steer both the compression rate as well as the smoothness properties of the point cloud independently. 

Additionally, we introduce a new regularization term for Cut Pursuit that can be interpreted as weighted $\ell_0$ regularization. We investigate the properties of this regularization term and derive an algorithm for point cloud sparsification. The $\ell_0$ regularization has the advantage that is yields very good results for point cloud sparsification, while being efficiently to compute. Indeed, this proposed method leads to a speed up of two orders of magnitude and thus is valuable for applications in which processing and analysis of point clouds in near-realtime is mandatory. We compare this novel regularization technique to traditional ones, e.g., isotropic $\ell_2$ (Tikhonov) or anisotropic/isotropic $\ell_1$ (total variation) regularization.

Another contribution is a new heuristic method to perform graph cuts} in the case of isotropic regularization functionals, which induce a challenging coupling of the data coordinates.

Using the proposed method we are able to compress big point cloud data with an enormous speed up compared to applying the same variational denoising method directly on the full point cloud as performed, e.g., in \cite{elmoataz2008}. We also introduce a preconditioning scheme for the arising optimization problems, which additionally increases the numerical efficiency. This overall efficiency boost renders our method a strong alternative to the current standard methods for point cloud sparsification. In particular we show that in one special case our method performs the octree sparsification strategy, and hence can be seen as generalization of well-known standard methods. 

Finally, we propose a debiasing step for the reconstruction of very noisy point cloud data that allows to correct from typical bias effects of non-smooth regularization functionals such as total variation regularization.

Note that by using finite weighted graphs for modeling the point cloud data the proposed optimization scheme is not restricted to unorganized 3D point clouds. First, if a 3D surface is given as a triangulated mesh then one can directly use the edges and vertices of this polygon mesh as a graph and perform the same steps as described in this paper. Second, as our method is not bounded to three-dimensional data one could use the same method for sparsification of high-dimensional point cloud data, e.g., feature points in machine learning applications. %

\subsection{Outline}
The outline of this paper is as follows. In Section \ref{s:graphs} we discuss how variational models and partial differential equations can be translated to finite weighted graphs. We also introduce an anisotropic and isotropic $p$-Laplace operator for a multidimensional vertex function $f$. Subsequently, we define in Section \ref{s:cut_pursuit} the variational model we apply for point cloud sparsification as well as the basic idea of the Cut Pursuit algorithm. For the denoising step of this method we deduce the needed updates for a primal-dual optimization strategy on graphs and describe a preconditioning scheme for the optimization problem. In Section \ref{s:numerics} we perform various numerical experiments to demonstrate the efficiency of the proposed optimization strategy on dense 3D point clouds. We compare different compression methods and regularization functionals on both unperturbed as well as noisy point cloud data. We conclude this paper by a short discussion of possible extensions to our method in Section \ref{s:discussion}.

\section{Finite weighted graphs}
\label{s:graphs}
Finite weighted graphs play an important role in many different fields of research today, e.g., image processing~\cite{elmoataz2008,gilboa2008}, machine learning~\cite{ZB11,BM16,GSBLB16,BH09}, or network analysis~\cite{LC12,M14,SNFOV13}. 
Their key advantage is that they allow to model and process discrete data of arbitrary topology. 
Recently, there has been a strong effort to translate well-studied tools from applied mathematics to finite weighted graphs, e.g., variational methods and partial differential equations. 
This enables one to apply these tools to many new application areas that cannot be tackled directly by traditional data modeling techniques, i.e., grids and finite elements. 
Furthermore, graphs allow to exploit repetitive patterns or self-similarity in the data by building edges between related data points. 
Hence, they can be used to process both local as well as nonlocal problems in the same unified framework.
Due to the abstract nature of the graph structure one may build hierarchical graphs to represent whole sets of entities by a single vertex, e.g., image regions consisting of neighboring pixels \cite{elmo_hierachical}. These coarse data representations lead to very efficient optimization techniques as we will discuss in Section~\ref{s:cut_pursuit} below.

Although the exact description of finite weighted graphs is dependent on the application, there exists a common consent of basic concepts and definitions in the literature~\cite{elmoataz2008,GGOB14,gilboa2008}.
In the following we recall these basic concepts and the respective mathematical notation, which we will need to introduce the proposed graph methods for point cloud sparsification below.

\subsection{Basic graph terminology}
\label{ss:basic_notation}
A \emph{finite weighted graph}~$G$ is defined as a triple~$G = (V, E, w)$ for which
\begin{itemize}
  \item[$\bullet$] $V = \{1, \dots, n\}, n\in\mathbb{N}$, is a finite set of indices denoting the \emph{vertices},
  \item[$\bullet$] $E \subset V \times V$ is a finite set of (directed) \emph{edges} connecting a subset of vertices,
  \item[$\bullet$] $w \colon E \rightarrow \mathbb{R}^+$ is a nonnegative \emph{weight function} defined on the edges of the graph.
\end{itemize}
For given application data each \emph{graph vertex}~$u \in V$ typically models an entity in the data structure, e.g., elements of a finite set, pixels in an image, or nodes in a network. 
It is important to distinguish between abstract data entities modeled by graph vertices and attributes associated with them. 
The latter can be modeled by introducing vertex functions as defined below.
A \emph{graph edge} \((u,v)\in E\) between a start node \(u\in V\) and an end node \(v\in V\) models a relationship between two entities, e.g., geometric adjacency, entity interactions, or similarity depending on the associated attributes. 
In our case, we consider graphs with \emph{undirected} edges, i.e.,~\( (u,v)\in E \Rightarrow (v,u)\in E\) in general.

A node $v \in V$ is called a \emph{neighbor} of the node~$u \in V$ if there exists an edge $(u,v) \in E$. 
For this relationship we use the abbreviation $v \sim u$, which reads as “\(v\) is a neighbor of \(u\)”.
If on the other hand \(v\) is not a neighbor of \(u\), we use~\(v\not\sim u\).
We define the \emph{neighborhood}~\(\mathcal N(u)\) of a vertex~$u \in V$ as $\mathcal{N}(u) \coloneqq \{ v\in V \colon v \sim u \}$.
The \emph{degree} of a vertex $u \in V$ is defined as the amount of its neighbors $\operatorname{deg}(u) = \lvert \mathcal{N}(u)\rvert$.

\subsection{Vertex and edge functions}
\label{ss:graph_functions}
To relate the abstract structure of a finite graph to some given data, one can introduce vertex and edge functions. Let $\mathcal{H}(V;\mathbb{R}^d)$  be the Hilbert space of vector-valued functions on the vertices of the graph, i.e., each function $f: V \to \mathbb{R}^d$ in $\mathcal{H}(V;\mathbb{R}^d)$ assigns a real vector $f(u)$ to each vertex $u \in V$. In the following will denote $\mathcal{H}(V;\mathbb{R}^d)$ with  $\mathcal{H}(V)$ for the sake of simplicity. For a function $f \in \mathcal{H}(V)$ the $\ell_p$- and $\ell_\infty$-norm of $f$ are given by:
\begin{align}
\label{eq:p_norm}
\begin{split}
  \|f\|_p \ &= \ \Bigl( \sum\limits_{u\in V} \|f(u)\|^p \Bigr)^{1/p} \ , \quad \text{ for } 1 \leqslant p < \infty \ ,\\
  \|f\|_\infty &= \  \max\limits_{u\in V}\bigl(\|f(u)\|\bigr)\ , \hspace{1.04cm} \text{ for }  p = \infty \ .
\end{split}
\end{align}
The Hilbert space $\mathcal{H}(V)$ is endowed with the following inner product
\begin{equation*}
\langle f,g \rangle_{\mathcal{H}(V)} = \sum_{u \in V}\langle f(u),g(u)\rangle_{\mathbb{R}^d},
\end{equation*}
with $f,g \in \mathcal{H}(V)$.

Similarly, let $\mathcal{H}(E;\mathbb{R}^m)$ be the Hilbert space of vector-valued functions defined on the edges of the graph, i.e., each function $F: E \to \mathbb{R}^m$ in $\mathcal{H}(E;\mathbb{R}^m)$ assigns a real vector $F(u,v)$ to each edge $(u,v) \in E$. As before we will abbreviate $\mathcal{H}(E;\mathbb{R}^m)$ by $\mathcal{H}(E)$. The Hilbert space $\mathcal{H}(E)$ is then endowed with the following inner product: 
\begin{equation*}
\langle F, G \rangle_{\mathcal{H}(E)} = \sum_{(u,v) \in E} \langle(F(u,v),G(u,v)\rangle,
\end{equation*}
for $F,G\in \mathcal{H}(E)$. It is easy to show that the dual space of $\mathcal{H}(V)$ is $\mathcal{H}(E)$.

To model the significance of a relationship between two connected vertices with respect to an application dependent criterion one introduces a \emph{weight function} $w \in \mathcal{H}(E;\mathbb{R})$.
Often, the weight function is chosen as a similarity function based on the attributes of the modeled entities, i.e., by the evaluation of associated vertex functions. 
For these cases the weight function $w$ is chosen such that it takes high values for important edges, i.e., high similarity of the involved vertices, and low values for less important ones. 
In many applications one normalizes the values of the weight function by $w \colon E \rightarrow [0,1]$.
Note that a natural extension of the weight function to the full set~$V \times V$ is given by defining~$w(u,v) = 0$, if~$v \not\sim u$ or~$u = v$ for any~$u,v \in V$.
Then the edge set of the graph can simply be characterized as~$E = \{ (u,v) \in V \times V \colon w(u,v) > 0 \}$. 
Often it is preferable to use \emph{symmetric} weight functions, i.e., \(w(u,v)=w(v,u)\). 
This also implicates that~$v \sim u \Rightarrow u \sim v$ holds for all $u,v \in V$ and thus all directed graphs with symmetric weight function can be interpreted as~\emph{undirected} graphs

\subsection{First-order partial difference operators on graphs}
\label{ss:first_order_differential}
Using the basic concepts from the previous sections we are able to introduce the needed mathematical tools to translate standard differential operators from the continuous setting to finite weighted graphs. The fundamental elements for this translation are first-order partial difference operators on graphs, which have been initially proposed in \cite{elmoataz2008,gilboa2008}. In the following we assume that the considered graphs are connected, undirected, with neither self-loops nor multiple edges between vertices.

Let $G = (V,E,w)$ be a finite weighted graph and let $f\in \mathcal{H}(V)$ be a function on the set of vertices $V$ of $G$. Then one can define the \textit{weighted partial difference} of $f$ at a vertex $u \in V$ in direction of a vertex $v \in V$ as:
\begin{equation}
\label{eq:finite_difference}
  \partial_vf(u) \ = \ \sqrt{w(u,v)}\left(f(v)-f(u)\right) \ .
\end{equation}
As for the continuous definition of directional derivatives, one has the following properties $\partial_vf(u) = -\partial_uf(v)$, $\partial_uf(u) = \vec{0}$, and if $f(u) = f(v)$ then $\partial_vf(u) = \vec{0}$.

Based on the definition of weighted partial differences in \eqref{eq:finite_difference} one can straightforwardly introduce the \textit{weighted gradient operator} on graphs $\nabla_w: \mathcal{H}(V) \rightarrow \mathcal{H}(E)$, which is simply defined as the weighted finite difference on the edge $(u,v) \in E$, i.e.,%
\begin{equation}
\label{eq:gradient}
  (\nabla_wf)(u,v) \ = \ \partial_vf(u)
\end{equation}
It gets clear that this operator is linear.
The {\emph adjoint operator} $\nabla_w^*\colon\mathcal{H}(E)\rightarrow \mathcal{H}(V)$ of the weighted gradient operator is a linear operator defined by 
\begin{equation*}
\langle \nabla_wf,G\rangle_{\mathcal{H}(E)}=\langle f,\nabla_w^*G\rangle_{\mathcal{H}(V)} \quad \text{ for all } f\in \mathcal{H}(V), G\in \mathcal{H}(E).
\end{equation*}
Note that for undirected graphs with a symmetric weighting function $w \in \mathcal{H}(E,\mathbb{R})$ the adjoint operator $\nabla_w^*$, of a function~$G\in \mathcal{H}(E)$ at a vertex~$u\in V$ has the following form:
\begin{equation}
\label{eq:adjoint}
  (\nabla_w^*G)(u) \ = \ \sum_{v\sim u}{\sqrt{w(u,v)}(G(v,u)-G(u,v))}.
\end{equation} 
One can then define the weighted divergence operator on graphs via the adjoint operator as $\operatorname{div}_w \coloneqq -\nabla_w^*$. The divergence on a graph measures the net outflow of an edge function in each vertex of the graph.

To measure the variation of a vertex function $f \in \mathcal{H}(V)$ with values in $\mathbb{R}^d$ we introduce a family of $p$-$q$-norms based on the weighted gradient operator for $p,q \geq 1$ as follows:
\begin{align}
\label{eq:p-q-norm}
\begin{split}
\Vert \nabla_w f \Vert_{p;q} &= \Bigl(\sum_{(u,v)\in E}  \Vert \nabla_wf(u,v)) \Vert^p_q \Bigr)^\frac{1}{p} \\
&= \left[\sum_{u\in V} \sum_{v\sim u} w(u,v)^\frac{p}{2} \Vert f(v) - f(u) \Vert_q^p \right]^\frac{1}{p}\\
&= \left[\sum_{u\in V} \sum_{v\sim u} \left( \sum_{j=1}^{d}  w(u,v)^\frac{q}{2} \vert f(v)_j - f(u)_j \vert^q \right)^\frac{p}{q} \right]^\frac{1}{p}.
\end{split}
\end{align}
\dt{
The advantage of using the general $p$-$q$ norm \eqref{eq:p-q-norm} is that it captures many interesting regularization terms from the literature, e.g., classical Tikhonov regularization ($p=q=2$), anisotropic total variation regularization ($p=q=1$), and isotropic total variation regularization ($p=1, q=2$). These regularization terms are widely used for denoising monochromatic and also vector-valued signals, e.g., see \cite{elmoataz2008,Moeller14} and references therein.
Depending on the choice of the parameters $p,q \geq 1$ we are able to analyze different regularization techniques in a unified framework in Section \ref{s:cut_pursuit} and incorporate different a-priori knowledge about the expected solutions of point cloud sparsification in Section \ref{s:numerics}.
}

\subsection{Graph $p$-Laplace operator}
\label{ss:graph_p_laplacian}
The continuous $p$-Laplace operator is an example of a second-order differential operator that can be defined on finite weighted graphs. It allows the translation of various partial differential equations to the graph setting and it has been used for applications in machine learning and image processing. For a detailed discussion of the graph $p$-Laplacian and its variants we refer to \cite{elmoataz2015}.

Based on the first-order partial difference operators introduced in \eqref{eq:gradient} and \eqref{eq:adjoint} one is able to formally derive a family of graph $p$-Laplace operators $\Delta_{w,p} \colon \mathcal{H}(V) \rightarrow \mathcal{H}(V)$ by minimization of the $p$-$q$-norm defined in \eqref{eq:p-q-norm} above. There are two special cases that lead to different definitions of the graph $p$-Laplace operator. For this paper we will derive a multidimensional version of the real $p$-Laplacian introduced in \cite{elmoataz2008}. For the sake of simplicity we assume that the finite weighted graph $G=(V,E,w)$ is undirected and has a symmetric weight function $w\in \mathcal{H}(E)$, i.e. $w(u,v) = w(v,u)$, in the following. Let $\vert\nabla_wf(u,v)\vert$ denote the point-wise absolute value in the gradient $\nabla_w f(u,v)$ and $\cdot$ be a point-wise product between vectors. Then we define 
\begin{align}\label{eq:pq-laplace}
\begin{split}
\Delta_{w,p;q} f(u) &=\frac{1}{2}\operatorname{div}_w\left(\Vert \nabla_w f\Vert_q^{p-q} \nabla_w f \cdot \vert \nabla_wf\vert^{q-2}\right)\\
& = \sum_{v\sim u} w(u,v)^\frac{p}{2}  \left\Vert f(v)-f(u) \right\Vert_q^{p-q}  (f(v)-f(u)) \cdot \vert f(v)-f(u)\vert^{q-2}.
\end{split}
\end{align}
More details on the computation of \eqref{eq:pq-laplace} can be found in Appendix \ref{app:pqlap}.\par
For the special case $p=q$ we get the multidimensional \textit{anisotropic} $p$-Laplacian given as:
\begin{align}\label{eq:pq-laplace_aniso_multi}
\Delta^a_{w,p} f(u) = \sum_{v\sim u} w(u,v)^\frac{p}{2}  \nabla_w f(u,v) \cdot \vert \nabla_wf(u,v)\vert^{p-2}.
\end{align}
On the other hand, if we choose $q=2$ we get the multidimensional \textit{isotropic} $p$-Laplacian
\begin{align}\label{eq:p-laplace_iso_multi}
\begin{split}
\Delta^i_{w,p} f(u) = \sum_{v\sim u} w(u,v)^\frac{p}{2}  \left\Vert f(v)-f(u) \right\Vert_2^{p-2}  \nabla_w f(u,v).
\end{split}
\end{align}

Note, in the terminology of \cite{elmoataz2015} both of these $p$-Laplacian would be called \emph{anisotropic} since the authors discussed only the one-dimensional case of vertex and edge functions. In this context the term \emph{isotropic} describes the relationship between neighbor vertices. In our more general case we relate the term \emph{isotropic} to the coupling of coordinates along all dimensions. Also note that in the anisotropic case the inner terms decouple and allow for an pairwise independent computation.

For $p=q=2$ we obtain a notion of a classical linear operator known as the unnormalized graph Laplacian, now in multiple dimensions, as
\begin{equation*}
  \Delta_{w}f(u) \ = \  \sum_{v\sim u} w(u,v)\left(f(v) - f(u)\right).
\end{equation*}

\section{Cut Pursuit for point cloud sparsification}
\label{s:cut_pursuit}
In this section we present our methodology for efficiently computing sparse point clouds using variational graph methods. Our approach is inspired by the Cut Pursuit algorithm proposed in \cite{cutloic,LoicDiss}. It can be applied for minimizing an energy functional $J$ on a finite weighted graph $G = (V,E,w)$ on the set $\mathcal{H}(V)$ given as
\begin{align}\label{eq:generalproblem}
\Bigl\lbrace J(f) = D(f,g) + \alpha  R(f) \Bigr\rbrace \ \rightarrow \ \underset{f\in \mathcal{H}(V)}{\operatorname{argmin}},
\end{align}
for which $\alpha >0$ is a fixed regularization parameter, $D$ is a differentiable, convex data fidelity term with the original data given as $g$, and $R$ is a convex regularization functional, which is decomposable into differentiable and non-differentiable parts and for which directional derivatives in $\mathcal{H}(V)$ exist. 

For point cloud sparsification we use a variational model that has already been proposed for this task in \cite{elmoataz2008}. 
However, in this paper we investigate a more general variant of this model. In particular, we focus on optimizing the following family of variational denoising problems for a fixed regularization parameter $\alpha > 0$
\begin{equation}
\label{eq:problem}
\Bigl\lbrace J(f) = \frac{1}{2}\Vert f-g\Vert_2^2 + \frac{\alpha}{2p}\Vert \nabla_w f \Vert_{{p;q}}^p \Bigr\rbrace \ \rightarrow \ \underset{f\in \mathcal{H}(V)}{\operatorname{argmin}}~ \tag{P}
\end{equation}
for \fg{$q, p\geq 1$} using the notation introduced in Section \ref{ss:graph_functions}, i.e., we minimize a $L^2$ data fidelity term together with a convex, (possibly) non-smooth regularization functional. Many algorithms for computing solutions to \eqref{eq:problem} are known in the literature, cf., e.g., \cite{chambolle_introduction} and references therein.

\subsection{Optimization via Cut Pursuit}\label{ssec:cutpursuit}
Instead of computing respective minimizers of the variational problem \eqref{eq:problem} by performing a (potentially) computational-heavy optimization directly on all vertices of the graph $G$, we follow the idea of the Cut Pursuit algorithm proposed by Landrieu and Obozinski in \cite{cutloic}. Here, the minimization of $J$ is done by an alternating iteration scheme that successively divides the set of vertices $V$ into increasingly smaller subsets and solves the original optimization problem on the relatively few vertices that represent the subsets induced by the partition. 
For this we first need the notion of the directional derivative of $J$ in terms of vertex functions.
\begin{defi}\textbf{(Directional derivative)}\\
	Let $J: \mathcal{H}(V) \rightarrow \mathbb{R}$ be a functional. Then the \emph{directional derivative} at a point $f\in \mathcal{H}(V)$ in direction $\vec{d}\in \mathcal{H}(V)$ is defined as
	\begin{equation*}
	J'(f;\vec{d}) \ = \ \lim_{t\rightarrow 0} \frac{J(f+t\vec{d})-J(f)}{t}
	\end{equation*}
	if the limit exists.
\end{defi}

\fg{
\dt{In the following, we extend the derivation of} the Cut Pursuit algorithm proposed in \cite{cutloic} \dt{to the case of the general regularization term}
\begin{equation}\label{eq:regularizer}
    R(f) = \frac{1}{p}\Vert \nabla_w f \Vert_{{p;q}}^p.
\end{equation} 
\dt{We begin by introducing the needed notation and basic }definitions. \fg{We start by defining two sets of edges in which the regularization functional $R$ is differentiable and non-differentiable as $S$ and $S^c$, respectively. Also we will denote $$ w(A,B) = \sum_{(u,v)\in A\times B} w(u,v).$$} Since we want to \dt{compute the solution of \eqref{eq:problem}} via successive splitting of the vertex set $V$ we introduce \dt{the partition of $V$ into subsets $A_1,\ldots,A_m \subset V$ as:}
\begin{equation}
    \label{eq:Pi}
    \Pi \ := \ \big\{ A_i \subset V \: | \: i\in I = \lbrace1,\ldots,m\rbrace, ~A_i \cap A_j = \emptyset \text{ for } i\neq j, ~V = \dot{\cup}_{i=1}^m A_i \big\}.
\end{equation}
\dt{Based on the} partition $\Pi$ we define the \emph{reduced graph} $G_r = (V_r, E_r, w_r)$ which is given by the vertex set $V_r \coloneqq \Pi$, the edge set
	\begin{equation}\label{eq:reducedEdgeset}
	E_r = \big\{ (A,B) \in V_r \times V_r \ \big| \  \big(A\times B\big)\cap E \neq \emptyset\big\},
	\end{equation}
and the reduced weight function as
\begin{equation}\label{eq:reducedWeights}
w_r\colon E_r \rightarrow \mathbb{R}_+ \ \text{ with } \ w_r(A,B) = \sum_{(u,v) \in (A\times B)\cap E} w(u,v) \ .
\end{equation}
\dt{Furthermore, we} define the characteristic function $1_A$ \dt{for a subset $A \subset V$} as
\begin{align}
    1_A(u) = \begin{cases}
        1, \quad\text{if } u \in A\\
        0, \quad\text{else.}
    \end{cases}
\end{align}
With this setting we can say a function $f\in \Hil(V)$ is piecewise constant on the sets $A\in \Pi$ with a value $c_A \in \mathbb{R}^d$ if
\begin{align}
    f = \sum_{A\in \Pi} 1_A c_A. 
\end{align}

Thus, we can define a vertex function $c: \Pi \rightarrow \mathbb{R}^d$ on the reduced set $V_r = \Pi$ \dt{such that} $c\in \Hil(\Pi)$. Let $m = \vert \Pi\vert$ \dt{be the cardinality of $\Pi$, i.e., the number of subsets $A_i \subset V$ induces by the partition $\Pi$,} then $\Hil(\Pi) \simeq \mathbb{R}^{m\times d}$ and we can write $c = \big( c_A\big) \in \mathbb{R}^{m\times d}$ as a vector.
In Section \ref{ssec:red_problem} we will discuss in detail how the reduced \dt{vertex} functions in $\Hil(\Pi)$ \dt{are related} to piecewise constant \dt{vertex} functions in $\mathcal{H}(V)$.
}
\fg{
\dt{So far we have not required that the partition $\Pi$ is an optimal partition of $V$ for solving \eqref{eq:problem}. Thus, in the following we aim} to find a subset $B\in \mathcal{P}(V)$ that splits the current partition \dt{$\Pi_k$ into new subsets} at the borders of $B$ and its complement $B^c$ \dt{in a way that decreases the energy functional $J$} the most \dt{and leads to a new partition $\Pi_{k+1}$}. \dt{The following proposition states how one can compute such an optimal subset $B$}. To learn more about the \dt{exact derivation of this result} we \dt{refer the interested reader} to Appendix \ref{app:deriviationCutPursuit}.
}
\clearpage
\dt{
\begin{prop}\label{prop:derivationOfCP}~\\
    Let $\Pi$ be the current partition of $V$ and $c\in \Hil(\Pi)$ a \dt{vertex} function on the reduced graph $G_r = (V_r, E_r, w_r)$. Let $f_\Pi\in \mathcal{H}(V)$ be a vertex function that is piecewise constant on the sets in $\Pi$ and is given as $f_\Pi = \sum_{A\in \Pi} 1_Ac_A$. Let $g\in \mathcal{H}(V)$ be \dt{a vertex function representing the given data} and $p, q\geq 1$ . \fg{Also let $\gamma_B^A, \gamma_{B^c}^A \in \mathbb{R}_+^d$ two descent directions for each set $A\in\Pi$.} Then a subset $B^*\in \mathcal{P}(V)$ that decreases the energy 
    $$ J(f) = D(f,g) + \frac{\alpha}{2} R(f)$$ with the regularizer $R(f) = \frac{1}{p}\Vert \nabla_w f\Vert^p_{p;q}$ the most can be found by solving 
    \begin{align}\label{eq:prop}
        B^*\in \argmin_{B\in \mathcal{P}(V)} \ \langle \nabla D(f_\Pi,g) + \alpha\nabla R_S(f_\Pi), \vec{\gamma} \rangle + \frac{\alpha}{2} R'_{S^c}(f_\Pi; \vec{\gamma}). 
    \end{align}
    with $$\vec{\gamma} = \sum_{A\in \Pi} 1_{A\cap B} (\gamma_B^A + \gamma_{B^c}^A).$$
\end{prop}
\begin{proof}
    see Appendix \ref{app:deriviationCutPursuit}
\end{proof}
}
\fg{
\dt{Proposition \ref{prop:derivationOfCP} shows us} how to find a new partition from a given $f_\Pi$, \dt{which directly leads to the question} of how to find an optimal $f_\Pi$ for some given partition $\Pi$. \dt{This question can be formulated as the following} optimization problem 
\begin{align}
	\label{eq:optimal_fpi}
    f_\Pi = \argmin_{f\in \Hil(\Pi)} \ D(f,g) + \frac{\alpha}{2} R_r,
\end{align}
which is defined on the reduced graph $G_r = (V_r, E_r, w_r)$ with $$ R_r(f) = \frac{1}{p} \Vert \nabla_{w_r} f\Vert_{p;q}^p.$$ \dt{The solution $f_\Pi$ of \eqref{eq:optimal_fpi}} can be then plugged into formula \eqref{eq:prop} and \dt{consequently} a new partition can be computed. 

We have gathered the \dt{necessary} ingredients to formulate the \dt{original} Cut Pursuit algorithm proposed in \cite{cutloic} to solve \eqref{eq:generalproblem} \dt{for the special case of} $p=q=1$ and $\vec{\gamma} = 1_B$.}
\begin{alg}[Cut Pursuit]
\label{alg:cut_pursuit}
\begin{gather*}
\left\lbrace J'(f_\Pi; \1_B) \right.  \left.= \langle \nabla D(f_\Pi,g) , \1_B  \rangle + \alpha\langle \nabla R_S(f_\Pi), \1_B \rangle + \frac{\alpha}{2} R_{S^c}'(f_\Pi; \1_B) \right\rbrace \rightarrow \min_{\dt{B \in \mathcal{P}(V)}}\\[0.2cm]
f_\Pi \ = \ \operatornamewithlimits{arg\min}_{f\in \mathcal{H}(\Pi)} D(f,g)+ \fg{\frac{\alpha}{2}} R(f).
\end{gather*}
\end{alg}
\dt{
The subset $B$ is a-priori unknown and has to be chosen from all $2^n$ possible subsets of the power set $\mathcal{P}(V)$. The indicator function $\1_B$ can be interpreted as unknown descent direction of the energy functional $J$. The set $\Pi$ is again the current partition of $V$. 
\fg{
The Alternating Minimization Scheme} \ref{alg:cut_pursuit} is an iterative method to compute a new partition $\Pi^{k+1}$ of $V$ by refining the current partition $\Pi^k$ based on a minimum graph cut that induces the set $B \in \mathcal{P}(V)$.
This leads to a consecutive decrease of the original energy functional \eqref{eq:generalproblem}, which is approximated by a sequence of reduced problems given on the subsets of the current partition $\Pi^{k+1}$ of $V$.
In \cite{cutloic} the authors show that in case certain conditions are met the alternating iteration scheme in Algorithm \ref{alg:cut_pursuit} converges to a solution of the original variational problem in \eqref{eq:generalproblem}. 
The main advantage of this coarse-to-fine approach is that it leads to very efficient solvers for optimization problems on finite weighted graphs, which we will exploit in the following for the task of point cloud sparsification.
}

In this work we \fg{not only introduce a new class of regularizers \dt{for Cut Pursuit}, that even can be isotropic, we also} deviate from the original Cut Pursuit formulation and allow the choice of two different regularization functionals $R,Q$ and \dt{corresponding} parameters $\alpha, \beta$. This approach allows us to control the properties of the solutions for the task of point cloud sparsification and gives additional flexibility as we will show in Section \ref{s:numerics}.
\dt{
Indeed, one can only guarantee convergence to a minimizer of the original functional $J$ in \eqref{eq:problem} in the special case $R=Q$ and $\alpha=\beta$.
However, decoupling the regularization terms in the original Cut Pursuit scheme \ref{alg:cut_pursuit} has a major advantages for point cloud sparsification. It allows to control the compression rate of the resulting point cloud, regulated by the term $\alpha R(\cdot)$, independently of the enforced smoothness, regulated by the term $\beta Q(\cdot)$. Thus, one can choose to have a very smooth point cloud without giving up any points ($\alpha <\!\!\!< \beta$) or a strongly compressed point cloud without any smoothness constraints ($\alpha >\!\!\!> \beta$).
This additional flexibility allows for a wider range of applications using the same methodology.
}

Based on our argumentation above, we propose the following alternating minimization scheme as a modified variant of the original Cut Pursuit scheme.
\begin{alg}[Modified Cut Pursuit]
\label{alg:proposed_scheme}\small
\begin{gather}
\begin{split}
\Bigl\lbrace J'(f_\Pi;& \1_B) = \langle \nabla D(f_\Pi,g) , \1_B  \rangle + \alpha\langle \nabla R_S(f_\Pi), \1_B \rangle + \alpha R_{S^c}'(f_\Pi; \1_B) \Bigr\rbrace \rightarrow \min_{B \dt{\in \mathcal{P}(V)}}\label{eq:partitionproblem}
\end{split}\tag{P1}\\
 \: f_\Pi = \operatornamewithlimits{arg\min}_{f\in \mathcal{H}(\Pi)} D(f,\bar{f}_0)+ \beta Q(f).\label{eq:reducedproblem} \tag{R1}
\end{gather}\normalsize
\end{alg}
\fg{
As \dt{Proposition \ref{prop:derivationOfCP} and the discussion} in Appendix \ref{app:deriviationCutPursuit} \dt{shows} the partition} problem in \eqref{eq:partitionproblem} is well-defined. The optimization of \eqref{eq:partitionproblem} yields a binary partition induced by the subset $B$, which \dt{induces} a new partition $\Pi$.
This new partition $\Pi $ then defines a span of piecewise constant functions on which we solve the reduced problem \eqref{eq:reducedproblem}. Evidently this reduced problem can be solved more efficiently than the original problem \eqref{eq:problem}.

\fg{We want to emphasize that the chosen regularizer
\begin{align*}
R(f) = \frac{1}{2p}\big\Vert \nabla_w f \big\Vert_{{p;q}}^p
\end{align*} 
has different differentiability properties for different choices of $p$ and $q$ that we will investigate now.}
As becomes clear the regularization functional $R$ is differentiable iff \fg{$p>1, q\geq 1$} and the derivative is given as
\small 
\begin{align}\label{eq:derivative}
\begin{split}
\frac{\partial }{\partial f(u)_j}R(f) = \sum_{(u,v)\in E} w(u,v)^\frac{p}{2}  \left\Vert f(v)-f(u) \right\Vert_q^{p-q} {|f(v)_j - f(u)_j|}^{q-2} (f(u)_j-f(v)_j).
\end{split}
\end{align}\normalsize
For the interesting non-smooth case, i.e., $q \geq p = 1$, we can show that the directional derivative exists and the regularization functional $R$ can be split into differentiable and non-differentiable parts. Furthermore, we can show that for $p=q$ the expression in \eqref{eq:derivative} corresponds to the multidimensional anisotropic graph $p$-Laplacian, while for $q=2, p\geq1$ it corresponds to the multidimensional isotropic graph $p$-Laplacian as introduced in Section \ref{ss:graph_p_laplacian}. For details on our observations we refer the interested reader to Appendix \ref{app:cut}.

Clearly, Algorithm \ref{alg:proposed_scheme} is a descent method that decreases the energy functional in \eqref{eq:problem} in every iteration step. The proposed scheme is stopped once a minimizer is found and a further partitioning would not decrease the energy functional anymore. At this stage the desired level-of-detail is reached based on the chosen regularization parameters $\alpha$ and $\beta$. Note that this approach can be interpreted as a hierarchical graph method, e.g., as described in \cite{elmo_hierachical}.
\begin{rmk}
For $\alpha=\beta$ and $R = Q$ being an anisotropic regularization functional, i.e., $q=p$ in \eqref{eq:regularizer}, we are able to derive similar convergence results as described in \cite{cutloic}. In particular the alternating iterative scheme converges to the unique solution of the original problem \eqref{eq:problem}. For a given partition $\Pi = \lbrace A_1,\ldots,A_m \rbrace$ this problem has the solution 
\begin{equation*}
B \cap A_i = \emptyset \ \vee \ B \cap A_i= A_i \qquad \text{ for all } i=1,\ldots,m
\end{equation*}
iff a minimizer has been found. 
\end{rmk}
In the case of $\alpha R \neq \beta Q$ there are two potential issues concerning convergence: First of all, it may be possible that the partition $\Pi$ is not refined although the minimizer of \eqref{eq:reducedproblem} is not yet a minimizer of the original problem. Thus, we stop with a suboptimal solution. This is an issue that may appear in practice, however typically only at very fine levels such that the computed solution is already close to the optimum. Second, it might happen that $\Pi$ is refined although the solution of \eqref{eq:reducedproblem} is already globally optimal. In this case the solution after refinement is still the same, but the \dt{final refinement step slightly decreases the overall efficiency of the scheme}.

In Section \ref{ssec:graph_cuts} below we first discuss how to solve the \dt{minimum} partition problem in \eqref{eq:partitionproblem} and subsequently discuss the optimization of the reduced problem \eqref{eq:reducedproblem} using a primal-dual minimization method in Section \ref{ssec:red_problem}. 

\subsection{Computing the optimal partition via minimum graph cuts}\label{ssec:graph_cuts}
\fg{
In this section we investigate how to solve the partition problem \eqref{eq:partitionproblem} and how to build a new partition $\Pi$ from a computed $B\in \mathcal{P}(V)$. To compute the optimal $B$ we will use the well-known energy formulation of \cite{kolmogorov}, which then can be transferred to a flow graph. Computing the max flow of this graph results in \dt{minimizer} of the energy, and thus solves the partition problem. Afterwards, we show how the different flow graphs are defined for different \dt{values of the parameters} $p$ and $q$.}

\dt{
\subsubsection{Finding an optimal descent direction $\1_B$}
\label{sss:optimal_descent}
}
To determine an optimal descent direction $\1_B \in \mathcal{H}(V)$, i.e., the direction of steepest descent of $J$, one would need to minimize the directional derivative $J'(f_{\Pi};\vec{\gamma})$ with respect to all possible subsets $B \dt{\in \mathcal{P}(V)}$, which is known to be a NP-hard problem \fg{(cf. \cite{kolmogorov})}. \fg{Note again, that $\vec{\gamma} =  \sum_{A\in \Pi} 1_{A\cap B} (\gamma_{B}^A + \gamma_{B^c}^A)$ and thus \dt{is depending} on $1_B$.} On the other hand, if such an optimal subset $B \dt{\in \mathcal{P}(V)}$ is given, then a new partition $\Pi_{new}$ of $V$ can simply be generated by splitting each subset $A\subseteq \Pi$ of the previous partition $\Pi$ along the boundary of $B$ and $B^c$, such that $A$ is divided into (possibly) two smaller subsets $A_B = A\cap B$ and $A_{B^c} = A\cap B^c$. Note that this division given by $B$ can be performed on the whole vertex set $V$ but also on each subset $A\subseteq \Pi$ independently, as the partitioning is only getting finer while preserving the boundaries of previous partitions. This is an important feature for the implementation of parallelized optimization algorithms \dt{since every subset can be treated independently of the other subsets}. 

\begin{figure}
	\centering
	\subcaptionbox{Initial partition $\Pi = \{A_1,A_2,A_3,A_4\}$.}%
	[.47\textwidth]{\includegraphics[trim={0 0 0 0}, width=0.47\textwidth]{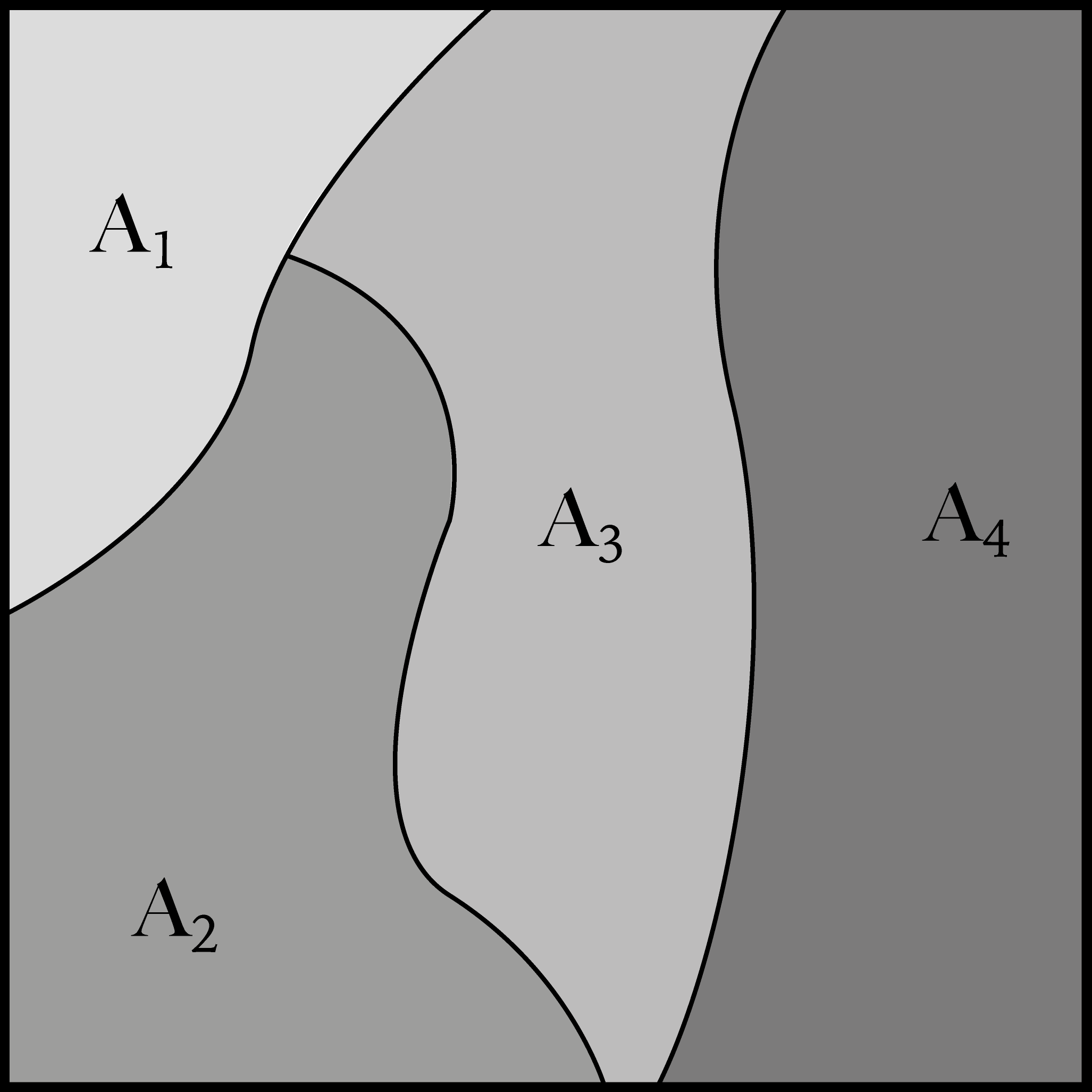} }\hfill
	\subcaptionbox{Steepest binary partition where $B \dt{\in \mathcal{P}(V)}$ is visualized by the white dashed set.}
	[.47\textwidth]{\includegraphics[trim={0 0 0  0}, width=0.47\textwidth]{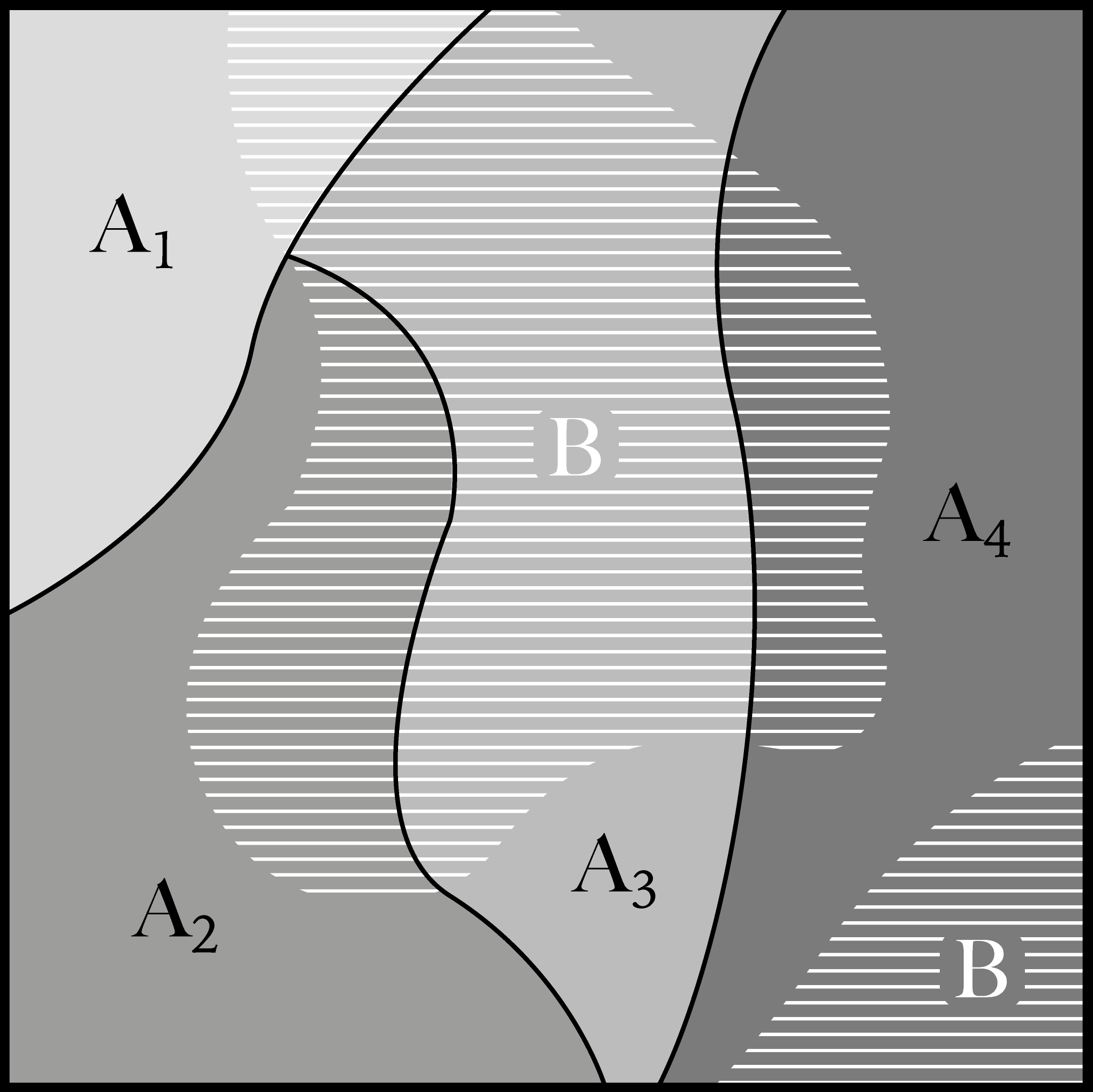} }\\[0.3cm]
	\subcaptionbox{Resulting partition\\ $\Pi_{new} = \{{A}_i \big| i\in [1,10]\}$ generated by the steepest binary cut and selecting connected components as the new partitions $A_i$.}%
	[.47\textwidth]{\includegraphics[trim={0 0 0  0}, width=0.47\textwidth]{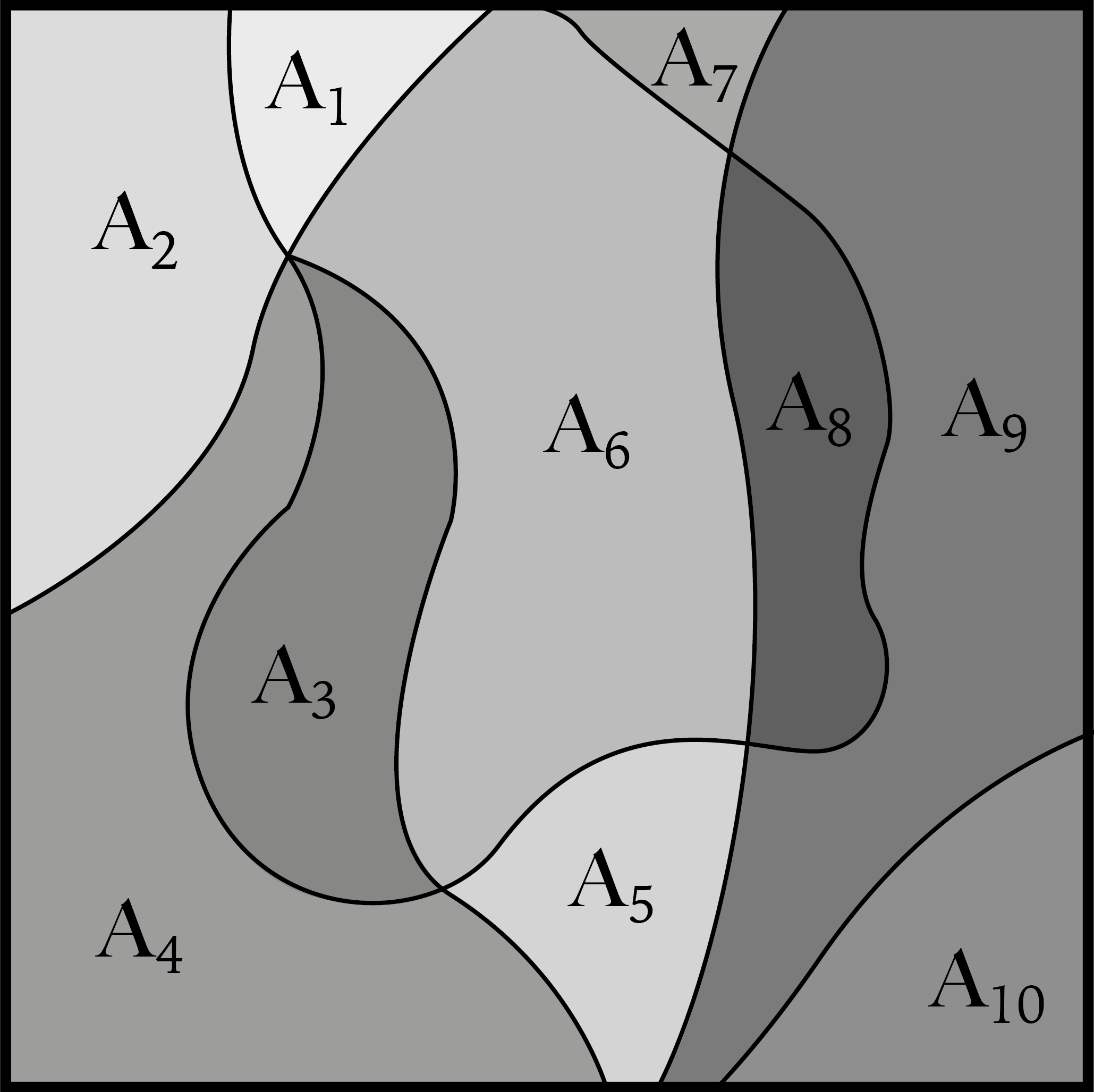} }\hfill
	\subcaptionbox{Resulting partition\\ $\Pi_{new} = \{{A}_i \big| i\in [1,8]\}$ generated by the steepest binary cut and selecting the new partition as $\Pi_{new}= (\Pi \cap B) \cup (\Pi \cap B^c)$.}%
	[.47\textwidth]{\includegraphics[trim={0 0 0  0}, width=0.47\textwidth]{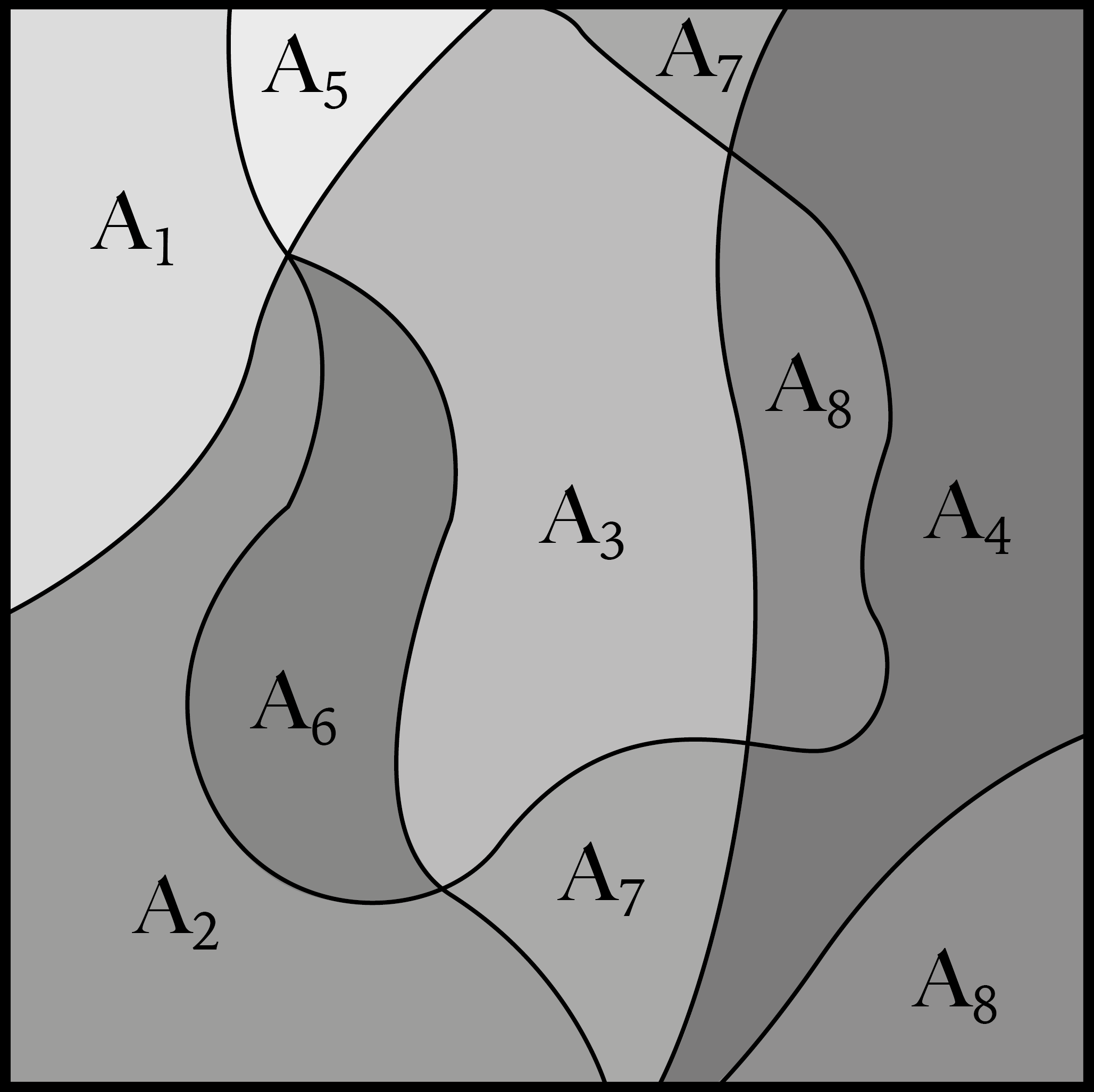} }
	\caption{\label{fig:partition} Illustration of two different methods to generate a new partition $\Pi_{new}$ from a given partition $\Pi$ and the set $B \dt{\in \mathcal{P}(V)}$  }
\end{figure}

There exist two possible options on how the subset $B \dt{\in \mathcal{P}(V)}$ can be used to generate a new partition $\Pi_{new}$ as illustrated in Figure \ref{fig:partition}. 
In the first variant the new partition can be written as $$\Pi_{new}= (\Pi \cap B) \cup (\Pi \cap B^c) = \bigl(\dot{\cup}_{i=1}^m A_i \cap B \bigr) \cup \bigl(\dot{\cup}_{i=1}^m A_i \cap B^c \bigr).$$ This means that one obtains a binary partition of each subset $A_i \subset \Pi$ leading to at most double the amount of subsets in $\Pi_{new}$ as compared to the previous partition $\Pi$.
\par The second variant treats every connected component \dt{$C$} of $\Pi \cap B$ and $\Pi \cap B^c$ as an own subset. 
\fg{Thus, the new partition can be written as 
\begin{align}\label{eq:NewPartitionConnComp}
    \Pi_{new} = \bigl(\dot{\cup}_{i=1}^m \dt{C}(A_i \cap B) \bigr) \cup \bigl(\dot{\cup}_{i=1}^m \dt{C}(A_i \cap B^c) \bigr).
\end{align}}
In this case a partition may lead to multiple new parts for each subset $A_i \subset \Pi$ as opposed to only two in the previous case. Hence, this strategy minimizes the energy at least as fast as the first strategy. In this paper we will focus only on the partition into connected components, since we aim for a fast sparsification of large point cloud data.

\begin{figure}[htb]
	\centering
	\subcaptionbox{Original graph $G$ with $8$ nodes and weights $w_{ij}$ connecting nodes $i$ and $j$. The initialization is the representation of the graph by one node $A_1$.}%
	[\textwidth]{\includegraphics[trim={0 0 0  0}, width=\textwidth]{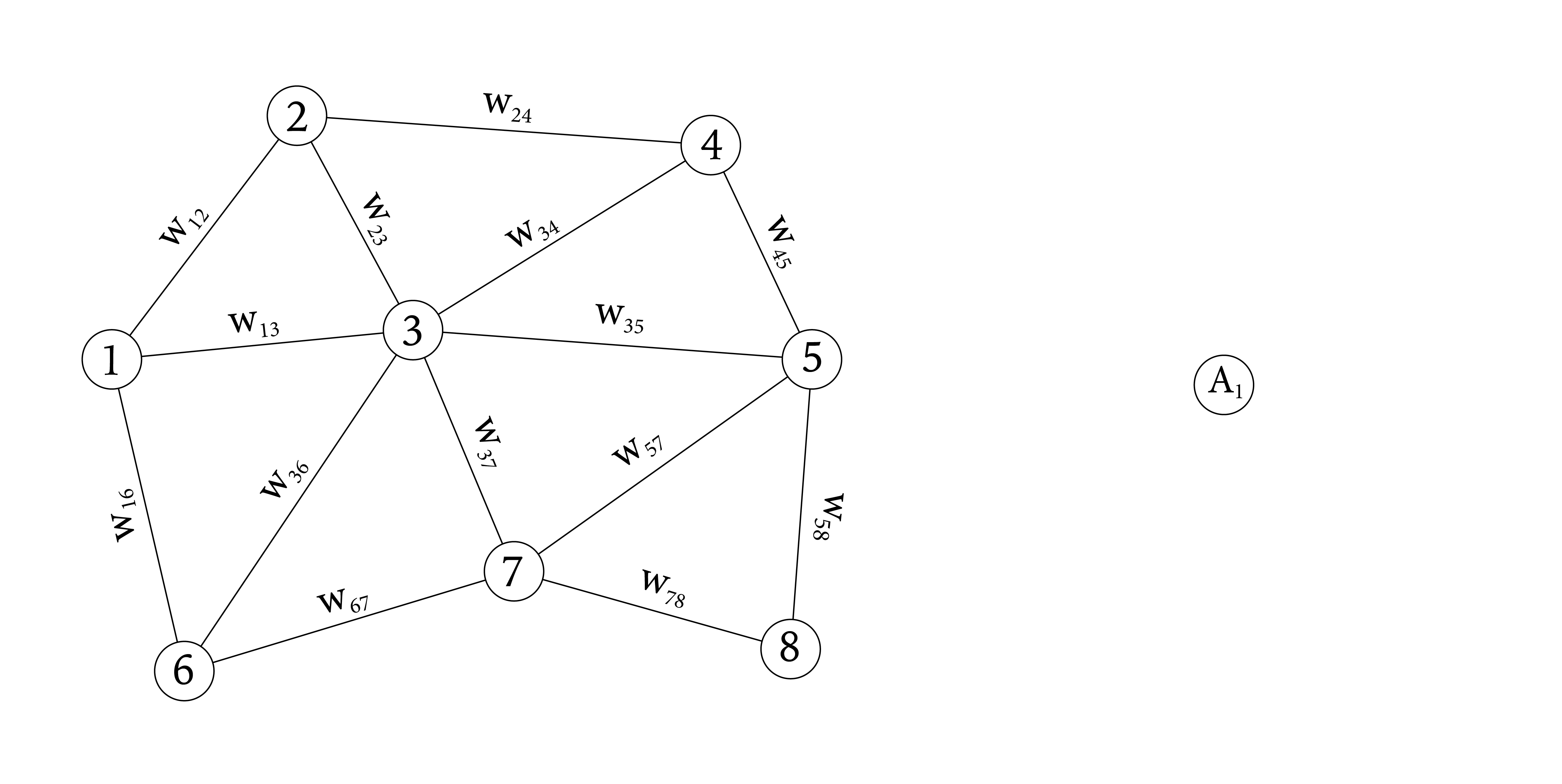}}
	\subcaptionbox{Cut (red line) dividing the graph $G$ into two subgraphs. This is represented as a graph with two nodes $A_1$ and $A_2$ connected by the edges that are cut between these two sets. The weights are the summed up weights of the connecting edges.}%
	[\textwidth]{\includegraphics[trim={0 0 0  0}, width=\textwidth]{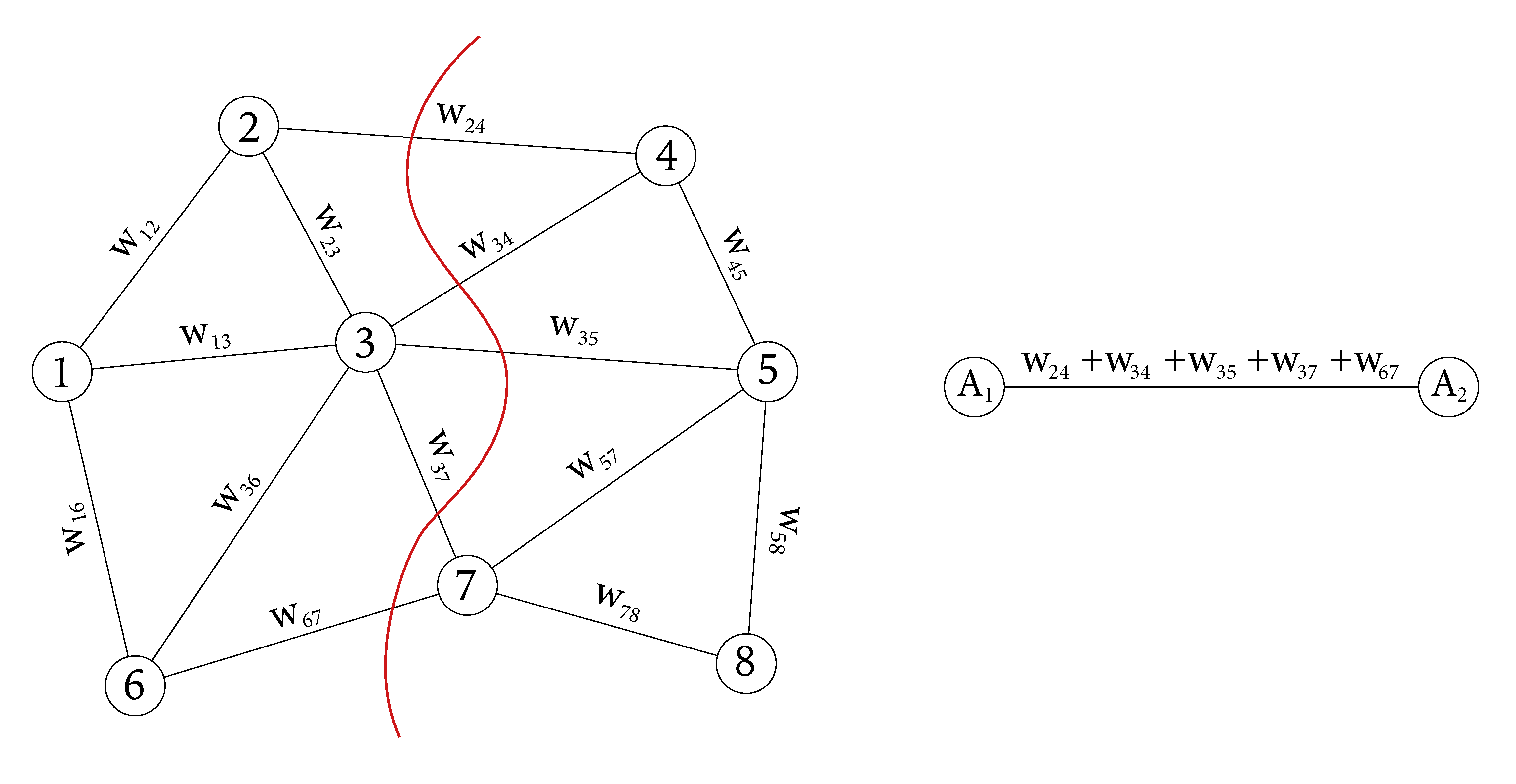}}
	\subcaptionbox{Another cut that cuts combined with the previous cut the graph $G$ into four subgraphs. The reduced graph is then represented by four nodes $A_1, A_2, A_3, A_4$ and the edges between the sets.}%
	[\textwidth]{\includegraphics[trim={0 0 0  0}, width=\textwidth]{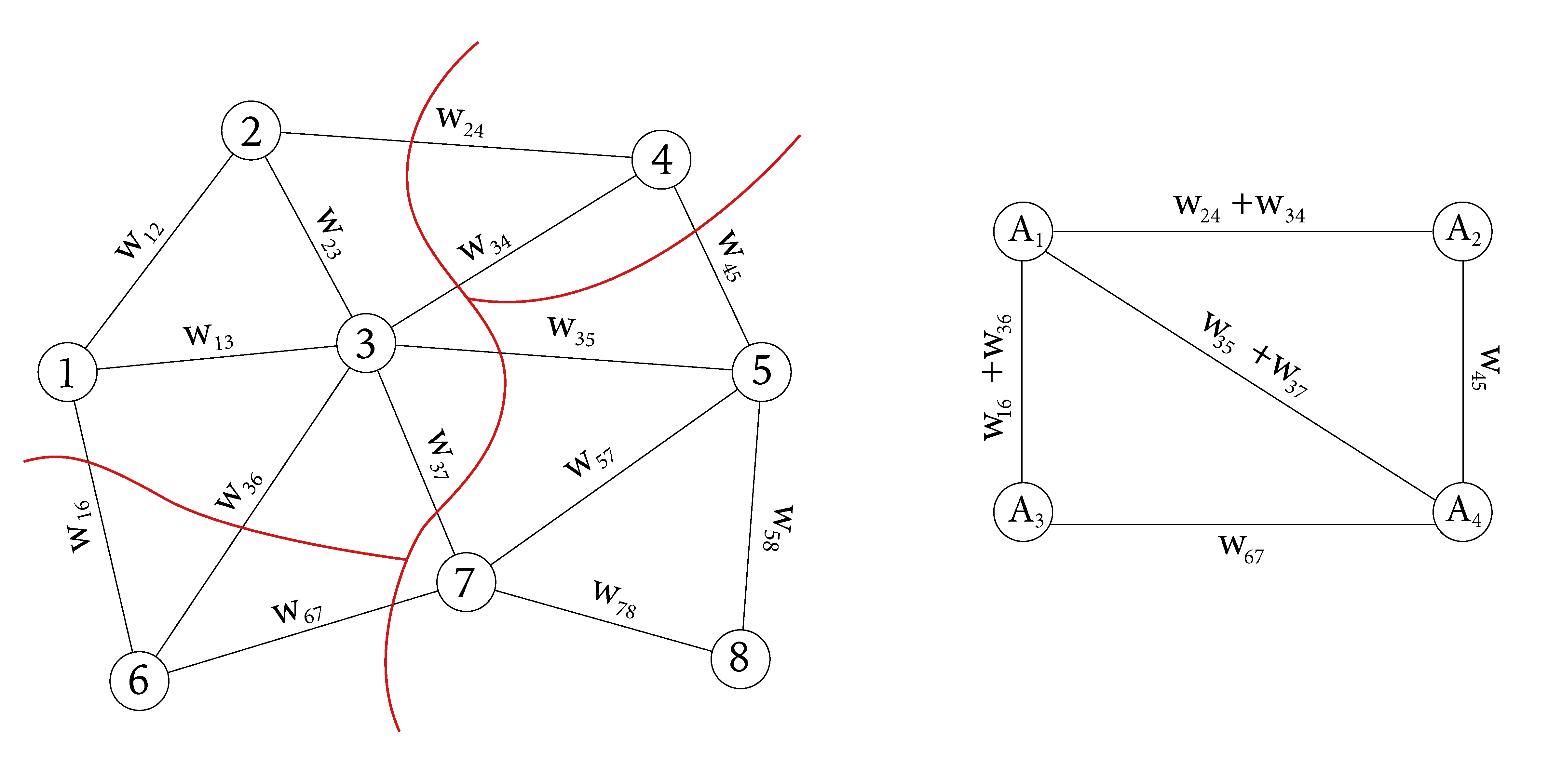}}
	\caption{\label{fig:reduced_graph}Illustration of an exemplary computation of a reduced graph by given cuts.}
	\label{fig:graph_cuts}
\end{figure}
In order to compute the optimal partition $\Pi$ based on some subset $B \dt{\in \mathcal{P}(V)}$ in each step of the alternating iteration scheme \eqref{eq:partitionproblem}, we recall the fact that if the minimization of the directional derivative $J'(f_\Pi; \vec{\gamma})$ is a binary partition problem and regular as described in \cite{kolmogorov}, minimizing the energy \eqref{eq:partitionproblem} is the same as computing a minimum cut of the corresponding flow graph for $J'(f_\Pi;\vec{\gamma})$. The regularity of $J'(f_\Pi;\vec{\gamma})$ is shown in Appendix \ref{app:submod}. In Figure \ref{fig:graph_cuts} we illustrate how a sequence of minimum graph cuts yields a sequence of reduced problems. Each reduced problem consists of a reduced set of vertices, where each vertex is a conglomerate of the original vertices within one subset $A_i$ of the partition $\Pi$ and the edges between these subsets are weighted by the sum of weights for cut edges in the original edge set.

\fg{As mentioned before we can build a flow graph corresponding to the energy given by $J'(f_\Pi; \vec{\gamma})$ and solve for $B$ with a graph cut by computing the max flow.}
The flow graph we consider in this work is defined as $G_{flow} = (V_{flow},E_{flow})$ with \fg{$V_{flow} = \lbrace 1, \ldots, dN\rbrace \cup \{s,t\}$, \dt{where} $N$ is the number of vertices} in the $d$-dimensional anisotropic case and \fg{$V_{flow} = \lbrace 1, \ldots, N\rbrace \cup \{s,t\}$} in the isotropic case. The anisotropic case is thus the $d$-fold vertex set of the original graph $G$ with two additional sink $t$ and source $s$ vertices. 
Note that this means in the anisotropic setting that each coordinate for every point of the point cloud data is modeled as an independent vertex in the flow graph.
The edge set of the flow graph is defined as $E_{flow} = \big\{ (u,v)\in V_{flow}\times V_{flow}\ \big| \ c(u,v) > 0 \big\}$, for which $c \in \mathcal{H}(E)$ is an edge function defining the edge capacities. 
These capacities are set in such a way, that the minimum cut of the flow-graph also minimizes the partition problem \eqref{eq:partitionproblem}. 
Note that one can compute the minimum graph cut on $G_{flow}$ by computing a solution of the equivalent maximum flow problem, for which efficient methods exist in the literature, e.g., cf. \cite{maxflowmincut}. 
For further details on this topic we refer to \cite{kolmogorov}.

In the following we describe how we set the capacities $c(u,v)$ for all edges $(u,v) \in E_{flow}$ of the flow graph. For the sake of simplicity, we will denote the set of differentiable directions as $S$ without an explicit case distinction of $S_1$ and $S_q$ as defined in Appendix \ref{app:C}.
Based on the directional derivatives for different values of $p$ and $q$ in Appendix \ref{app:C} we can tackle the partition problem \eqref{eq:partitionproblem} for $p,q\geq 1$. Let 
\begin{equation*}
\nabla J_S(f) \ = \ \nabla D(f,g) + \nabla R_S(f)\in \mathbb{R}^{Nd}
\end{equation*}
be the combined gradient of the differentiable parts of $J$.
Then the partition problem \eqref{eq:partitionproblem} can be rewritten as 
$$
\min_{B \dt{\in \mathcal{P}(V)}} \ \langle \nabla J_S(f) , \1_B  \rangle  + R_{S^c}'(f; \1_B).
$$

In the following we will divide the analysis of different choices for $p,q \geq 1$ into three different cases. \fg{First, we will discuss the the well-known anisotropic total variation regularizer for $p=q=1$, then the non-differentiable isotropic case for $q>p=1$ and finally the trivial differentiable case for $p,q>1$.}

Note that in our proposed approach for point cloud sparsification these parameter settings can be used to control the appearance of the resulting point clouds via the choice of a suitable regularizer in \eqref{eq:reducedproblem}. This is demonstrated in Section \ref{s:numerics}.

\paragraph{\bfseries{Case 1: $q=p=1$}}\mbox{}\par
\fg{
In this case the regularizer corresponds to the \emph{weighted anisotropic Total Variation} given as

\begin{align*}
R(f) &= \frac{1}{1}\Vert \nabla_w f\Vert_{1,1}^1 \ = \sum_{(u,v)\in E} \sqrt{w(u,v)}\sum_{j=1}^d  \vert f(v)_j - f(u)_j\vert.
\end{align*}
The directional derivative \dt{is} computed in Appendix \ref{app:pqlap} by the derivative of $R_S$ in \eqref{eq:derivativepq} and the directional derivative of $R_{S^c}$ in \eqref{eq:derive_directional_derivative}. 
Since, this regularizer decouples over the dimensions we do not have to choose a $d$-dimensional direction, but \dt{each dimension can be treated separately as a scalar vertex function}. %
\dt{We can} set $(\gamma_B^A)_j = \gamma_B > 0$ and  $(\gamma_{B^c}^A)_j = \gamma_{B^c} > 0$ for every $1\leq j \leq d$. If we select $p=q=1$ as a special case the regularizer \dt{corresponds to} the \emph{anisotropic total variation} regularizer
$$ R(f) = \Vert \nabla_w f\Vert_{1,1} =  \sum_{(u,v) \in E } \sqrt{w(u,v)} \sum_{j=1}^d \vert f(v)_j - f(u)_j\vert.  $$
For this the minimization problem \eqref{eq:prop} becomes 
\begin{align}
    \argmin_{B\in \mathcal{P}(V)} \ (\gamma_B + \gamma_{B^c})\langle \nabla D(f,g) + \alpha \nabla R_S(f), 1_B\rangle + (\gamma_B +\gamma_{B^c})\alpha w(B,B^c)
\end{align}
where $(\gamma_B + \gamma_{B^c})$ can actually be dropped. Thus, we can \dt{set} $\gamma_B+\gamma_{B^c} = 1$ and get
\begin{align}
    \argmin_{B\in \mathcal{P}(V)} \ \langle \nabla D(f,g) + \alpha \nabla R_S(f), 1_B\rangle + \alpha w(B,B^c)
\end{align}
which is \dt{the same problem that has been extensively discussed in the original Cut Pursuit analysis in} \cite{cutloic}.
}
\fg{Following \cite{cutloic} let us introduce the following two sets based on the directional derivatives
$$ 
\nabla^+ = \Big\{(u,j)\in V\times \lbrace1,\ldots,d\rbrace\  \big| \ \nabla J_S(f)_{(u,j)} \geq 0\Big\},
$$
$$
\nabla^- = \big( V \times \{1, \ldots, d\}\big)\setminus\nabla^+.
$$}
Note, that each tuple $(u,j)\in V\times \lbrace1,\ldots,d\rbrace$ can be described by a single vertex $u_j \in V_{flow}$.

We call this case the \textbf{non-differentiable, anisotropic case} for which the capacity function $c \in \mathcal{H}(E)$ is set as follows
\begin{align}\label{eq:flowgraph1} 
\begin{cases}
c(u_j,t) = \vert\nabla J_S(f)_{(u,j)}\vert, & (u,j)\in\nabla^-\\
c(s,u_j) = \ \ \nabla J_S(f)_{(u,j)}, & (u,j)\in \nabla^+\\
c(u_j,v_j) = \alpha \sqrt{w(u,v)}, & f(u)_j = f(v)_j, v\sim u,
\end{cases}
\tag{F1}
\end{align} 
and the corresponding flow graph can be constructed as described above.
Note that in this case a cut of this graph is the same as cutting $d$ independent flow graphs for which each one is related to a one-dimensional vertex function given by the coordinates of the original data. This comes from the fact that the capacities of \eqref{eq:flowgraph1} only connect vertices in the same respective dimension and there is no coupling between different dimensions.

\paragraph{\bfseries{Case 2: $q>p=1$}}\mbox{}\par
In the following we will discuss the most interesting setting, i.e., the non-differentiable, \emph{isotropic} case. We are mainly interested in solving a minimum graph cut problem with an isotropic TV regularization, \dt{which is much more challenging than the above discussed anisotropic case, since here the dimensions are coupled}. In this case \dt{the regularization functional is given as}
\begin{align*}
    R(f) = \sum_{(u,v)\in E} \sqrt{w(u,v)} \Vert \partial_v f(u)\Vert_q.
\end{align*}
For the sake of clarity we only discuss the special case of $p=1$ and $q=2$, which is in fact \dt{total variation variation with isotropy over the dimensions}. Note, that the argumentation in this paragraph holds also for the general case $q > p = 1$. 
\fg{
    The \dt{regularization functional in this case} is given as 
\begin{align*}
    R(f) & = \sum_{(u,v)\in E} \sqrt{w(u,v)} \Vert \partial_v f(u)\Vert_2\\ 
    &= \sum_{(u,v)\in E} \sqrt{w(u,v)} \sqrt{\sum_{j=1}^d (f(v)_j - f(u)_j)^2}.
\end{align*}
}
The directional derivative can be computed with \eqref{eq:derivativeq2} and \eqref{eq:dirregp}. \fg{Since we are in the isotropic case we have to choose a \dt{normalized} direction $\vec{\gamma}$, and thus $\gamma_B^A, \gamma_{B^c}^A\in \mathbb{R}^d$ for each set $A\in \Pi$ as described in Appendix \ref{app:deriviationCutPursuit}. It is crucial to \dt{perform} this for every \dt{subset independently}. For simplicity and also since we want to split every partition into two parts, we only have to \dt{determine} one direction $\gamma_A = \gamma_B = \gamma_{B^c}$. \dt{In Section \ref{sss:choosing_directions} we motivate a heuristical approach to choose a reasonable direction $\gamma_A \in \mathbb{R}^d$ for each subset $A \in \Pi$}. Plugging this into \eqref{eq:prop} \dt{we} get 
\begin{align}
    \argmin_{B\in\mathcal{P}(V)} \ \langle \nabla D(f,g) + \alpha \nabla R_S(f), \sum_{A\in\Pi} 1_{B\cap A}\gamma_A \rangle + \alpha w(B,B^c)
\end{align}
}

To compute the corresponding flow graph one has to set\fg{
\begin{align*}
\nabla^+ \ &= \ \Big\{u\in V \big| \ \langle \nabla J_S(f)_u, \gamma_A \rangle \geq 0, u \in A\Big\},\\
\nabla^- \ &= \ V\setminus\nabla^+.
\end{align*}}
We call this case the \textbf{non-differentiable, isotropic case} for which the capacity function $c \in \mathcal{H}(E)$ is set as follows
\begin{align}\label{eq:flowgraph3} 
\begin{cases}
c(u,t) \ = \  - \langle\nabla J_S(f)_u, \gamma_A \rangle, & u\in\nabla^- \wedge u\in A\\
c(s,u) \ = \ \phantom{-} \langle\nabla J_S(f)_u, \gamma_A \rangle, & u\in \nabla^+ \wedge u \in A\\
c(u,v) \ = \ \phantom{-} \alpha \sqrt{w(u,v)}, & f(u) = f(v), v\sim u.
\end{cases}
\tag{2}
\end{align} 

\paragraph{\bfseries{Case 3: \fg{$q, p > 1$}}}\mbox{}\par
{
In this \dt{easy case the regularization functional} becomes 
 \begin{align*}
    R(f) &= \frac{1}{p}\Vert \nabla_w f\Vert_{q,p}^p \ = \ \frac{1}{p}\sum_{(u,v)\in E} \left( \sum_{j=1}^d w(u,v)^\frac{q}{2} \vert f(v)_j - f(u)_j\vert^q\right)^\frac{p}{q}\\ 
    &= \frac{1}{p}\sum_{(u,v)\in E} w(u,v)^\frac{p}{2}\left( \sum_{j=1}^d  \vert f(v)_j - f(u)_j\vert^q\right)^\frac{p}{q} \\
    &= \frac{1}{p}\sum_{(u,v)\in E} w(u,v)^\frac{p}{2} \Vert \partial_v f(u)\Vert_q^p,
 \end{align*}
\fg{which is differentiable and an \emph{isotropic} regularizer for $d>1$ and $p\neq q$ since then the \dt{dimensions are} coupled. For $p=q$ this again becomes an \emph{anisotropic} regularizer.} Consequently, we can compute the directional derivative as  
\begin{align}
    J'(f;\vec{\gamma}) = \langle \nabla D(f,g) + \alpha\nabla R(f), \vec{\gamma}\rangle.
\end{align}
and again \dt{choose the normalized} direction $\gamma_A = \gamma_B^A = \gamma_{B^c}^A$ \dt{analogously to} Case 2. Plugging this into the minimization problem \eqref{eq:prop} we get 
\begin{align}
    \argmin_{B\in \mathcal{P}(V)} \ \langle \nabla J(f,g), \sum_{A\in \Pi} 1_{A\cap B} \gamma_A \rangle.
\end{align}
}

This is the trivial case where the functional $J$ is \emph{differentiable} everywhere, and thus $S^c = \emptyset$ and $R'(f;\1_B) = 0$. 
To compute the corresponding flow graph one has to set\fg{
\begin{align*}
\nabla^+ \ &= \ \Big\{u\in V \big| \ \langle \nabla J(f)_u, \gamma_A \rangle \geq 0, u \in A\Big\},\\
\nabla^- \ &= \ V\setminus\nabla^+.
\end{align*}
}

We call this case the \textbf{differentiable case} for which the capacity function $c \in \mathcal{H}(E)$ is set as follows
\begin{align}\label{eq:flowgraph2} 
	\begin{cases}
	c(u,t) = -\langle \nabla J(f)_u, \gamma_A \rangle, & u\in\nabla^-\\
	c(s,u) = \phantom{-}\langle \nabla J(f)_u, \gamma_A \rangle, & (u)\in \nabla^+\\
	c(u,v) = 0, & \forall u\in V
	\end{cases}\tag{F3}
\end{align} 
Note that the corresponding flow graph connects every vertex to either the sink $s$ or the source $t$, depending on the sign of the directional derivative, but there are no edges between the vertices themselves. Thus, the minimum cut is just a trivial cut $(S,T)$ with $S = \nabla^+,  T = \nabla^-$, i.e., a simple thresholding at zero. This allows to compute a minimum cut by just looking at the directional derivatives without constructing the flow graph $G_{flow}$ itself.

\dt{
\subsubsection{Choosing directions $\gamma_A$ for each subset $A \subset V$}
\label{sss:choosing_directions}
The only question that remains for discussion is how to choose a proper direction $\gamma_A$ for each subset $A \subset \Pi$. If we assume that the subset $A \subset \Pi$ can be well separated into two different parts, then intuitively it makes sense to determine a graph cut that removes edges between these two sets.
Ideally, this graph cut realizes a separation of the data points via a $(d-1)$-dimensional hyperplane $\Gamma$, i.e., a linear classifier in machine learning.
Assuming the hyperplane $\Gamma$ separates the two different parts of the subset $A$ well, then the normal vector of this hyperplane is a reasonable direction $\gamma_A \in \mathbb{R}^d$ for computing the capacities in \eqref{eq:flowgraph3}. This can be explained as follows: if one sets the regularization parameter $\alpha=0$ in \eqref{eq:flowgraph3} then the subset $A \subset V$ can be easily separated into two parts by determining the sign of the dot product of each data point with the normal vector $\gamma_A \in \mathbb{R}^d$ of the hyperplane $\Gamma$. Note that it is irrelevant if one uses $\gamma_A$ or $-\gamma_A$ as direction as it will only switch the sign of the dot product. In Figure \ref{fig:hyperplane} we illustrate this conceptual idea in the case of a two-dimensional point cloud.
}
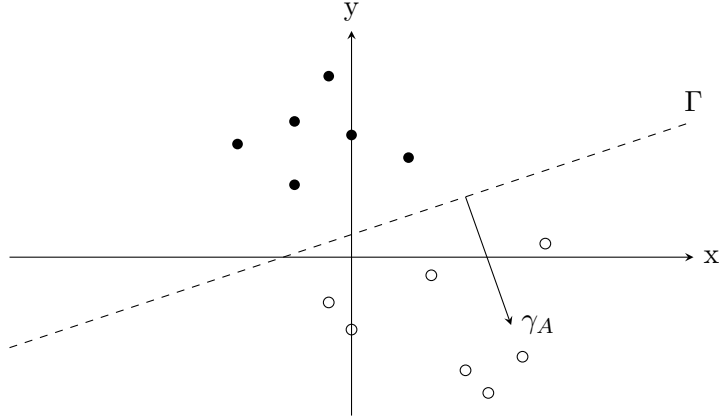
\begin{figure}[htb]
\centering
\begin{tikzpicture}[x=1.5cm,y=0.6cm]
  \draw[-stealth] (-3,0)--(3,0) node[right]{x}; %
  \draw[-stealth] (0,-3.5)--(0,5) node[above]{y}; %
  \draw[dashed] (-3,-2)--(3,3) node[above]{$\Gamma$}; %
  \draw[-stealth] (1,8/6)--(1.4,-9/6) node[right]{$\gamma_A$};

  \fill (0.5,2.2) circle[radius=2pt];
  \fill (-0.2,4)  circle[radius=2pt];
  \fill (0,2.7)  circle[radius=2pt];
  \fill (-1,2.5)  circle[radius=2pt];
  \fill (-0.5,3)  circle[radius=2pt];
  \fill (-0.5,1.6)  circle[radius=2pt];

  \draw (1.5,-2.2) circle[radius=2pt];
  \draw (-0.2,-1)  circle[radius=2pt];
  \draw (1.7,0.3)  circle[radius=2pt];
  \draw (1,-2.5)  circle[radius=2pt];
  \draw (1.2,-3)  circle[radius=2pt];
  \draw (0,-1.6)  circle[radius=2pt];
  \draw (0.7,-0.4)  circle[radius=2pt];
\end{tikzpicture}
\caption{Illustration of choosing a direction $\gamma_A$ as normal vector of a $(d-1)$-dimensional hyperplane $\Gamma$ that partitions the subset $A \subset V$ linearly for the case of a two-dimensional point cloud.}
\label{fig:hyperplane}
\end{figure}

\dt{
To compute a reasonably separating hyperplane $\Gamma$ one has two options. First, one can perform a principal component analysis (PCA) for the vertex function $f \in \mathcal{H}(V)$ restricted to the vertices in the subset $A \subset V$. The optimal hyperplane $\Gamma$ for separating the data is then given by the corresponding eigenvectors of the $d-1$ smallest eigenvalues of the covariance matrix. Consequently, the corresponding eigenvector of the single largest eigenvalue of the covariance matrix is the optimal direction $\gamma_A \in \mathbb{R}^d$. This makes sense as this eigenvector points in the direction of highest variance in the data and thus the orthogonal hyperplane $\Gamma$ spanned by the remaining eigenvectors separates the data according to this feature.
In the case of 3D point cloud sparsification that means one can compute a PCA for each subset $A \subset V$ and use the corresponding eigenvector of the largest eigenvalue of the covariance matrix as optimal direction $\gamma_A$.

Alternatively,  one can follow a standard approach from unsupervised machine learning, i.e., perform a $2$-means clustering on the subset $A \subset V$, which yields two good candidates $m_1, m_2 \in A \subset \mathbb{R}^d$ for cluster centers. Based on these one determines the optimal direction $\gamma_A \in \mathbb{R}^d$ as a normalized vector pointing from cluster center to the other, i.e.,
\begin{equation}
\label{eq:optimal_cut_direction}
\gamma_A = \frac{m_2-m_1}{\Vert m_2- m_1\Vert_2}.
\end{equation}

The heuristic approach presented above} allows us to reduce the multi-dimensional graph cut problem in the \dt{non-differentiable} isotropic case $q>p=1$ to a one-dimensional graph cut problem. During our numerical experiments we observed that \dt{this proposed method} leads to significantly better approximations than choosing random directions $\gamma_A$ for each subset $A \subset \Pi$.

To conclude our discussion we want to point out that following \cite{kolmogorov} the minimization of (\ref{eq:partitionproblem}) for some choices of $p$ and $q$ is the same as computing the minimum graph cut of the given flow graphs (\ref{eq:flowgraph1}) or (\ref{eq:flowgraph2}). Hence, one can solve the partition problem via standard maximum flow methods as described in \cite{maxflowmincut}. 

\subsection{Primal-dual optimization for the reduced problem}\label{ssec:red_problem}
For solving the reduced minimization problem \eqref{eq:reducedproblem} we will derive a primal-dual optimization algorithm as has been proposed by \cite{primaldual}. Let us consider a general minimization problem with proper, l.s.c., and convex functions $F$ and $G$, and a linear operator $K$ as follows
\begin{align}\label{eq:generalprimaldualproblem}
\min_{u\in X} G(u) + F(K u).
\end{align}
Following the argumentation in \cite{primaldual} one can derive the equivalent saddle-point formulation
\begin{align}
\min_{u\in X} \max_{y\in X^*} G(u) + \langle y, Ku \rangle -  F^*(y).
\end{align}
This can be solved by an iterative scheme that performs the following update
\begin{align*}
y^{k+1} &= \text{prox}_{\sigma F^*}\big(y^k + \sigma K \bar{u}^{k}\big)\\
u^{k+1} &= \text{prox}_{\tau G}\big(u^k - \tau K^* y^{k+1}\big)\\
\bar{u}^{k+1} &= u^{k+1} + \theta\big(u^{k+1}-u^k\big)
\end{align*}
with $\tau,\sigma >0$ and $\theta \in [0,1]$.

We are interested in solving the reduced minimization problem \eqref{eq:reducedproblem} in the general case with $q\geq p\geq 1$ in order to control the properties of the resulting sparse point clouds. In Section \ref{s:numerics} we will demonstrate the effect of different regularization functionals by various settings of $p$ and $q$. Note, that in the case $Q$ is differentiable, i.e., for $q\geq p>1$, there exist methods that are more suitable for the optimization of \eqref{eq:reducedproblem}, e.g., gradient descent methods as summarized in \cite[Section 4]{chambolle_introduction}. For the sake of simplicity, we will also cover this case in our general discussion below.

We are interested in deducing the necessary updates to compute a solution of the following variational problem
\begin{align}\label{eq:minproblem}
\min_{f\in \mathcal{H}(V)} \frac{1}{2}\Vert f - g\Vert_2^2 + \beta\Vert \nabla_w f\Vert_{{p;q}}.
\end{align}
Note that this is not exactly the same regularizer as given in \eqref{eq:problem}, except for the case $p=1$. However, since $Q$ is monotonic for all $p,q \geq 1$ a solution to \eqref{eq:minproblem} yields the same minimizer for appropriate rescaling of $\alpha$, cf. \cite{bungert2018} for details.  
To transfer the variational problem into the notation of \eqref{eq:generalprimaldualproblem} we set $K = \nabla_w$ and
\begin{align*}
F(\nabla_w f) \ = \ \beta \Vert \nabla_w f\Vert_{{p;q}} \ = \ \beta \left[\sum_{(u,v)\in E} \Big(\sum_{j=1}^d \vert \partial_v f(u)_j \vert^p\Big)^\frac{p}{q}\right]^\frac{1}{p}.
\end{align*}
Now we have to compute the convex conjugate $F^* = (\beta\Vert \cdot \Vert_{{p;q}})^*$. As shown in \cite{Sra_2012} the dual norm of the norm $\Vert \cdot \Vert_{{p;q}}$ is given by $\Vert \cdot \Vert_{{p^*;q^*}}$ with  $\frac{1}{p} + \frac{1}{p^*} = 1$ for $\frac{1}{q} + \frac{1}{q^*} = 1$, and $p,q\geq 1$. 
Hence, it follows that
\begin{align*}
\big( {\beta} \Vert \cdot \Vert_{p;q}\big)^*(y) \ = \ \delta_{B_{p^*;q^*}({\beta})} \ = \ \begin{cases}
0, & \Vert y\Vert_{{p^*;q^*}} \leq {\beta}\\
\infty, & \text{else}
\end{cases}
\end{align*}
with $y\in  \mathcal{H}(E)$.

We recall that the proximity operator of the characteristic function $\delta_C$ over a set $C\subset X$ is a projection, which is given as
\begin{equation}
\begin{split}
\prox_{\tau\delta_C}(z) \ &= \ \argmin_{x\in X} \left\{ \frac{1}{2\tau}\Vert x-z \Vert_2^2 + \delta_C (x) \right\} \\
\ &= \ \argmin_{x\in C} \left\{ \frac{1}{2\tau}\Vert x-z \Vert_2^2\right\} \ \eqqcolon \ \proj_C(z).
\end{split}
\end{equation}
Consequently, the proximity operator for $C = {B_{p^*;q^*}(\beta)}$ is the projection of an element $z\in \mathcal{H}(E)$ onto the ${p^*q^*}$-ball of radius $\beta$.
Thus, we get 
\begin{align}
\text{prox}_{\tau F^*}(z) 
\ &= \ \text{proj}_{B_{p^*;q^*}({\beta})}(z) 
\end{align}
See Appendix \ref{app:pqball} for a distinction of the projection for different choices of $p$ and $q$.

To conclude the derivation we have to compute the proximity operator for the update of the primal variable, which is given by
\begin{align*}
\text{prox}_{\tau F}(z) &= \arg\min_x\left\{ \frac{1}{2\tau} \Vert x-z \Vert_2^2 + F(x) \right\}\\
&=  \arg\min_x\left\{ \frac{1}{2\tau} \Vert x-z \Vert_2^2 + \frac{1}{2}\Vert x-g \Vert_2^2 \right\}.
\end{align*}
By computing the necessary optimality condition for $x$ it follows that
\begin{align*}
x \ &= \ \frac{z+\tau g}{1+\tau }.
\end{align*}
Plugging this into the proximity operator we get the following primal-dual algorithm for solving (\ref{eq:problem}) as a result

\begin{align}\label{alg:finalprimaldual}
\begin{split}
f^{k+1} &= \frac{f^k - \tau \nabla_w^* y^k+\tau g}{1+\tau}\\
y^{k+1} &= \proj_{B_{q^*;p^*}(\beta)}\big(y^k + \sigma \nabla_w f^{k+1}\big).
\end{split}\tag{PD}
\end{align}

Above we have derived an iterative algorithm to solve the minimization problem in (\ref{eq:minproblem}). We want to use this method to solve the reduced problem in (\ref{eq:reducedproblem}) with $R = \Vert \nabla_w\cdot\Vert_{{p;q}}$. 
Therefore, let $\Pi = \big\{A_1, \ldots, A_m \big\}\subset V$ be a fixed partition of the vertex set $V$ and $\mathcal{H}(\Pi)$ be the Hilbert space induced by this partition as defined above. 

\fg{The corresponding reduced graph $G_r = (V_r, E_r, w_r)$ is defined as in $V_r = \Pi$, \eqref{eq:reducedEdgeset} and \eqref{eq:reducedWeights}}.
\fg{Recall that any piecewise constant function $f_c\in \Hil(V)$ can be represented by a function $c = \left(c_A\right)_{A\in \Pi} \in \mathcal{H}(Vr)$ as $$f_c = \Big(\sum_{A\in\Pi} {c_A}_j 1_A\Big)_{j=1}^d$$ on the reduced graph. }

\fg{To simplify the notation in the computations later on} we introduce a \dt{matrix} operator $P := \big(\1_{A_1} \ \ldots \ \1_{A_m}\big) \subset \{0,1\}^{N\times m}$ with the following properties that are easy to show.
\begin{lem}\textbf{(Properties of the expansion operator $P$)}\\
\label{lemma:expandoperator}
	Let $\Pi = \lbrace A_i \mid i=1,\ldots,m \rbrace$ be a partition of $V$ as defined above and let the operator $P := \big(\1_{A_1} \ldots \1_{A_m}\big) \in \{0,1\}^{N\times m}$. Then the following properties hold:
	\begin{align}
	&Pc \in \mathcal{H}(V), \tag{i}\\[2mm]
	&P^*P = \mathrm{diag}\big(\vert A_1\vert, \ldots, \vert A_m \big\vert), \tag{ii}\\[2mm]
	&P^*\nu = \big(\nu_A\big)_{A\in \Pi} \text{ for any } \nu\in \mathbb{R}^{N\times d} \text{ with }\nu_A = \sum_{u\in A} \nu(u)\in \mathbb{R}^d. \tag{iii}
	\end{align}
\end{lem}
We call $Pc$ an \textit{expansion} of $c$ from $\mathcal{H}(\Pi)$ to a piecewise constant function in $\mathcal{H}(V)$ and $P^*f$ a \textit{reduction} of $f\in \mathcal{H}(V)$ to the reduced space $\mathcal{H}(\Pi)$. 
Based on the expansion operator we can construct a piecewise function $f_c\in\mathcal{H}(V)$ such that
 \begin{align}\label{eq:liftedf}
 \fg{f_c = Pc}.
 \end{align}
  For functions of the form \eqref{eq:liftedf} we introduce the subspace $\mathcal{S}_\Pi \subset \mathcal{H}(V)$ of piece-constant functions induced by the partition $\Pi$ as
\begin{equation}\label{eq:span}
 \mathcal{S}_\Pi := \ \Big\{f_c\in \mathcal{H}(V) \ \Big|  \  f_c = Pc, \ P = \big(\1_{A_1} \ \ldots \ \1_{A_m}\big), \ c\in \mathcal{H}(\Pi) \Big\}.
\end{equation}

We aim to solve a reduced problem over the piecewise constant functions $f\in \mathcal{S}_\Pi$ given by 
\begin{align}\label{eq:notreduced}
	\operatornamewithlimits{arg\min}_{f\in \mathcal{S}_\Pi} \frac{1}{2} \Vert f - g \Vert_2^2 + \beta \Vert \nabla_w f \Vert_1.
\end{align}
With the properties of the operator $P := \big(\1_{A_1} \ldots \1_{A_m}\big) \in \{0,1\}^{N\times m}$ given in Lemma \ref{lemma:expandoperator} we rewrite the data term of \eqref{eq:notreduced} as
\begin{align*}
\Vert f-g\Vert_2^2 &= \Vert Pc - g\Vert_2^2.
\end{align*}
The following proposition yields that for $f = Pc$ we can deduce that $\Vert \nabla_w f\Vert_1 = \Vert \nabla_{w_r} c\Vert_1$.
\begin{prop}
	Let $G = (V,E,w)$ be a finite weighted graph and $G_r = (V_r,E_r,w_r)$ a reduced graph corresponding to the partition $\Pi$ of $V$. 
	Let $f\in \mathcal{S}_\Pi$ with $f = \sum_{A\in \Pi} c_A 1_A$ and $c = (c_A)_{A\in\Pi} \in \mathcal{H}(\Pi)$. Then the following equality holds 
	$$\Vert \nabla_w f\Vert_1 = \Vert \nabla_{w_r} c\Vert_1. $$
\end{prop}
\begin{proof}
	Let $E_{AB} = (A\times B) \cap E$ be the set of edges between the partitions $A$ and $B$ and note that $E = \bigcup_{(A,B)\in E_r} E_{AB}$. Then we can deduce
	\begin{align*}
		\Vert \nabla_{w_r} c \Vert_1 & = \sum_{(A,B)\in E_r} w_r(A,B)\vert c_B - c_A\vert\\
		& = \sum_{(A,B)\in E_r} \sum_{ (u,v) \in E_{AB}} \sqrt{w(u,v)} \  \vert c_B - c_A\vert\\
		& = \sum_{(A,B)\in E_r} \sum_{ (u,v) \in E_{AB}} \sqrt{w(u,v)} \  \vert f(v) - f(u)\vert\\		
		& = \sum_{(u,v)\in E} \sqrt{w(u,v)} \  \vert f(v) - f(u)\vert\\
		& = \Vert \nabla_w f\Vert_1
	\end{align*}
	\qed
\end{proof}

Now we can rewrite problem \eqref{eq:notreduced} to a reduced form that only depends on $c\in \mathcal{H}(\Pi)$ and write it as the reduced problem
\begin{align}
\label{eq:red_minproblem}
	f_\Pi \ = \ \operatornamewithlimits{arg\min}_{c\in \mathcal{H}(\Pi)}\ \frac{1}{2} \Vert Pc - g \Vert_2^2 + \alpha \Vert \nabla_{w_r} c \Vert_1.
\end{align}

The only difference now between the original problem \eqref{eq:minproblem} and the reduced formulation \eqref{eq:red_minproblem} is the operator $P$. Since this $P$ has only influence on the primal variable update, we have to compute a different primal variable update with the same strategy as before by computing the proximal operator
\begin{align*}
\mathrm{prox}_{\tau F}(z) &= \arg\min_{c\in \mathcal{H}(\Pi)}\left\{ \frac{1}{2\tau} \Vert c-z \Vert_2^2 + F(c) \right\}\\
&=  \arg\min_{c\in \mathcal{H}(\Pi)}\left\{ \frac{1}{2\tau} \Vert c-z \Vert_2^2 + \frac{1}{2}\Vert Pc-g \Vert_2^2 \right\}.
\end{align*}
By computing the necessary optimality condition for $c$ it follows that
\begin{align*}
c = \left(I+\tau P^*P\right)^{-1} \left(z+\tau P^*g\right).
\end{align*}
With this we deduce the following primal variable update 
\begin{align}
	c^{k+1} = \Big(I+\tau P^*P \Big)^{-1}\Big(c^k - \nabla_{w_r}^*y^k + \tau P^*g \Big).
\end{align}

Using Lemma \ref{lemma:expandoperator} we can write $P^*g = (f_{0A})_{A\in \Pi}$ and $P^*P = \text{diag}\left(|A_1|, \ldots, |A_m| \right)$ and it follows that $$\left(I+\tau P^*P\right)^{-1} = \text{diag}\left( \frac{1}{1+\tau|A_1|}, \ldots, \frac{1}{1+\tau|A_m|}\right)$$ which implies the following update for each partition $A\in \Pi$
\begin{align}\label{eq::primalupdatereduced}
	c_A^{k+1} = \frac{1}{1+\tau|A|}\left( c_A^k + \big(\nabla_{w_r}^*y^k\big)_A + \tau f_{0A}  \right).
\end{align}

Interestingly, the matrix $\tau P^*P$ can be interpreted as a variant of diagonal preconditioning. However, the acceleration methods as described in \cite{primaldual} and a diagonal preconditioning as in \cite{diagonalPPD} can still be applied additionally. 

\subsection{Diagonal preconditioning}\label{ss:diagprecon}
In this section we want to \dt{investigate preconditioning} of the reduced problem and the corresponding operator $\nabla_{w_r}$. As we pointed out before any vertex function \fg{$f\in \Hil(V)$ can be described by a vector $$f = \left(f(u)\right)_{u\in V} \in \mathcal{R}^{N\times d}$$ with $N$ as the number of vertices. The weighted gradient can also be described by a vector given as $$\nabla_w f = \left(\sqrt{w(u,v)}(f(v)-f(u))\right)_{(u,v)\in E} \in \mathcal{R}^{M\times d}$$ with $M$ as the number of edges. Thus, we can give a differential operator matrix $\mathcal{D}\in \mathbb{R}^{M\times N}$ representing the graph operator $\nabla_w$ i.e. $\mathcal{D} f = \nabla_w f$. As $E$ is finite we can find a corresponding $e_i = (u_i, v_i)\in E$ for every $i\in \{1, \ldots, M\}$ and we can define $\mathcal{D}$ as follows
\begin{align}
	\mathcal{D}_{i,\tilde{u}} = \begin{cases}
	\phantom{-}\sqrt{w(u_i,v_i)} \  , & \tilde{u} = u_i\\
	-\sqrt{w(v_i,u_i)}, & \tilde{u} = v_i\\
	\phantom{-}0, & \text{else.}
	\end{cases}
\end{align} }
As we can see, the number of entries in a column for a given vertex $u\in V$ depends on the number of neighbors. Thus, for graphs with a rather inhomogeneous  structure, e.g. a symmetrized $k$-nearest neighbors on unstructured point clouds, the norm of the operator might not be a good choice for the step sizes $\tau$ and $\sigma$ in \eqref{alg:finalprimaldual} as it might be too conservative for most vertices. Applying preconditioning often is a good measure to enhance the convergence speed of the algorithm.
In order to apply the preconditioning scheme in \cite[Lemma 2]{diagonalPPD} we have to compute the row and column sums of the absolute values in $\mathcal{D}$. Assuming that $w(u,v) = w(v,u)$ for all $(u,v)\in E$ the component-wise preconditioners for $\mathcal{D}$ are then given by
\begin{align}
	\tau_u &= \frac{1}{\sum_{i=1}^M \vert\mathcal{D}_{i,u}\vert^{2-\alpha}} = \frac{1}{\sum_{v\sim u} w(u,v)^\frac{2-\alpha}{2}} , \quad \forall u\in V\\
	\sigma_i &= \frac{1}{\sum_{u\in V}\vert\mathcal{D}_{i,u}\vert^\alpha} = \frac{1}{2 w(u,v)^\frac{\alpha}{2}}, \quad \forall i\in\{1,\ldots, M\}
\end{align}
for any $\alpha \in [0,2]$.
This leads to the diagonal preconditioners 
\begin{align}
	T &= \mathrm{diag}\left( \tau_1, \ldots, \tau_N \right)\\	
	\Sigma &= \mathrm{diag}\left( \sigma_1, \ldots, \sigma_M \right).
\end{align}
As we can see, the preconditioning for the primal update $T$ takes the number of edges and their weights directly into account, and thus well improves the condition of this problem. 

In the reduced problem we get an even worse condition, since the size of the partitions, the summed up weights of the combined edges and the number of neighboring partitions might differ heavily. As we have seen in \eqref{eq::primalupdatereduced} the size of the partitions is already handled in the primal update. We propose to apply a diagonal preconditioning to the reduced primal-dual approach but now on the reduced differential operator matrix $\mathcal{D}_r$ which is defined for $G_r$ as $\mathcal{D}$ is defined on $G$. 
This can be applied as described before and for the reduced primal update we thus get a diagonal preconditioning as 
$$ \tau_A = \frac{\vert A\vert}{\sum_{(A,B)\in E_r} w_r(A,B)^{2-\alpha}}.$$
With the preconditioning schemes proposed above we are able to alleviate the problem of a bad condition of $\mathcal{D}_r$ and achieve a significant convergence acceleration as we will demonstrate in Section \ref{s:numerics}.

\subsection{Weighted $l_0$ regularization}
\label{ss:l0}
Finally, we want to give a special highlight on a regularization functional that is related but yet not \fg{covered} by the formulation \eqref{eq:problem}. In particular we want to discuss a Cut Pursuit strategy for energy functionals of the form 
\begin{align}
\label{eq:weighted_l0}
 J_0(f;g) \ = \ D(f,g) + \alpha \!\!\!\! \sum_{(u,v)\in E} \sqrt{w(u,v)} \  \1_{S_0}
\end{align}
with $S_0 = S_0(f) = \{(u,v)\in E \ | \ f(u) \neq f(v)\}$. The {proposed regularization term in \eqref{eq:weighted_l0}} can be interpreted as weighted $\ell_0$ regularization \dt{for which we analyse its properties in the following.}

Let $\Pi$ be some partition of the vertex set $V$ and let $f_\Pi = Pc\in \mathcal{S}_\Pi$. Also let $G_r = (V_r, E_r, w_r)$ be the reduced graph corresponding to $\Pi$ as defined before. Notice that the functional $J_0$ in \eqref{eq:weighted_l0} is differentiable for every $(u,v) \in S_0(f_\Pi)$. 
\fg{The formulation of the partition problem in this case is given by 
\begin{align}
    \argmin_{B\in \mathcal{P}(V)} \ \langle \nabla D(f_\Pi,g), \vec{\gamma}\rangle + w(B,B^c)
\end{align}
as we have derived in Appendix \ref{app:derivationMinimalPartition} and written in \eqref{eq:appendixMinPartition}.}
\par To deduce the reduced problem we first emphasize that 
\begin{equation*}
\sum_{(u,v)\in E} \sqrt{w(u,v)} \  \1_{S_0(f_\Pi)} \ = \sum_{(A,B)\in E_r} w_r(A,B)
\end{equation*}
which is not depending on $f_\Pi$. Thus, it is a constant and can be dropped for minimization which leads to
\begin{align}
	\argmin_{c\in\mathcal{H}(\Pi)} \ \frac{1}{2} \Vert Pc - g\Vert_2^2 \ = \ \argmin_{c\in\mathcal{H}(\Pi)} \ \frac{1}{2} \Vert Pc - g\Vert_2^2 + \alpha \!\!\!\!\!\! \sum_{(A,B)\in E_r} w_r(A,B).
\end{align}
We can formulate the necessary optimality condition as 
\begin{equation*}
P^*Pc - P^*g \ = \ 0,
\end{equation*} 
which leads with Lemma \ref{lemma:expandoperator} to the component-wise solution 
\begin{equation}\label{eq:solutionL0}
c_A = \frac{\fg{\sum_{u\in A} g(u)}}{\vert A\vert}, \ \forall A\in \Pi,
\end{equation}
i.e., the optimal piecewise constant approximations are the mean values of the respective subsets $A$ induced by the partition $\Pi$.

In conclusion we get the following Cut Pursuit algorithm for the case of a weighted $\ell_0$ regularization functional.
\begin{alg}[Cut Pursuit with $\ell_0$ regularization] 
\label{alg:proposed_schemel0}\small
\begin{gather*}
\Bigl\lbrace J'(f_\Pi; \vec{\gamma}) = \langle \nabla D(f_\Pi,g) , \vec{\gamma} \rangle + \alpha\!\!\! \sum_{(u,v)\in S_0^c} \sqrt{w(u,v)} \  \vert \vec{\gamma}(u) - \vec{\gamma}(v)\vert \Bigr\rbrace \rightarrow \min_{B \in \mathcal{P}(V)},\label{eq:partitionprobleml0}\\
\fg{c_{Aj} = \frac{\sum_{u\in A} f_{0j}(u)}{|A|}, \ \forall A\in \Pi, \ 1\leq j\leq d}.
\end{gather*}\normalsize
\end{alg}
\fg{This algorithm is different from the one given in \cite{cutloic} and much closer to the original Cut Pursuit approach. It also covers a different class of data terms, since in \cite{cutloic} the data term is supposed to be separable, but can be non-differentiable. Here, it has not to be separable but differentiable. This allows to use the \dt{Cut Pursuit scheme with $\ell_0$ regularization in a wider range} of applications.}

\fg{\begin{algorithm}[b]
	\DontPrintSemicolon
	\KwData{A $d$-dimensional point cloud $g\colon V\rightarrow \mathbb{R}^d$}
	\textbf{Method:}\\
		$G = (V,E,w) \leftarrow $ constructGraph($g$)\\
		$\Pi \leftarrow \lbrace V\rbrace$\\
		\While{$J'(f_\Pi;\1_B)<0$}{
		$G_{flow} \leftarrow$ buildFlowGraph($V$,$g$,$\alpha$) for given methods \eqref{eq:flowgraph2}, \eqref{eq:flowgraph1}, \eqref{eq:flowgraph3}\\
		$B \leftarrow $ computeMaxFlow($G_{flow}$) via maxflow algorithm (cf. \cite{maxflowmincut})\\
		$\Pi \leftarrow$ computePartition($V$,$B$,$\Pi$) as $\Pi_{new}$ in \eqref{eq:NewPartitionConnComp}\\
		$G_{r} = (V_r, E_r, w_r)\leftarrow$ computeReducedGraph($\Pi$) with $V_r = \Pi$, $E_r$ as in \eqref{eq:reducedEdgeset} and $w_r$ as in \eqref{eq:reducedWeights}\\
		$f_{\Pi} \leftarrow$ solveReducedProblem($g, \Pi, G_{r}$) with Primal Dual algorithm as in \cite{primaldual} with primal update \eqref{eq::primalupdatereduced} or for $\ell_0$ with update \eqref{eq:solutionL0}
		}				
		\KwResult{A sparse point cloud $f_{\Pi} \colon \Pi \rightarrow \mathbb{R}^d$}
	\caption{Cut Pursuit for 3D point cloud sparsification \label{alg:cp}}
\end{algorithm}}

\section{Numerical experiments}\label{s:numerics}
In this section we evaluate the performance and effectiveness of the proposed minimization schemes in Algorithms \ref{alg:proposed_scheme} and \ref{alg:proposed_schemel0} for the task of point cloud sparsification. The algorithms presented in Section \ref{s:cut_pursuit} were implemented in the programming language \textsc{MathWorks Matlab (R2018a)} without any additional external libraries. We did not use any built-in parallelization paradigms of Matlab except for vectorization. Thus, one can expect that the absolute time needed for computing a sparse point cloud can still be optimized by using techniques such as parallelization on modern general purpose GPUs or distributed computing. This is feasible in our situation since every subset $A_i \subset V$ of a partition $\Pi$ can be treated independently from the other subsets in the subsequent iterations of the proposed minimization scheme. 

The minimum graph cut was computed by the built-in Matlab function \texttt{maxflow} to which we pass the constructed flow graph as described in \eqref{eq:flowgraph1}. The primal-dual minimization algorithm was implemented in an over-relaxed version with step size updates as described in \cite{primaldual}. Since there is no universal stability condition for primal-dual optimization on finite weighted graphs, we estimate the spectral norm of the weighted gradient operator via a power iteration scheme \cite{powertits}. 

We performed our experiments directly on the raw point clouds without any preprocessing or triangulation of the surface. To build a finite weighted graph on the point cloud we connect each point to its $k$-nearest neighbors ($k=8$) \dt{in terms of the Euclidean distance} and symmetrize the edges to have an \emph{undirected} graph structure. As presented in Section \ref{s:graphs} we only consider undirected edges due to a simplification of the involved graph operators. However, we underline that the proposed minimization scheme is independent of the graph structure and thus can also be used for directed graphs. 
\dt{We define a vertex function $f \colon V \rightarrow \mathbb{R}^3$ as the three-dimensional coordinates of the given point cloud.} We set the weight function to be the inverse squared Euclidean distance between connected points $f(u)$ and $f(v)$ as proposed in \cite{elmoataz2008}, i.e.,
\begin{equation*}
w(u,v) \ = \ \frac{1}{\Vert f(u) - f(v) \Vert_2^2}.
\end{equation*}
\dt{This is meaningful since edges to neighboring 3D points which have a smaller Euclidean distance get a higher weight and thus have higher influence on a graph vertex.}

The overall structure of the implemented method for point cloud sparsification is summarized in Algorithm \ref{alg:cp} below. 
\begin{algorithm}[b]
	\DontPrintSemicolon
	\KwData{A $d$-dimensional point cloud $g\colon V\rightarrow \mathbb{R}^d$}
	\textbf{Method:}\\
		$G = (V,E,w) \leftarrow $ constructGraph($g$)\\
		$\Pi \leftarrow \lbrace V\rbrace$\\
		\While{$J'(f_\Pi;\1_B)<0$}{
		$G_{flow} \leftarrow$ buildFlowGraph($V$,$g$,$\alpha$)\\
		$B \leftarrow $ computeMaxFlow($G_{flow}$).\\
		$\Pi \leftarrow$ computePartition($V$,$B$,$\Pi$) \\
		$G_{r} \leftarrow$ computeReducedGraph($\Pi$)\\
		$f_{\Pi} \leftarrow$ solveReducedProblem($g, \Pi, G_{r}$)
		}				
		\KwResult{A sparse point cloud $f_{\Pi} \colon \Pi \rightarrow \mathbb{R}^d$}
	\caption{Cut Pursuit for 3D point cloud sparsification \label{alg:cp}}
\end{algorithm}
We have put an implementation of the proposed method as open source on Github. The interested reader can download the source code via the URL \url{To Be Inserted After Review}.

\subsection{Special case: Octree approximation}
\label{ss:results_octree}
In the following we discuss a special case of our proposed method that is currently used as a standard technique for 3D point cloud sparsification. If we set the regularization parameters $\alpha = \beta = 0$ in \eqref{eq:partitionproblem} and \eqref{eq:reducedproblem}, respectively, then we observe that Algorithm \ref{alg:cp} performs exactly an octree approximation of the original data.
The reason for this is the fact that the flow graph described in Section \ref{ssec:graph_cuts} has zero capacities between neighboring vertices since the regularization parameter is set to zero. Hence, the anisotropic graph cut assigns each coordinate according to its relative position to the current cluster center its vertex is associated to. As shown in \eqref{eq:flowgraph2} the octree approximation is performed by a simple thresholding operation based on the sign of the $L^2$ data fidelity term. Each iteration of Algorithm \ref{alg:cp} leads to a higher level-of-detail in the process of 3D point cloud sparsification. 
\begin{figure}[htb]
\begin{subfigure}[c]{0.47\textwidth}
	\includegraphics[height=4cm]{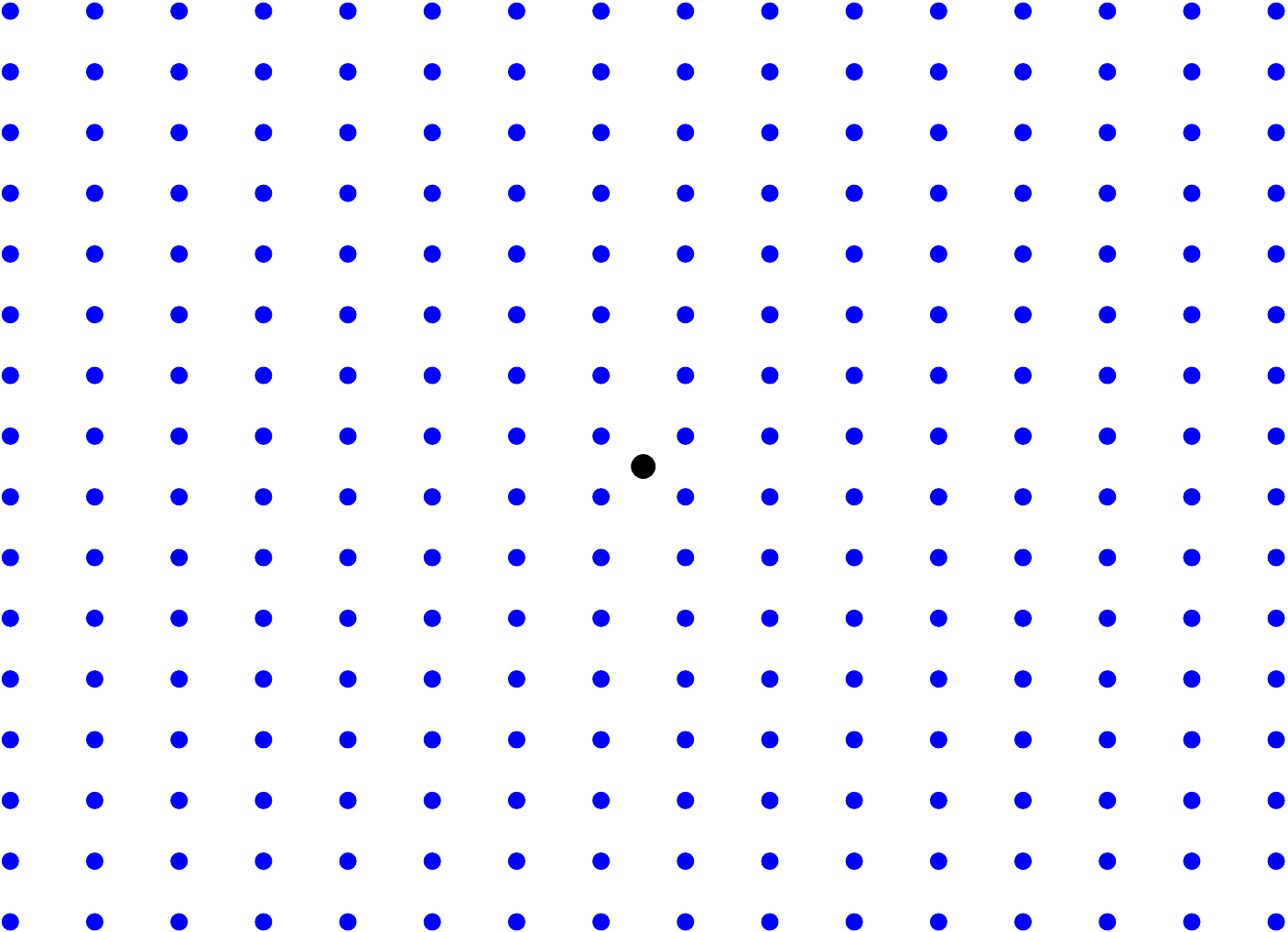}
	\subcaption{Partition after iteration $1$}
\end{subfigure}\hfill
\begin{subfigure}[c]{0.47\textwidth}
	\includegraphics[height=4cm]{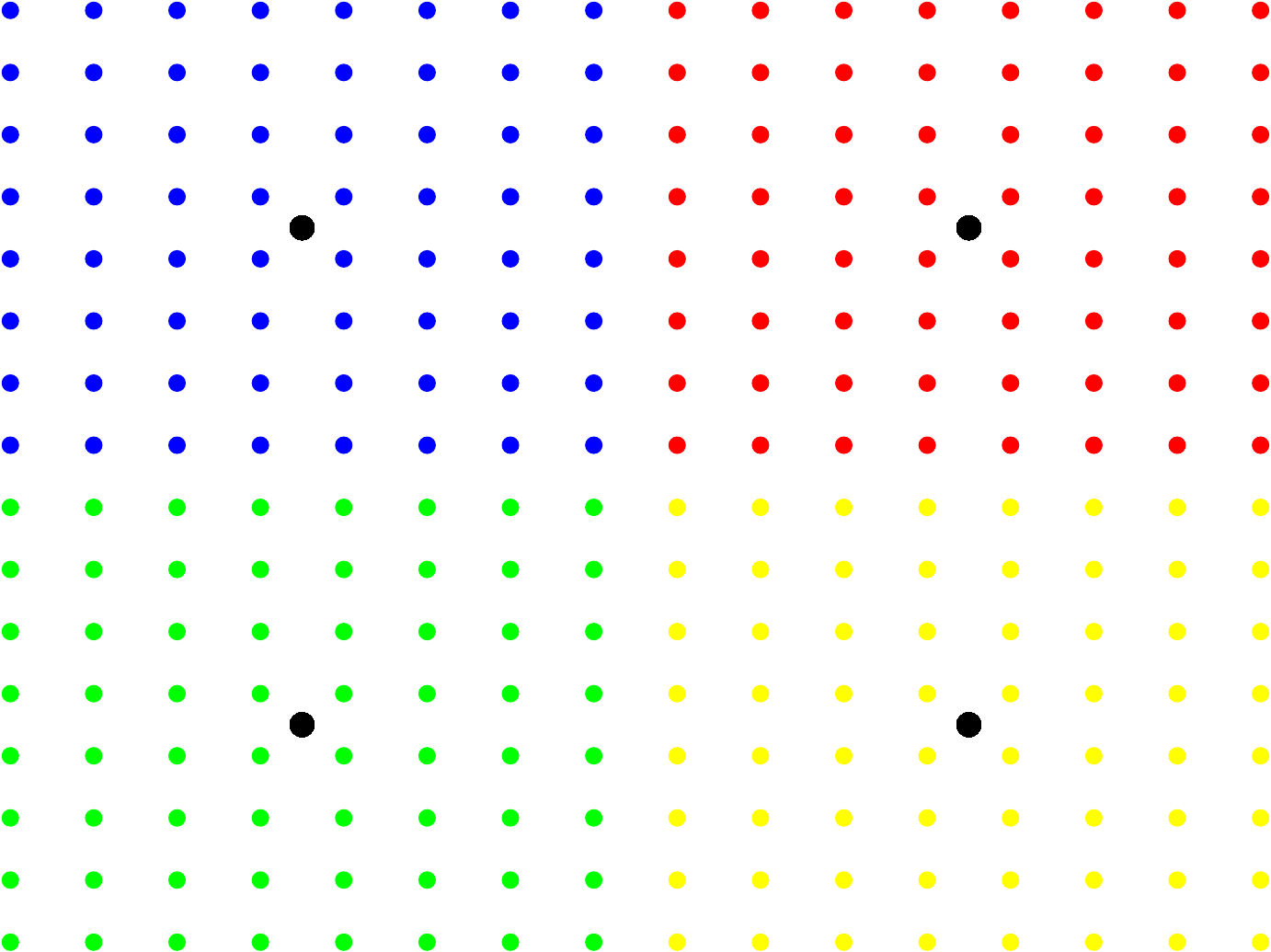}
	\subcaption{Partition after iteration $2$}
\end{subfigure}\\[0.4cm]
\begin{subfigure}[c]{0.47\textwidth}
	\includegraphics[height=4cm]{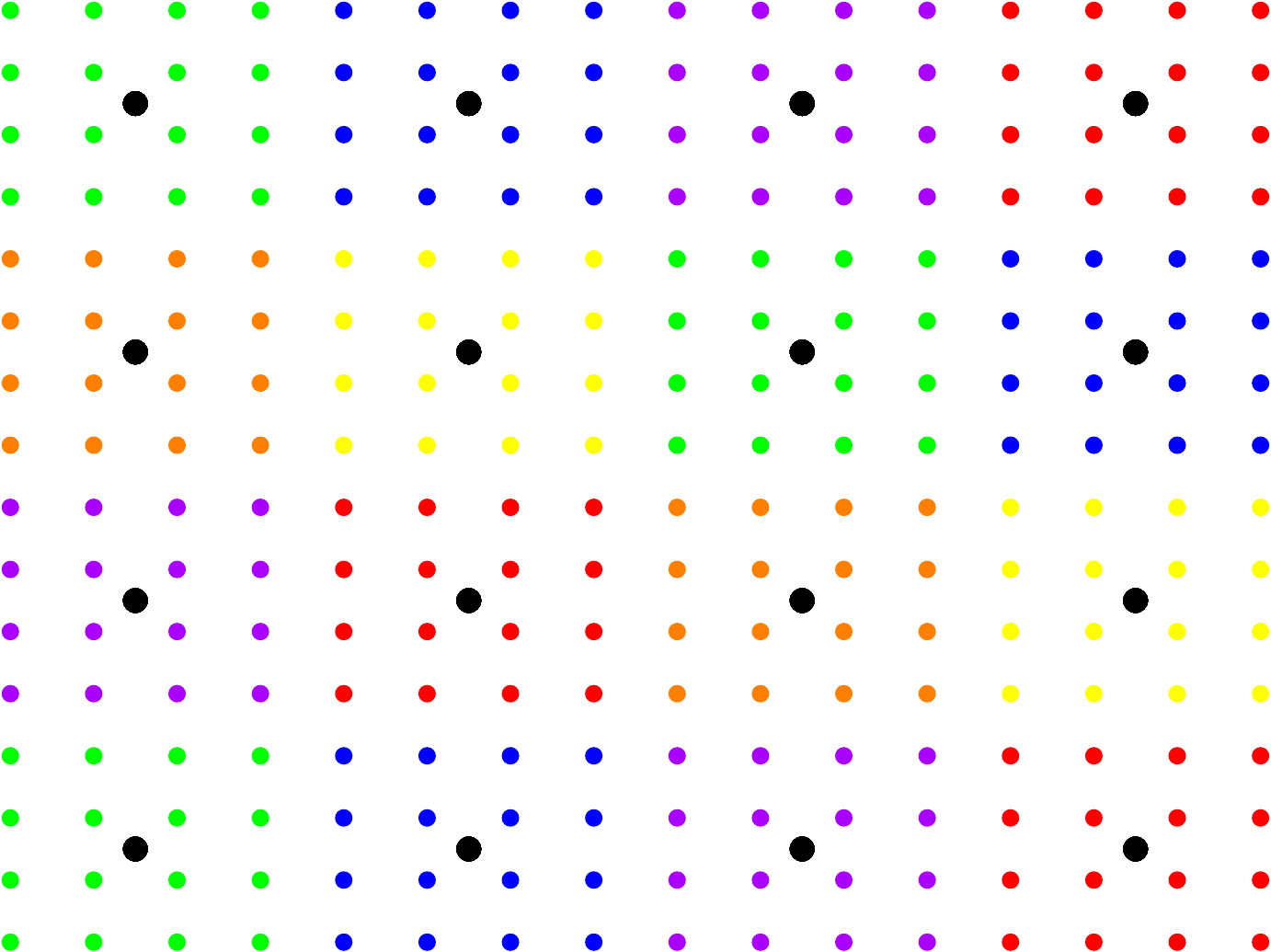}
	\subcaption{Partition after iteration $3$}
\end{subfigure}\hfill
\begin{subfigure}[c]{0.47\textwidth}
	\includegraphics[height=4cm]{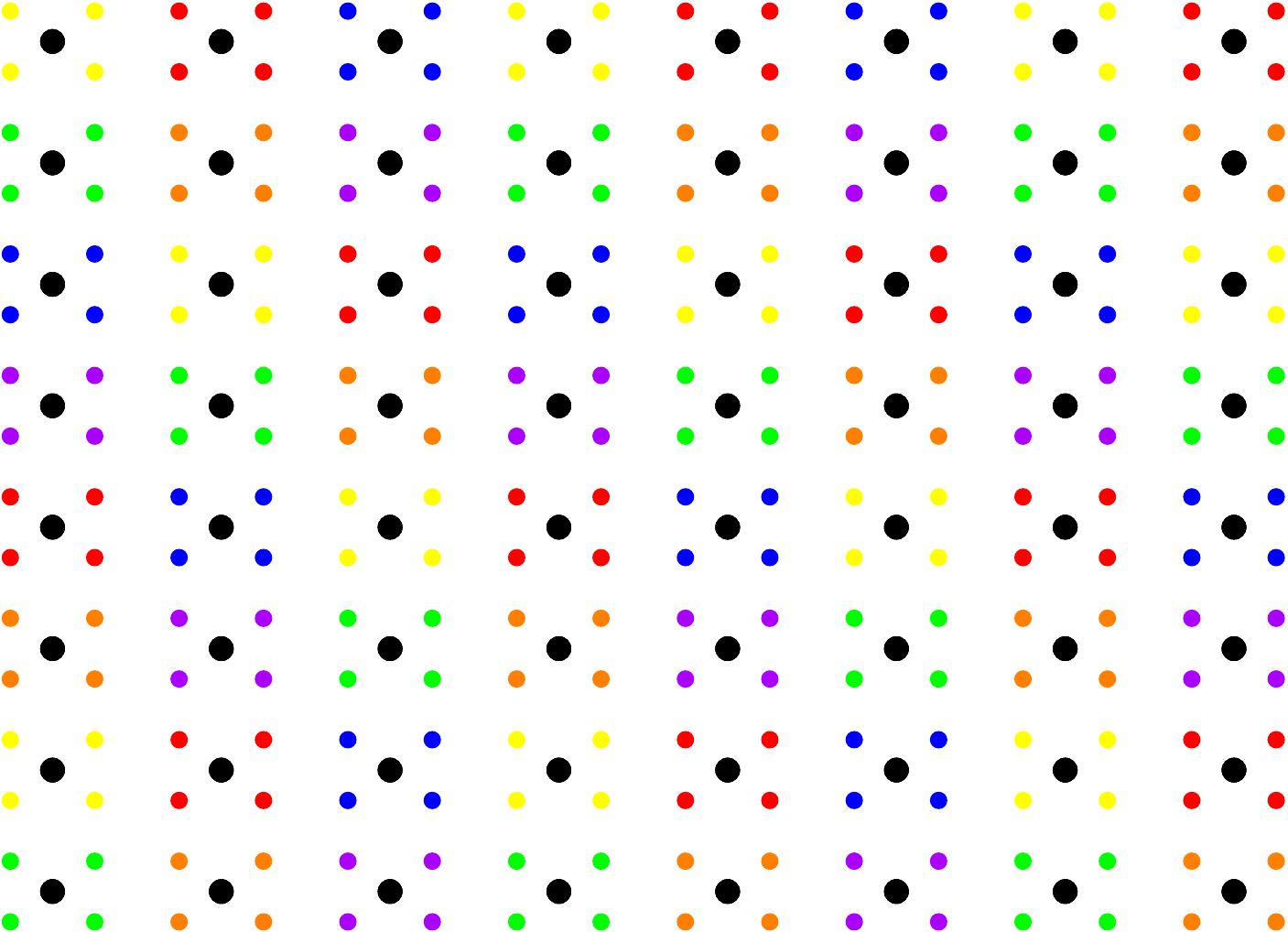}
	\subcaption{Partition after iteration $4$}
\end{subfigure}
\caption{Subsequent iterations of our proposed method and the special case $\alpha=\beta=0$ for point cloud sparsification on a two-dimensional grid. Points belonging to the same subset of the current partition have the same color, while the larger black dots represent the respective cluster centers.}
\label{fig:Octree}
\end{figure}

In Figure \ref{fig:Octree} we demonstrate this special case of the proposed method on a given point cloud for increasing number of iterations. For the sake of visualization we perform this experiment only on a two-dimensional point cloud consisting of $16^2$ points on an equidistant grid, hence, obtaining a quadtree approximation. Points being assigned to the same subset of the current partition are shown in the same color. For each subset we compute the current mean value as cluster center illustrated by a larger black dot. As can be seen between the different iterations the next partition solely depends on the relative position of the points to the current cluster centers.

Note that the user has to terminate the iteration scheme in Algorithm \ref{alg:cp} at the desired level-of-detail by stopping at a certain iteration, as otherwise the octree approximation scheme will iterate until the original point cloud is obtained. This is another disadvantage of this standard method for point cloud sparsification. In Section \ref{ss:noisy_data} we compare the octree approximation scheme to our proposed approach on noisy data.

\subsection{Comparison of fine-to-coarse and coarse-to-fine sparsification strategies}
In the following we compare the results of point cloud sparsification on three different 3D point clouds via the proposed Cut Pursuit algorithm in Section \ref{ssec:cutpursuit} and a direct minimization of the energy functional \eqref{eq:problem} via a primal-dual method using all vertices of the original data. For the following numerical experiments we are using only dense point clouds without any additional geometric noise. We perform minimization of the full variational model via the primal-dual algorithm as introduced in Section \ref{ssec:red_problem} until a relative change $\Delta J_{rel}$ of the energy functional between two subsequent iterations is below $10^{-5}$. The resulting point clouds show many clusters of points that have been attracted to common coordinates. We apply a filtering step on these resulting clusters that removes all but one point in a neighborhood of radius $\epsilon = 10^{-3}$ relative to the size of the data domain. This approach can be seen as \textit{fine-to-coarse} sparsification and has been used before, e.g., in \cite{lozesdiss,lozes_2014_imagestuff}. On the other hand the proposed Cut Pursuit method is clearly a \textit{coarse-to-fine} sparsification strategy.

\subsubsection{Run time comparison of the two strategies for \dt{anisotropic} $\ell_1$ regularization}
To analyze the run time behavior of these approaches, we compare three datasets, namely \textit{Bunny}, \textit{Happy} and \textit{Dragon}, from the Stanford 3D Scanning Repository \cite{stanfordmodels} for \dt{anisotropic $\ell_1$ regularization, i.e., anisotropic total variation for $p=q=1$ in \eqref{eq:p-q-norm},} and two different regularization parameters. Additionally, we compare the simple Cut Pursuit algorithm with a variant in which the reduced problem is solved by a primal-dual method with additional diagonal preconditioning \cite{diagonalPPD} as described in Section \ref{ss:diagprecon}.
For the fine-to-coarse strategy we use the same primal-dual algorithm with diagonal preconditioning as otherwise the optimization would be slower by orders of magnitude. This statement holds also true when using a step size update acceleration as described in \cite{primaldual}.
	
In the following we will compare the run time results of our numerical experiments gathered in Table \ref{t:compare_CPandPD} for two different regularization parameters. As the results of both optimization strategies is almost identically and cannot be seen visually on the resulting sparsified point clouds, we refrain from showing any point cloud visualization here. However, the results of point cloud sparsification using \dt{anisotropic} $\ell_1$ regularization can be seen in Figure \ref{fig:CompareMethods} below.

\begin{table}[tbh]
	\resizebox{\textwidth}{!}{
		\begin{tabular}{|C{0.39\textwidth}|C{0.17\textwidth}|C{0.17\textwidth}|C{0.17\textwidth}|}
			\hline
			\multicolumn{1}{|c|}{Data set / Regularization parameter} & \multicolumn{1}{C{0.17\textwidth}|}{Direct optimiz.\newline via PPD}& \multicolumn{1}{C{0.17\textwidth}|}{Cut Pursuit\newline with PD} & \multicolumn{1}{C{0.17\textwidth}|}{Cut Pursuit\newline with PPD} \\
			\hline
			Bunny ($35,947$ points):  & & &\\
			 $\alpha = 0.001$ & $50$s   & $413$s &$\fg{23}$s\\
			 $\alpha = 0.01$ & $119$s   & $145$s &$\fg{6}$s\\ \hline
			Dragon ($435,545$ points):  & & &\\
			  $\alpha = 0.005$  & $5,636$s  & $1,097$s  & $\fg{117,4}$s \\
			  $\alpha = 0.002$  & $6,314$s  & $591$s  & $\fg{62,5}$s \\ \hline
			Happy ($543,524$ points):  & & &\\
			  $\alpha = 0.0002$ & $1,568$s & $2,400$s &$\fg{222,2}$s\\ 
			  $\alpha = 0.001$ & $3,407$s & $1,186$s &$\fg{93.2}$s\\
			\hline
		\end{tabular}
	}
	\caption{Comparison of overall runtime in seconds between a direct optimization via primal-dual optimization (PD) and Cut Pursuit where the reduced problem was solved with a primal-dual and with a diagonal Preconditioned primal-dual (PPD) algorithm on different point cloud data for \dt{anisotropic $\ell_1$ regularization and} two different regularization parameters $\alpha$. \label{t:compare_CPandPD}}
\end{table}

The first and most obvious observation is that the direct optimization approach, i.e., the fine-to-coarse strategy, performs only well for small point clouds as in the \textit{Bunny} data set. For the \textit{Happy} and \textit{Dragon} data set the measured run time is not reasonable for any real application. Additionally, one can see that the direct optimization approach takes increasingly longer for higher regularization parameters $\alpha$. This means that for an increasingly sparse results one has to take a longer computation time into account.

While comparing the two variants of the Cut Pursuit algorithms with only using primal-dual optimization (PD) and the diagonally preconditioned primal-dual algorithm (PDD) we observed that the latter one is always faster than the simple version. This is due to the reasons we pointed out earlier in Section \ref{ss:diagprecon}, i.e., bad conditioning due to different amount of vertices gathered in each subset of the partition and highly varying values of the accumulated weights between these subsets. Notably, in all tested experiments except the \textit{Bunny} data set the simple Cut Pursuit algorithm \textit{without preconditioning} is significantly faster than the fine-to-coarse strategy using even preconditioned primal-dual minimization.

one interesting observation is that the coarse-to-fine strategy proposed in this paper is not necessarily getting faster for an increasingly higher regularization as one might expect. This can be seen for the \textit{Happy} data set. The reason for this is that there are two opposite effects overlapping. With increasingly higher regularization parameter $\alpha$ one can expect the total number of graph cuts to decrease, which leads to less iterations in Algorithm \ref{alg:proposed_scheme}. However, at the same time the costs of computing the maximum flow within the finite weighted graph may increase due to the increased flow graph edge capacities. Thus, in some cases the computational costs of minimum graph cuts outweighs the benefit of computing less graph cuts for higher regularization parameters. In case of the preconditioned primal-dual algorithm the overall run times are less affected by the choice of the regularization parameters compared to the standard primal-dual variant.

In summary we can observe that for large point clouds a Cut Pursuit approach with a diagonal preconditioned version of the primal-dual optimization algorithm is significantly faster than a direct fine-to-coarse strategy.

\subsubsection{Run time comparison and visual differences for \dt{anisotropic} $\ell_1$ and $\ell_0$ regularization}
In Figure \ref{fig:CompareMethods} we compare the sparsification results of the \dt{anisotropic} $\ell_1$\dt{, i.e., the case $p=q=1$ in \eqref{eq:p-q-norm},} and the \dt{weighted} $\ell_0$ regularization on the three different test data sets used before. We choose the regularization parameters for both methods in such a way that they yield roughly the same number of points in the resulting sparse point clouds. As one can see, the resulting point cloud of the \dt{$\ell_1$} regularization for different data sets always induces a very strong \dt{blocky} structure to the data. This is clear as we have chosen an anisotropic TV regularization for this experiment. In addition to this structural bias one can also observe a volume shrinkage in the resulting point cloud. This is comparable to the typical contrast loss when using \dt{anisotropic} TV regularization for denoising on images, e.g., cf. \cite{debias}. On the other hand we see that \dt{the proposed} $\ell_0$ regularization yields a much more detailed and bias-free result.
\begin{figure}[tbh]
	\subcaptionbox{Result of \dt{anisotropic} $\ell_1$ regularization with $\alpha=0.5$. Time needed: $126$ seconds. Points left: $7794$ $(21.68 \%)$}%
	[.48\textwidth]{\includegraphics[trim={0 0 0  0}, width=0.48\textwidth]{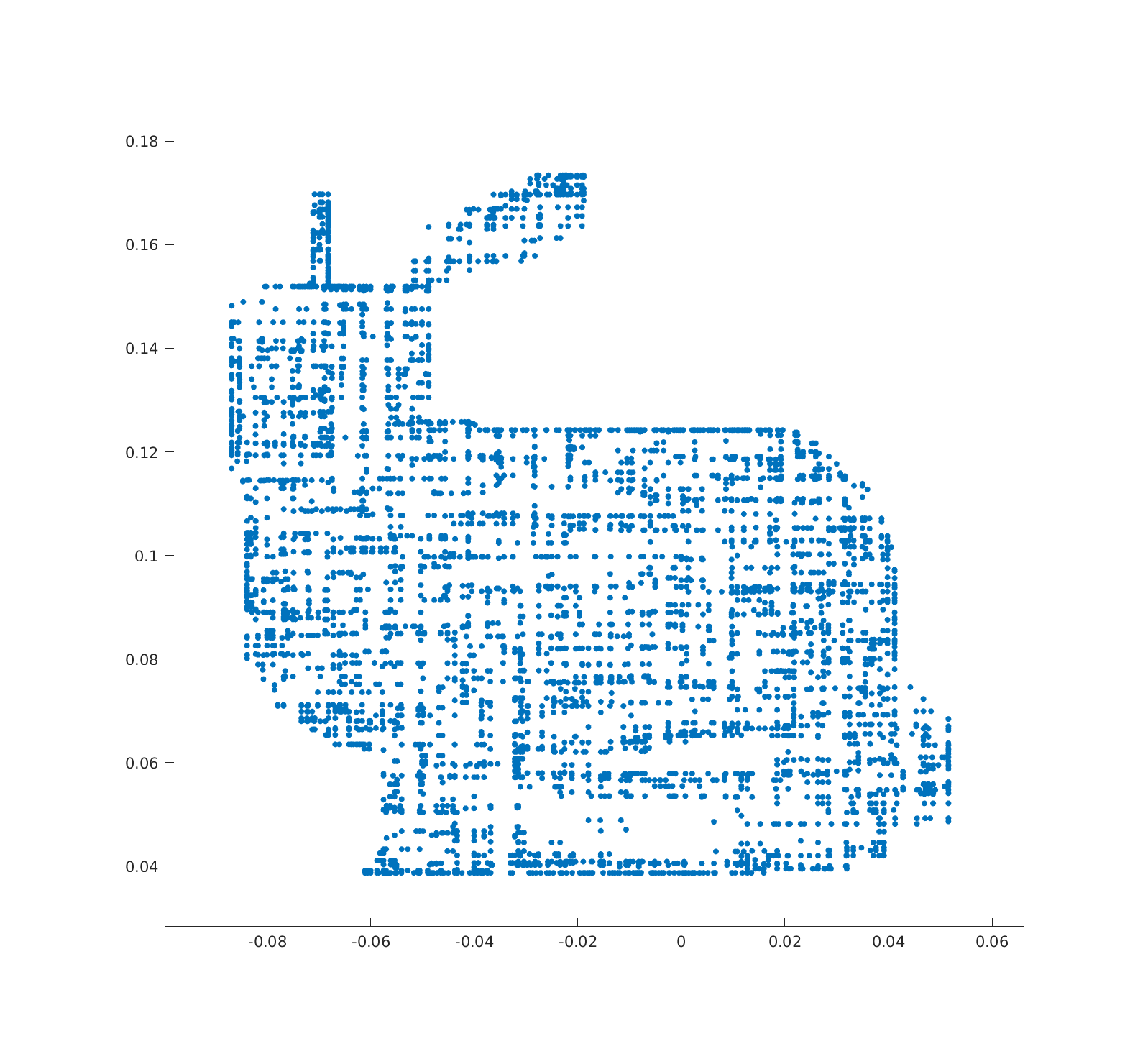} }\hfill
	\subcaptionbox{Result of proposed $\ell_0$ regularization with $\alpha=5$. Time needed: $3.7$ seconds. Points left: $8034$ $(22.35 \%)$}
	[.48\textwidth]{\includegraphics[trim={0 0 0  0}, width=0.48\textwidth]{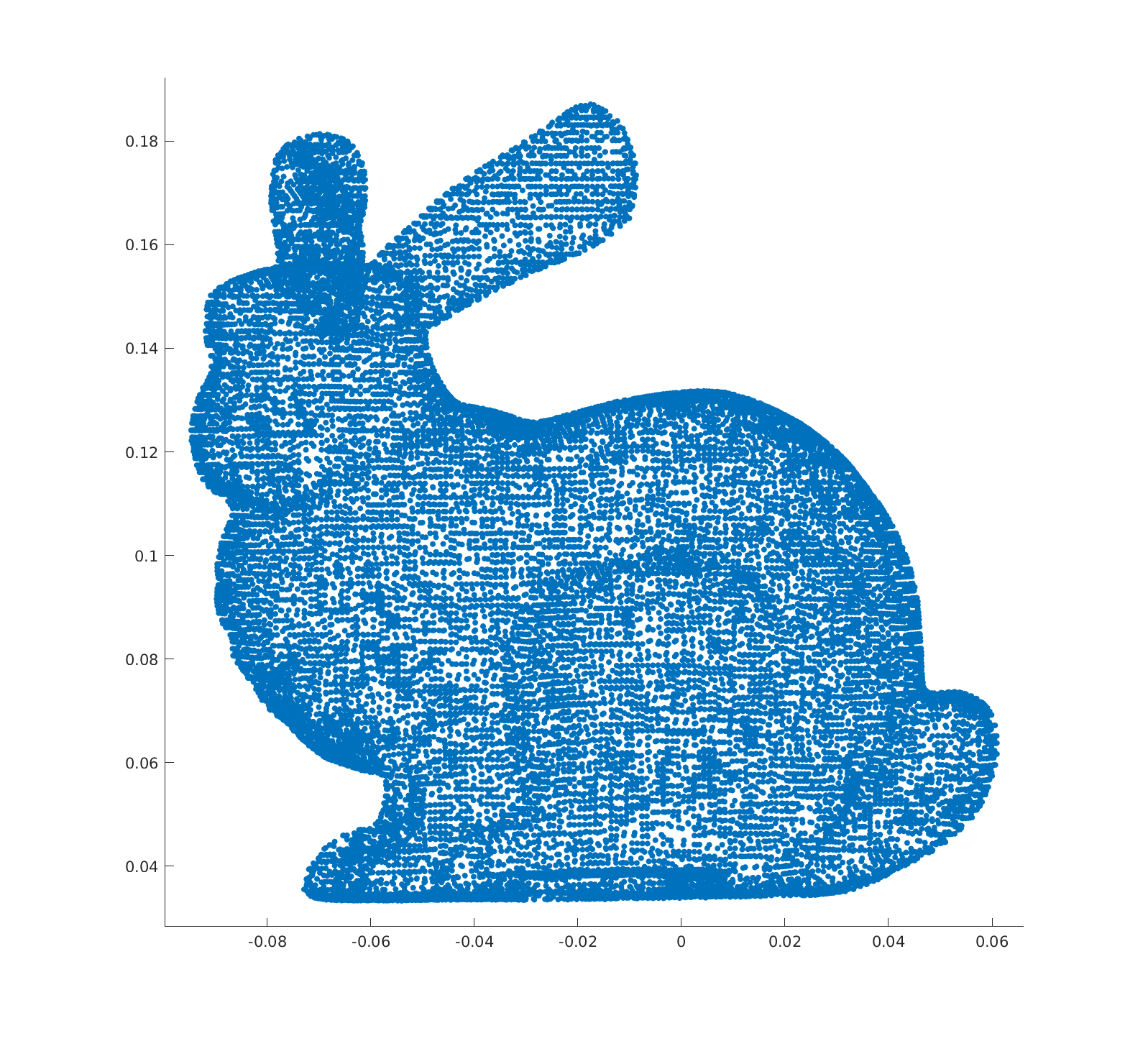} }
	\\
	\subcaptionbox{Result of \dt{anisotropic} $\ell_1$ regularization with $\alpha=0.5$. Time needed: $3,305$ seconds. Points left: $26247$ $(4.83 \%)$}%
	[.22\textwidth]{\includegraphics[trim={0 0 0  0}, width=0.22\textwidth]{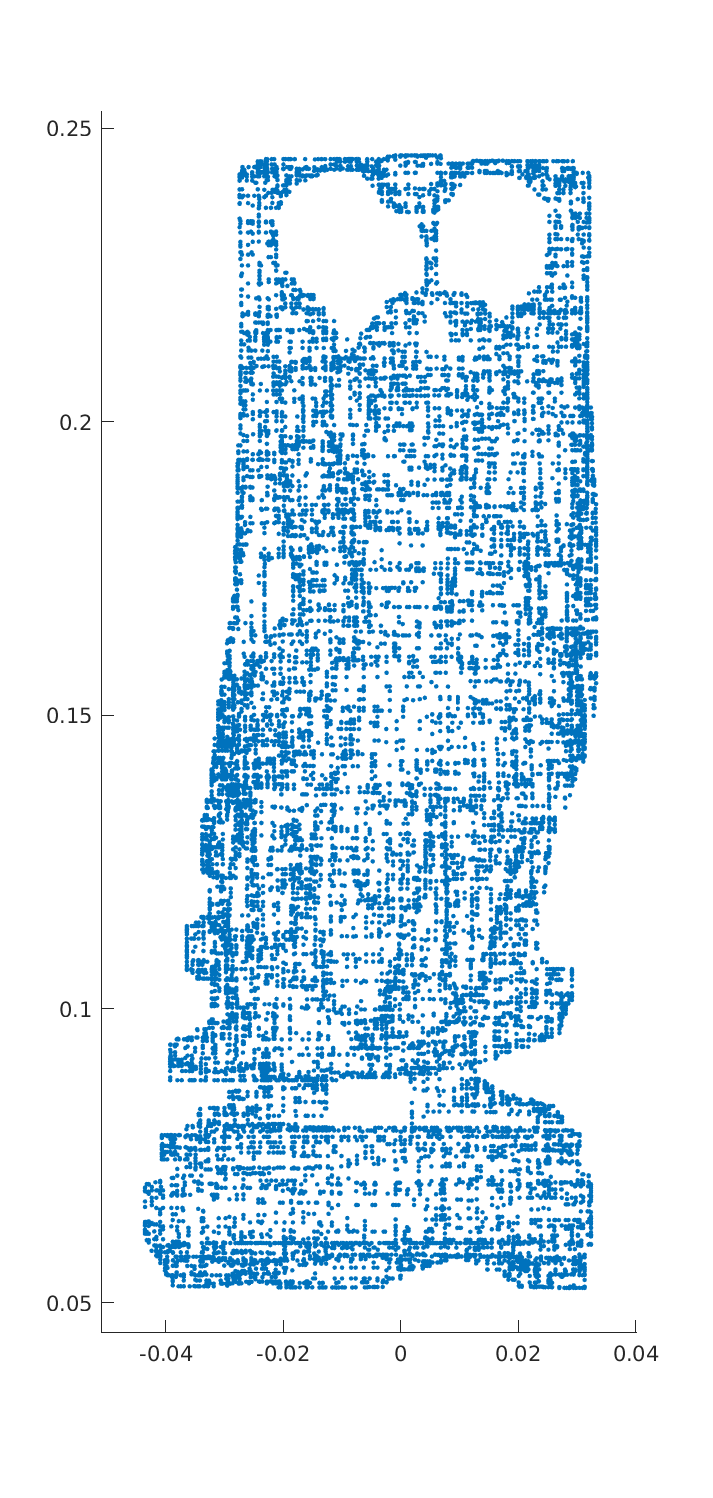} }\hfill
	\subcaptionbox{Result of proposed $\ell_0$ regularization with $\alpha=4$. Time needed: $94$ seconds. Points left: $29168$ $(5.37 \%)$}
	[.22\textwidth]{\includegraphics[trim={0 0 0  0}, width=0.22\textwidth]{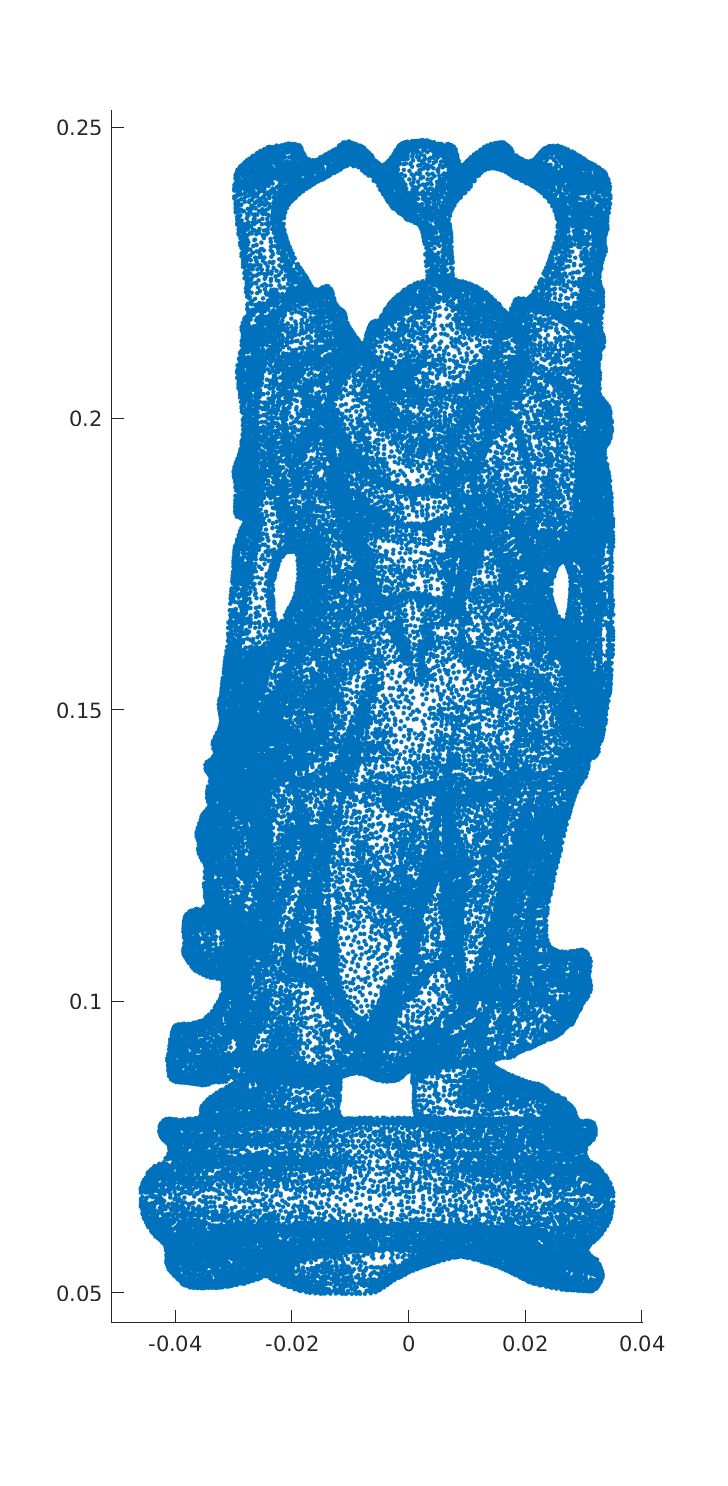} }\hfill
	\subcaptionbox{Result of \dt{anisotropic} $\ell_1$ regularization with $\alpha=0.5$. Time needed: $3,305$ seconds. Points left: $26247$ $(4.83 \%)$}%
	[.22\textwidth]{\includegraphics[trim={0 0 0  0}, width=0.22\textwidth]{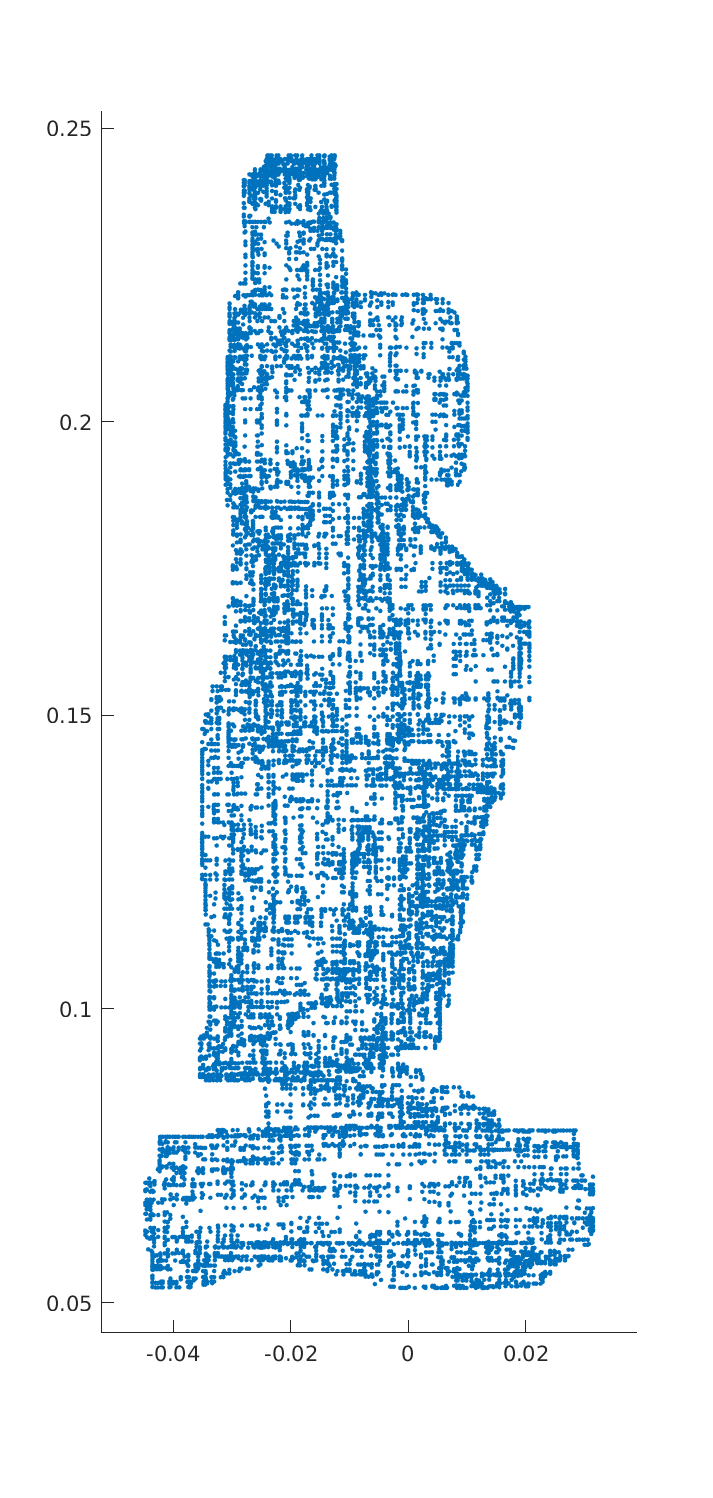} }\hfill
	\subcaptionbox{Result of proposed $\ell_0$ regularization with $\alpha=4$. Time needed: $94$ seconds. Points left: $29168$ $(5.37 \%)$}
	[.22\textwidth]{\includegraphics[trim={0 0 0  0}, width=0.22\textwidth]{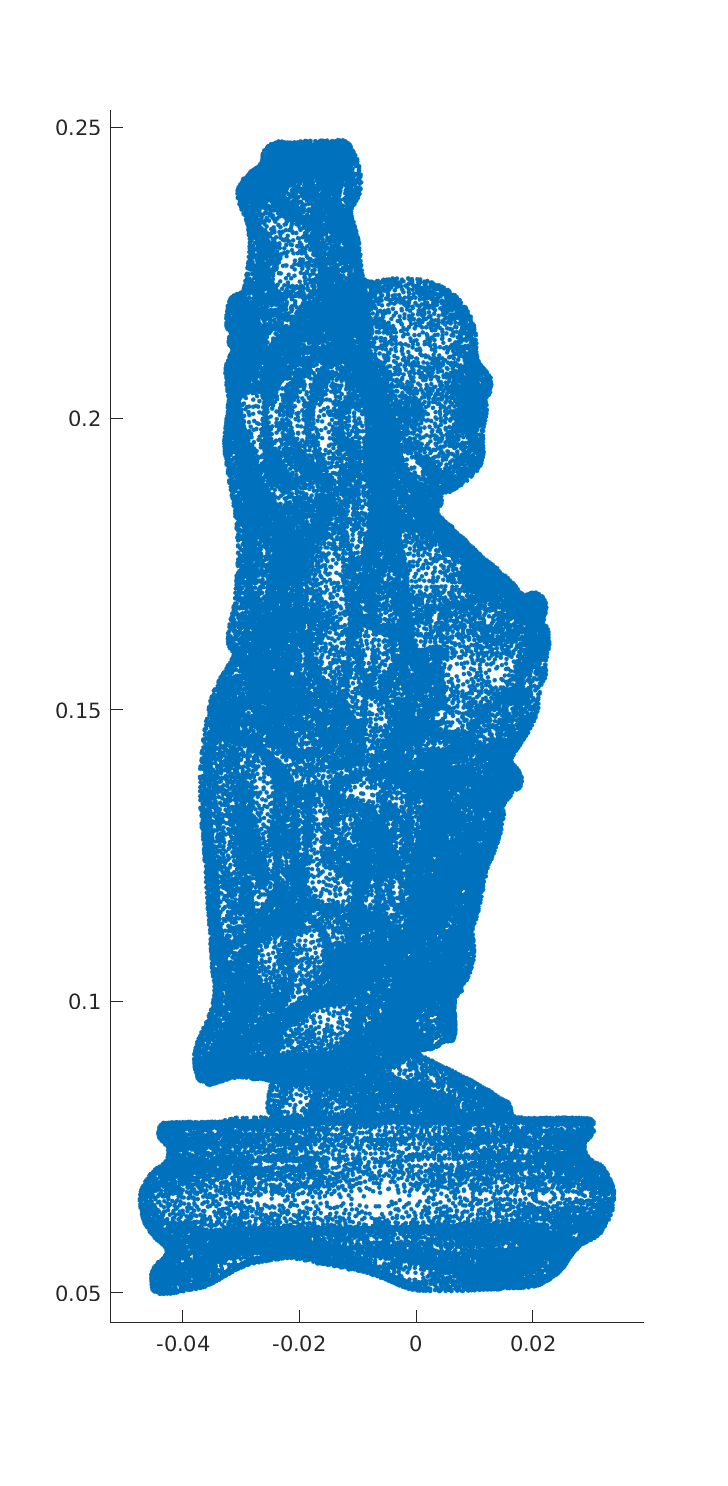} }
	\\
	\subcaptionbox{Result of \dt{anisotropic} $\ell_1$ regularization with $\alpha=0.5$. Time needed: $8,239$ seconds. Points left: $16438$ $(3.77 \%)$}[.48\textwidth]{\includegraphics[trim={0 02cm 2cm  0}, width=0.48\textwidth]{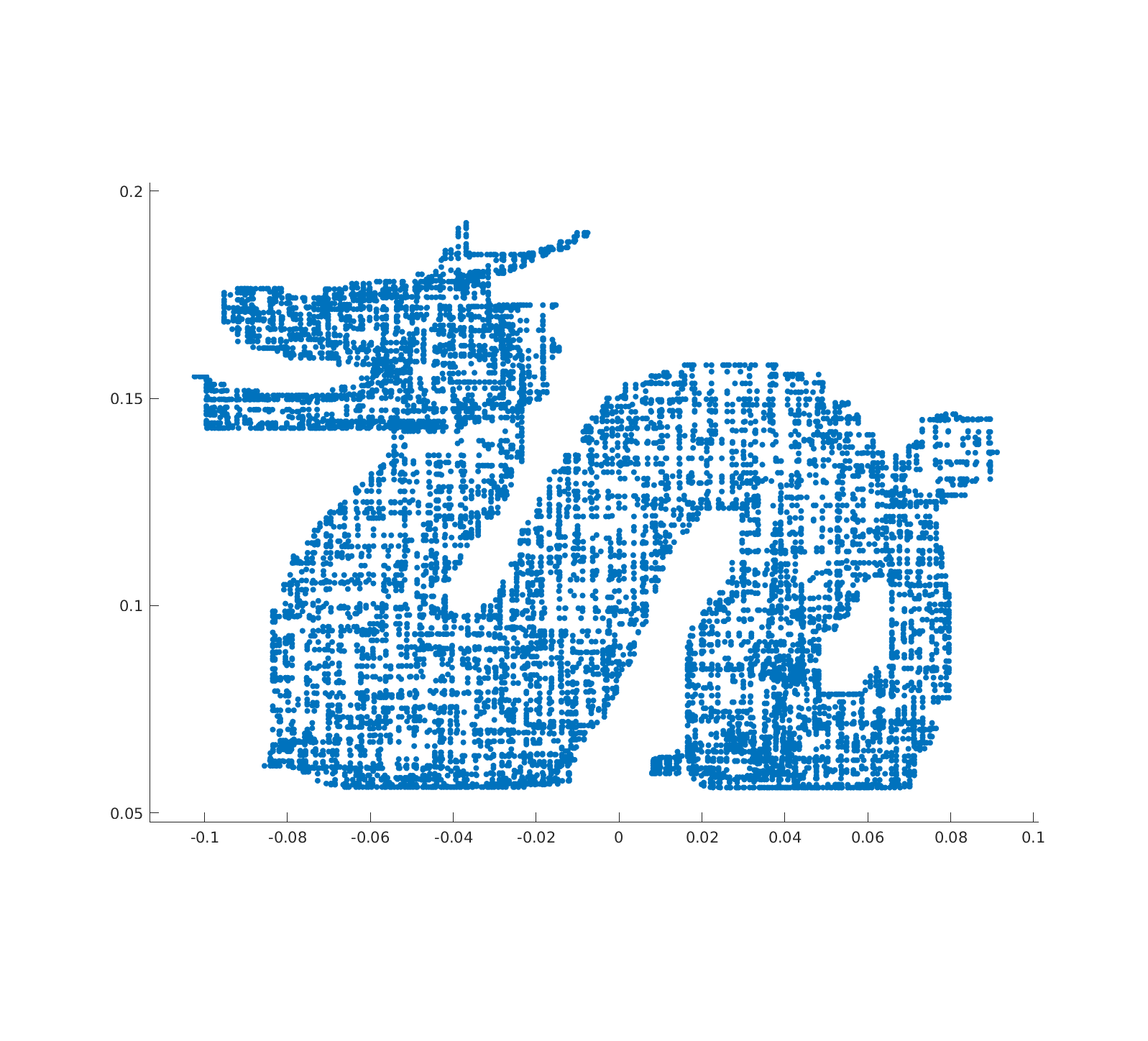} }\hfill
	\subcaptionbox{Result of proposed $\ell_0$ regularization with $\alpha=6.5$. Time needed: $70.6$ seconds. Points left: $16138$ $(3.71 \%)$}	[.48\textwidth]{\includegraphics[trim={0 2cm 2cm  0}, width=0.48\textwidth]{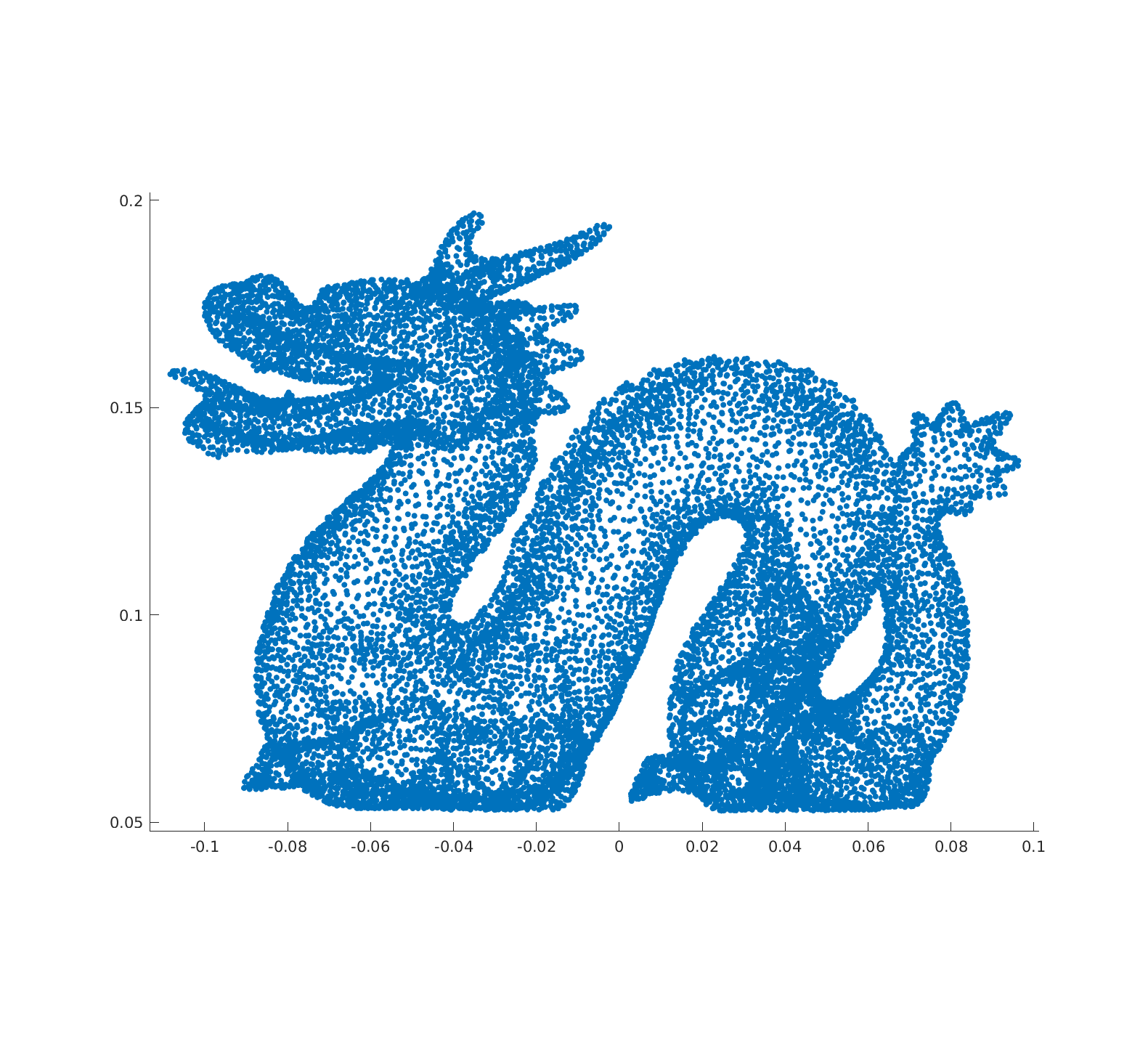} }
	\caption{Comparison of results by \dt{anisotropic} $\ell_1$ regularization (left) and by $\ell_0$ regularization (right) for point cloud sparsification.}
	\label{fig:CompareMethods}
\end{figure}

As we would like to highlight by this experiment, the striking argument for our proposed method is the significant efficiency gain for point cloud sparsification, which can be seen by comparing the computational times in Table \ref{t:computational_time}. Comparing the fine-to-coarse strategy proposed in \cite{lozesdiss,lozes_2014_imagestuff} there is a speed-up by a factor of between \fg{$60$ to $290$} depending on the number of points in the original data set.
Note that \dt{modified Cut Pursuit scheme \ref{alg:proposed_schemel0} with the proposed weighted $\ell_0$} minimizes the energy \dt{very efficiently} since the solution of the reduced problem \eqref{eq:reducedproblem} is just the mean value of each partition. \dt{This speed up of two orders of magnitude (without exploiting any parallelization techniques) renders the proposed method valuable for applications in which point cloud data has to be processed and analysed in near-realtime conditions.}
\begin{table}[tbh]
\resizebox{\textwidth}{!}{
	\begin{tabular}{|R{0.06\textwidth}L{0.14\textwidth} |C{0.32\textwidth}|C{0.3\textwidth}|}
		\hline\multicolumn{2}{|c|}{Data set} & \multicolumn{1}{C{0.32\textwidth}|}{Direct optimization via PPD}& \multicolumn{1}{C{0.3\textwidth}|}{Weighted $\ell_0$ Cut Pursuit} \\
		\hline
		\multicolumn{1}{|r}{\multirow{2}{*}{Bunny:}} &  \multicolumn{1}{l|}{\multirow{2}{*}{$35,947$ points}}   & $126$s   & $2.3$s \\
		& & $8,034$ points left ($22.35\%$) & $8,065$ points left ($22.42\%$)\\ \hline
		\multicolumn{1}{|r}{\multirow{2}{*}{Buddha:}} &  \multicolumn{1}{l|}{\multirow{2}{*}{$543,524$ points}} & $3,305$s & $47$s   \\
		& & $29,168$ points left ($5.37\%$) &  $28,484$ points left ($5.24\%$) \\ \hline
        \multicolumn{1}{|r}{\multirow{2}{*}{Dragon:}} &  \multicolumn{1}{l|}{\multirow{2}{*}{$435,545$ points}}  & $8,239$s & $28.2$s  \\
		&& $16,438$ points left ($3.77\%$) & $16,405$ points left ($3.77\%$)\\
		\hline
	\end{tabular}
	}
\caption{Comparison of overall runtime in seconds between a direct optimization via preconditioned primal-dual optimization (PPD) and the weighted $\ell_0$ Cut Pursuit algorithm for point cloud sparsification tested on the three different datasets presented in Figure \ref{fig:CompareMethods}.\label{t:computational_time}}
\end{table}

To summarize our observations above, we can state that when noise-free data is given, point cloud sparsification can best be performed using the weighted $\ell_0$ regularization as described in Algorithm \ref{alg:proposed_schemel0}.

\subsection{Comparison of \dt{qualitative impact of} different regularization functionals}
In the following experiment we compare the results of point cloud sparsification of the Cut Pursuit algorithm with three different choices of regularization functionals $Q$ and different parameters settings for $\beta$ in the reduced problem \eqref{eq:reducedproblem}.
In particular, we compare the impact of \dt{isotropic} $\ell_2$ and \dt{both anisotropic as well as isotropic} $\ell_1$ regularization in Algorithm \ref{alg:proposed_scheme} and the weighted $\ell_0$ regularizaton described in Algorithm \ref{alg:proposed_schemel0} on the appearance of the resulting sparse point clouds.

\subsubsection{Comparison of \dt{isotropic} $\ell_2$ vs. \dt{anisotropic} $\ell_1$ regularization}
In the first experiment we choose the \textit{Bunny} data set without any geometric noise perturbations and visually compare different levels of point cloud sparsification for \dt{isotropic} $\ell_2$ \dt{($p=q=2$)} and \dt{anisotropic} $\ell_1$ \dt{($p=q=1$)} regularization.
\begin{figure}[tbh]
\begin{subfigure}[c]{0.47\textwidth}
	\includegraphics[height=4cm]{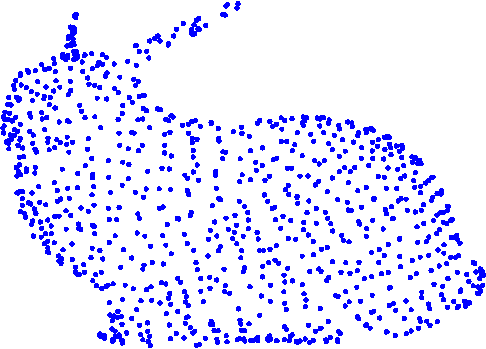} 
	\subcaption{Point cloud for $p=q=2$, $\beta=10$}
	\label{fig:cmp_reg_a}
\end{subfigure}
\hfill
\begin{subfigure}[c]{0.47\textwidth}
	\includegraphics[height=4cm]{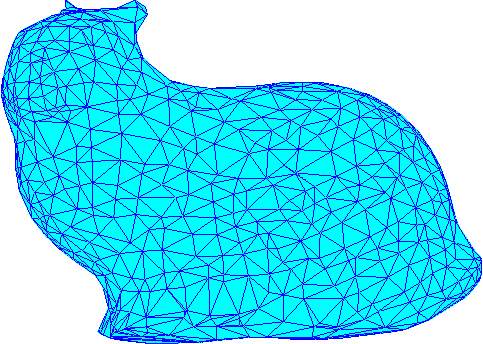}
	\subcaption{Triangulation for $p=q=2$, $\beta=10$}
\end{subfigure}\\[0.4cm]
\begin{subfigure}[c]{0.47\textwidth}
	\includegraphics[height=4cm]{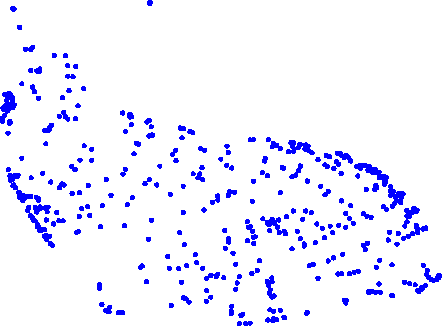}
	\subcaption{Point cloud for $p=q=2$, $\beta=70$}
\end{subfigure}
\hfill
\begin{subfigure}[c]{0.47\textwidth}
	\includegraphics[height=4cm]{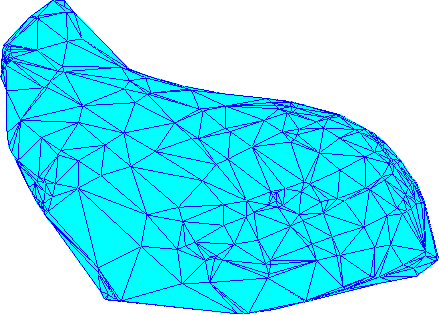}
	\subcaption{Triangulation for $p=q=2$, $\beta=70$}
	\label{fig:cmp_reg_d}
\end{subfigure}\\[0.4cm]
\begin{subfigure}[c]{0.47\textwidth}
	\includegraphics[height=4cm]{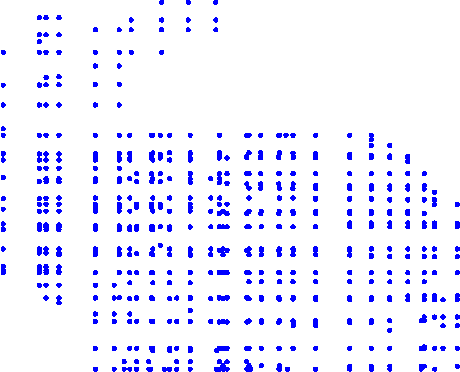}
	\subcaption{Point cloud for $p=q=1$, $\beta=10$}
	\label{fig:cmp_reg_e}
\end{subfigure} \hfill
\begin{subfigure}[c]{0.47\textwidth}
	\includegraphics[height=4cm]{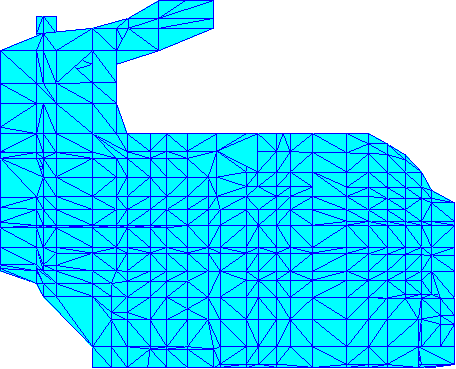}
	\subcaption{Triangulation for $p=q=1$, $\beta=10$}
\end{subfigure}\\[0.4cm]
\begin{subfigure}[c]{0.47\textwidth}
	\includegraphics[height=4cm]{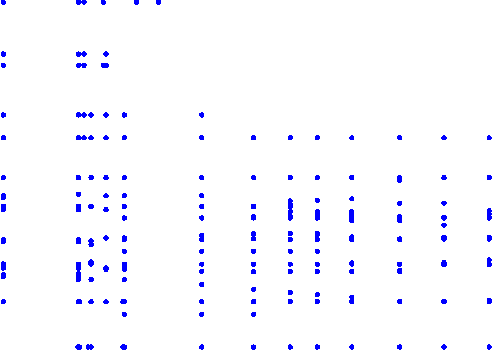}
	\subcaption{Point cloud for $p=q=1$, $\beta=50$}
\end{subfigure} \hfill
\begin{subfigure}[c]{0.47\textwidth}
	\includegraphics[height=4cm]{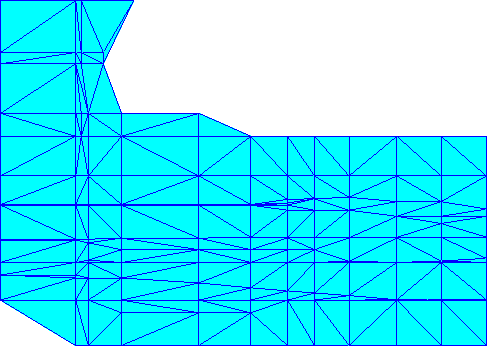}
	\subcaption{Triangulation for $p=q=1$, $\beta=50$}
	\label{fig:cmp_reg_h}
\end{subfigure}
\caption{Comparison of point cloud sparsification results using different regularization settings based on the parameters $p,q$ and $\beta$ in \eqref{eq:reducedproblem}.}
\label{fig:ComparisonRegularizers}
\end{figure}
In the left column of Figure \ref{fig:ComparisonRegularizers} we show the sparse point clouds after convergence of the proposed minimization scheme in Algorithm \ref{alg:proposed_scheme}, and in the right column we show the resulting triangulation of the models surface. As one can observe with increasing regularization parameter $\beta$ we force the solution to be more biased in terms of the appearance we dictate by the regularizer.  In particular, if we choose $p=q=2$ the solution of the reduced problem \eqref{eq:reducedproblem} corresponds to filtering by the standard graph Laplacian, which leads to rather smooth and round surface approximations as illustrated in Figure \ref{fig:cmp_reg_a}-\subref{fig:cmp_reg_d}. On the other hand, if we choose $p=q=1$ we perform an anisotropic total variation filtering on the 3D points, which yields the results presented in Figure \ref{fig:cmp_reg_e}-\subref{fig:cmp_reg_h}. The resulting sparse point clouds show planar surface regions with steep jumps between them. This blocky appearance can be interpreted as a well-known artifact of anisotropic total variation regularization known as 'staircase effect', e.g., in image processing. This regularization is rather inappropriate for 3D point clouds of natural objects but might be interesting for special application cases in which the scanned object is known to have planar surfaces, e.g., in industrial fabrication.

\dt{
\subsubsection{Visual difference between anisotropic/isotropic $\ell_1$ and $\ell_0$ regularization}
In the following we compare the qualitative difference of point cloud sparsification between the anisotropic ($p=q=1$) and the isotropic ($p=1, q=2$) $\ell_1$ regularization term in the reduced problem \eqref{eq:reducedproblem}. We use the same regularization terms for the minimum graph cut step \eqref{eq:partitionproblem}, i.e., $Q = R$ in the alternating minimization scheme \ref{alg:proposed_scheme}. In the isotropic case we use the proposed heuristic method for determining an optimal descent direction as explained in Section \ref{sss:optimal_descent}, Case 3 ($p=1, q>1$). Additionally, we visually compare the results of the $\ell_1$ regularized point cloud sparsification with the results of the proposed weighted $\ell_0$ regularization from Section \ref{ss:l0}. For the latter case we use the same graph cut method as for the isotropic $\ell_1$ regularization.

As point cloud data we chose the \fg{Fandisk model (cf. \cite{bilateralfiltering})}, which consists of \fg{a combination of roundish and} flat surfaces as well as sharp edges. This data set is often used to evaluate the effectiveness of point cloud denoising methods in the literature, e.g., see \fg{\cite{bilateralfiltering,bilateralnormal2011, l0denoising2015}}.
For this experiment we constructed a symmetrized $k$-nearest neighbor graph for $k= \fg{7}$ and set the regularization parameter $\alpha$ such that all resulting point clouds have roughly the same compression rate of $17\%$. The regularization parameters used for each regularization term are indicated in Figure \ref{fig:Fandisk}.

In Figure \ref{fig:Fandisk} the results of point cloud sparsification with the three described regularization terms are displayed. The left column shows a mesh triangulation of the data, while we present a corresponding surface rendering with Phong lighting in the right column.
The first row in Figure \ref{fig:original_Fandisk_blue}-\ref{fig:original_Fandisk} shows the original Fandisk data set, which consists of $11,949$ 3D points. In the second and third row we present the results of anisotropic and isotropic $\ell_1$ regularization, respectively. As can be seen the anisotropic $\ell_1$ regularization induces flat surface regions that coincide with the planes that are spanned between the coordinate axes of the data set. This is not surprising as the anisotropic $\ell_1$ regularization decouples the 3D point coordinates and only enforces regularity within each dimension. This leads to typical staircase artifacts in Figure \ref{fig:fandisk_aniso_l1_blue}-\ref{fig:fandisk_aniso_l1} as it is well-known for total variation regularization in imaging applications. On the other hand, using the isotropic $\ell_1$ regularization term for $q=2$ couples the coordinates of each 3D point and hence does not lead to any staircase artifacts as can be seen in Figure \ref{fig:fandisk_iso_l1_blue}-\ref{fig:fandisk_iso_l1}. The resulting surfaces appear much smoother compared to the previously discussed anisotropic $\ell_1$ regularization. While the round and flat parts of the data set are well-preserved by this regularization term the sharp edges are lost as can been observed. The last row shows the results of our proposed weighted $\ell_0$ regularization. As we demonstrate in Figure \ref{fig:fandisk_iso_l0_blue}-\ref{fig:fandisk_iso_l0} the mesh triangulation of the sparsified point cloud is much more regular compared to our experiments with the $\ell_1$ regularization terms. Additionally, the sharp edge features of the Fandisk data set are significantly better preserved.
\begin{figure}[tbh]
    \begin{subfigure}[c]{0.47\textwidth}
        \includegraphics[height=4cm, trim=4cm 9cm 4cm 9cm, clip]{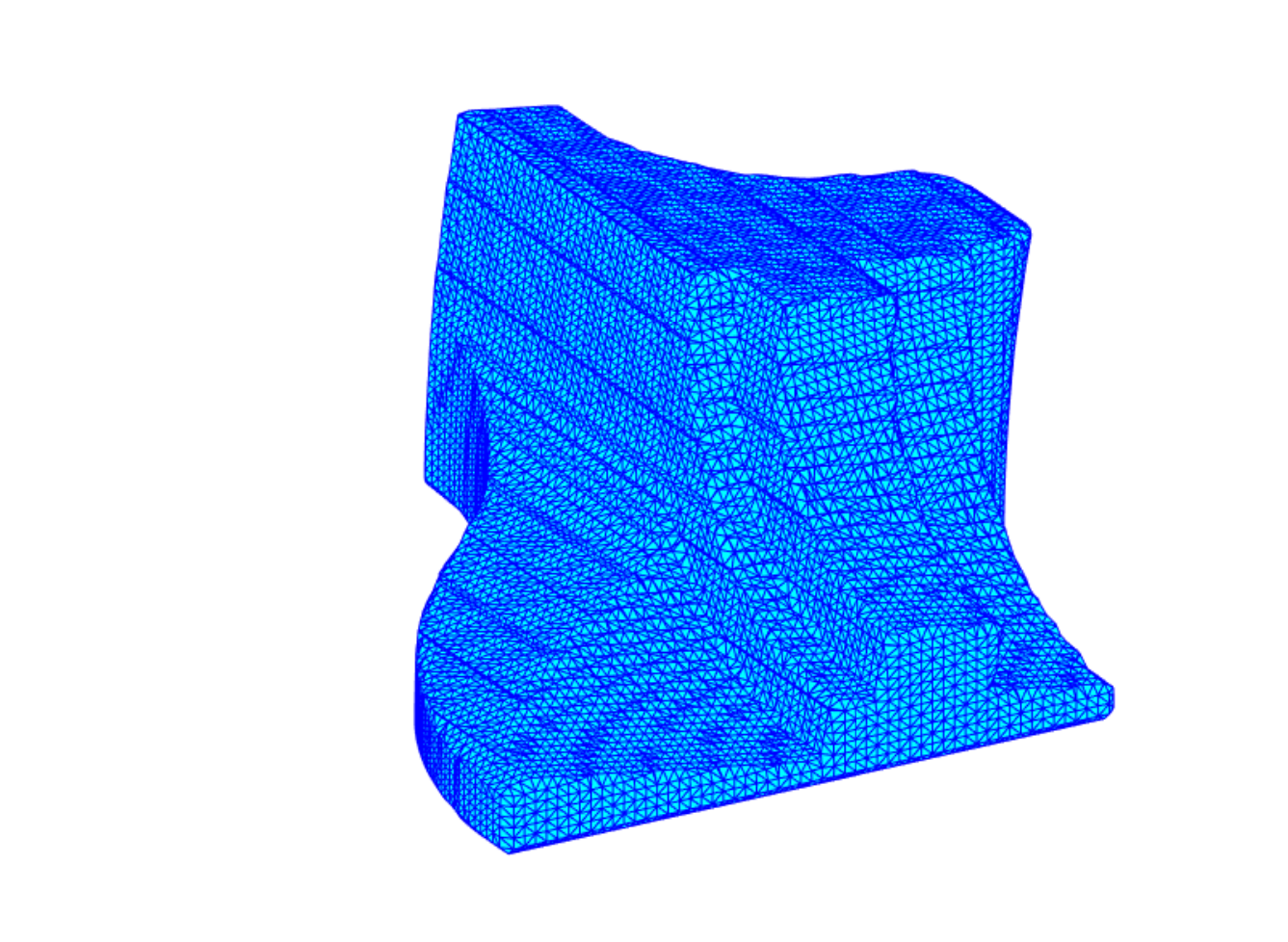}
        \subcaption{Mesh triangulation of the full point cloud of the Fandisk data set}
        \label{fig:original_Fandisk_blue}
    \end{subfigure}\hfill
	\begin{subfigure}[c]{0.47\textwidth}
        \includegraphics[height=4cm, trim=4cm 9cm 4cm 9cm, clip]{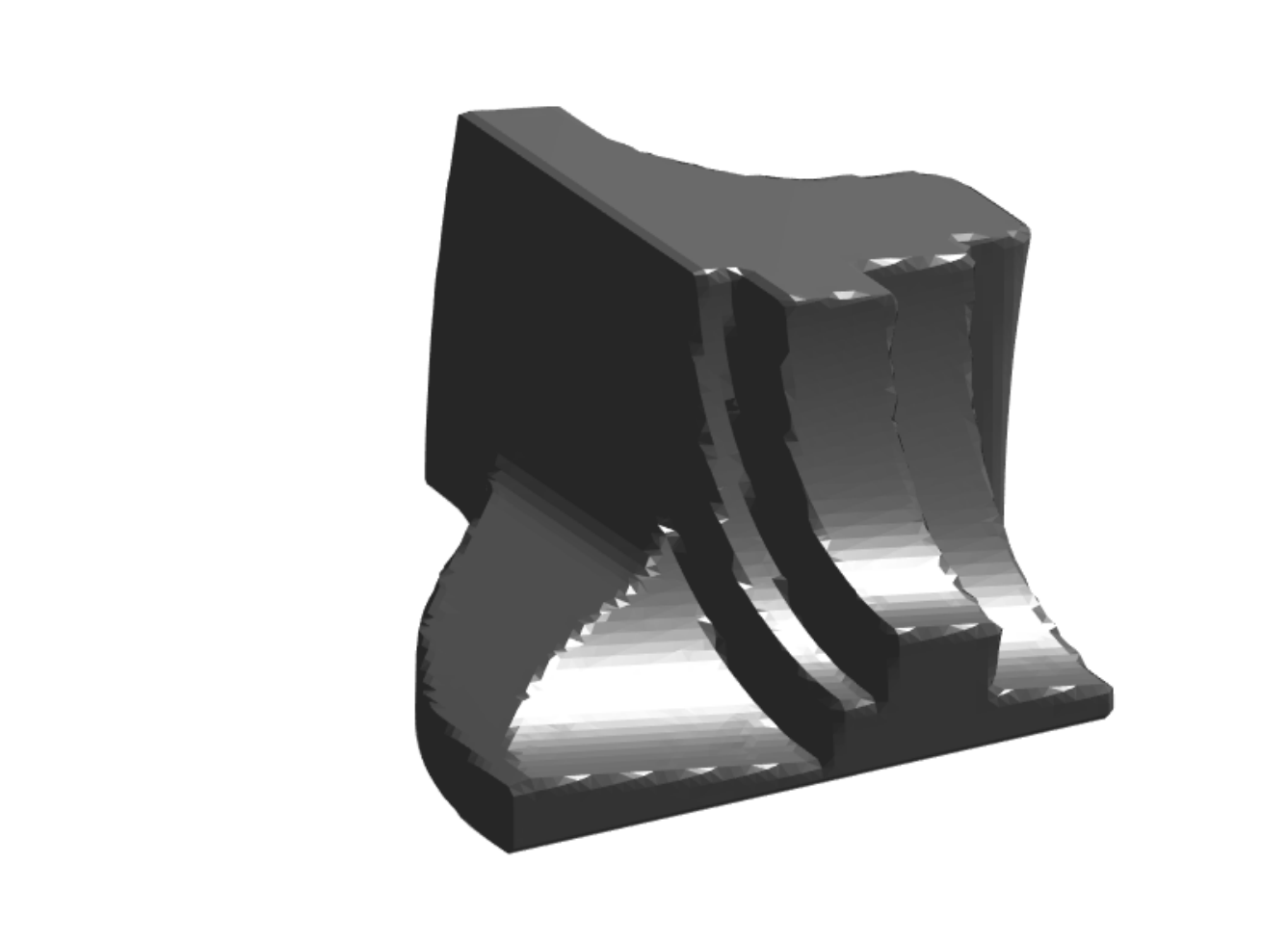}
        \subcaption{Surface rendering of the full point cloud of the Fandisk data set}
        \label{fig:original_Fandisk}
    \end{subfigure}\\[0.4cm]    
	\begin{subfigure}[c]{0.47\textwidth}
        \includegraphics[height=4cm, trim=4cm 9cm 4cm 9cm, clip]{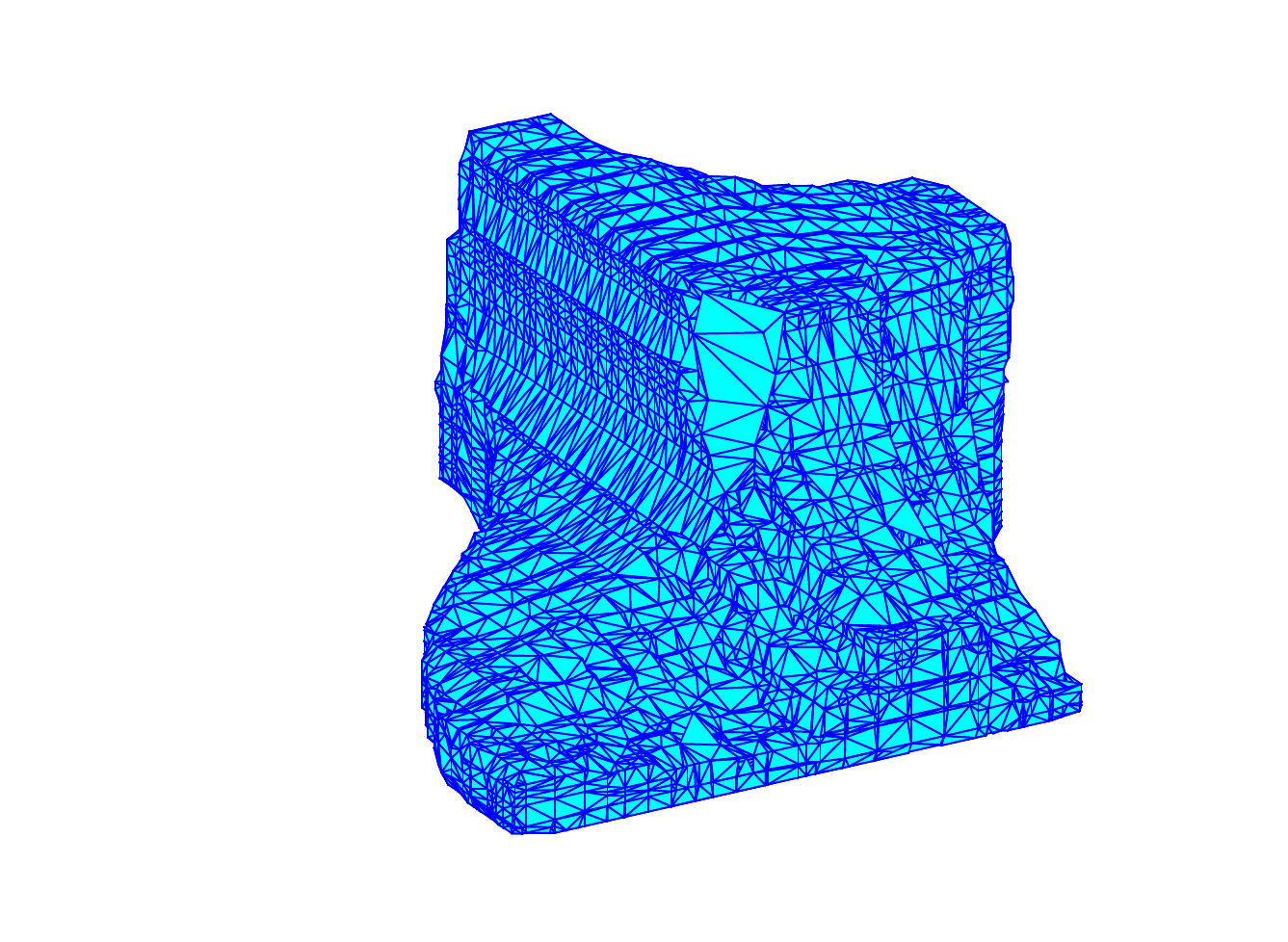}
        \subcaption{Mesh triangulation of the sparsified point cloud using anisotropic $\ell_1$ regularization ($p=q=1$) for $\alpha = \beta = 0.15$}
        \label{fig:fandisk_aniso_l1_blue}
    \end{subfigure}\hfill    
    \begin{subfigure}[c]{0.47\textwidth}
        \includegraphics[height=4cm, trim=4cm 9cm 4cm 9cm, clip]{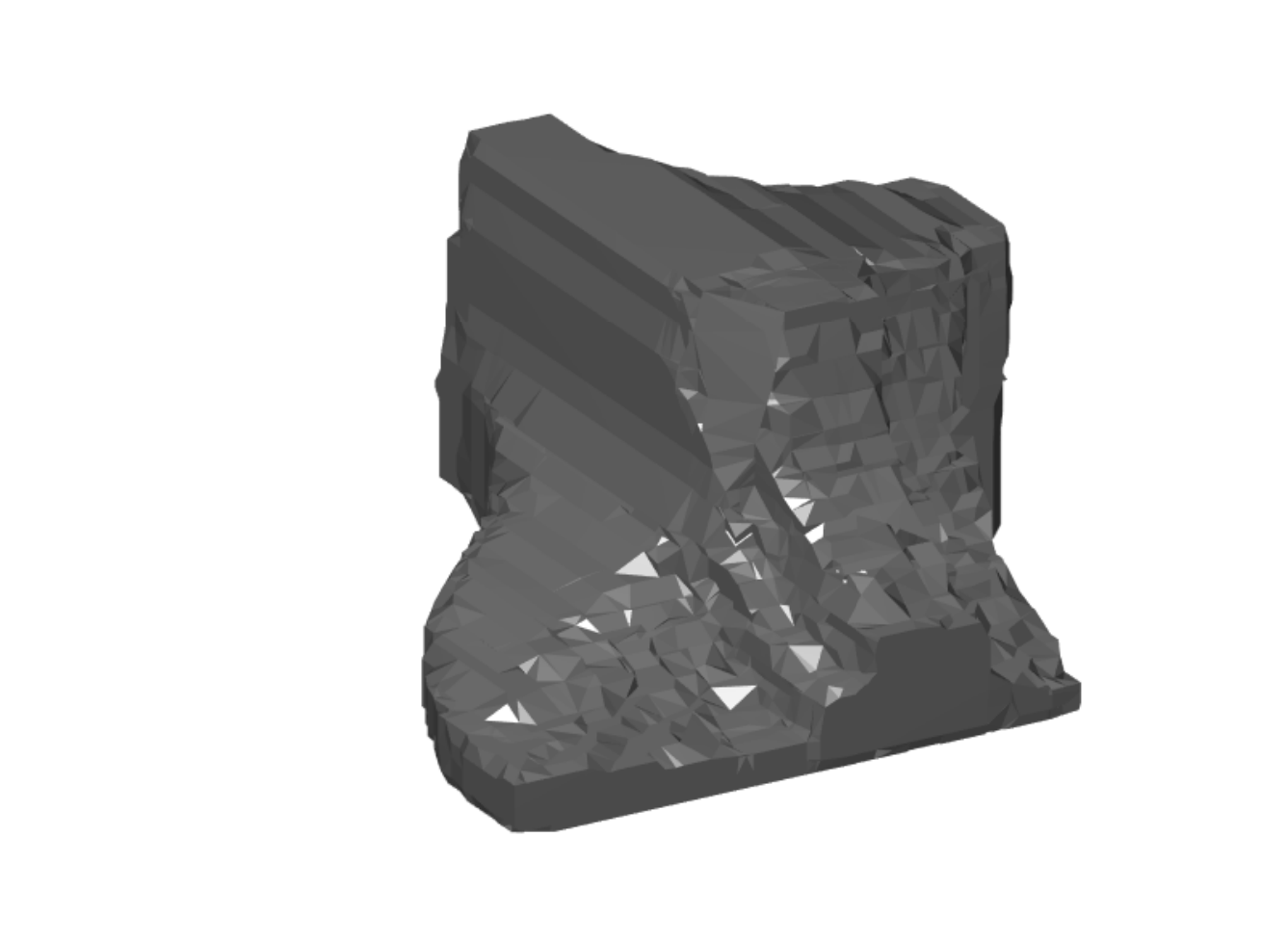}
        \subcaption{Surface rendering of the sparsified point cloud using anisotropic $\ell_1$ regularization ($p=q=1$) for $\alpha = \beta = 0.15$}
        \label{fig:fandisk_aniso_l1}
    \end{subfigure}
    \\[0.4cm]
	\begin{subfigure}[c]{0.47\textwidth}
        \includegraphics[height=4cm, trim=4cm 9cm 4cm 9cm, clip]{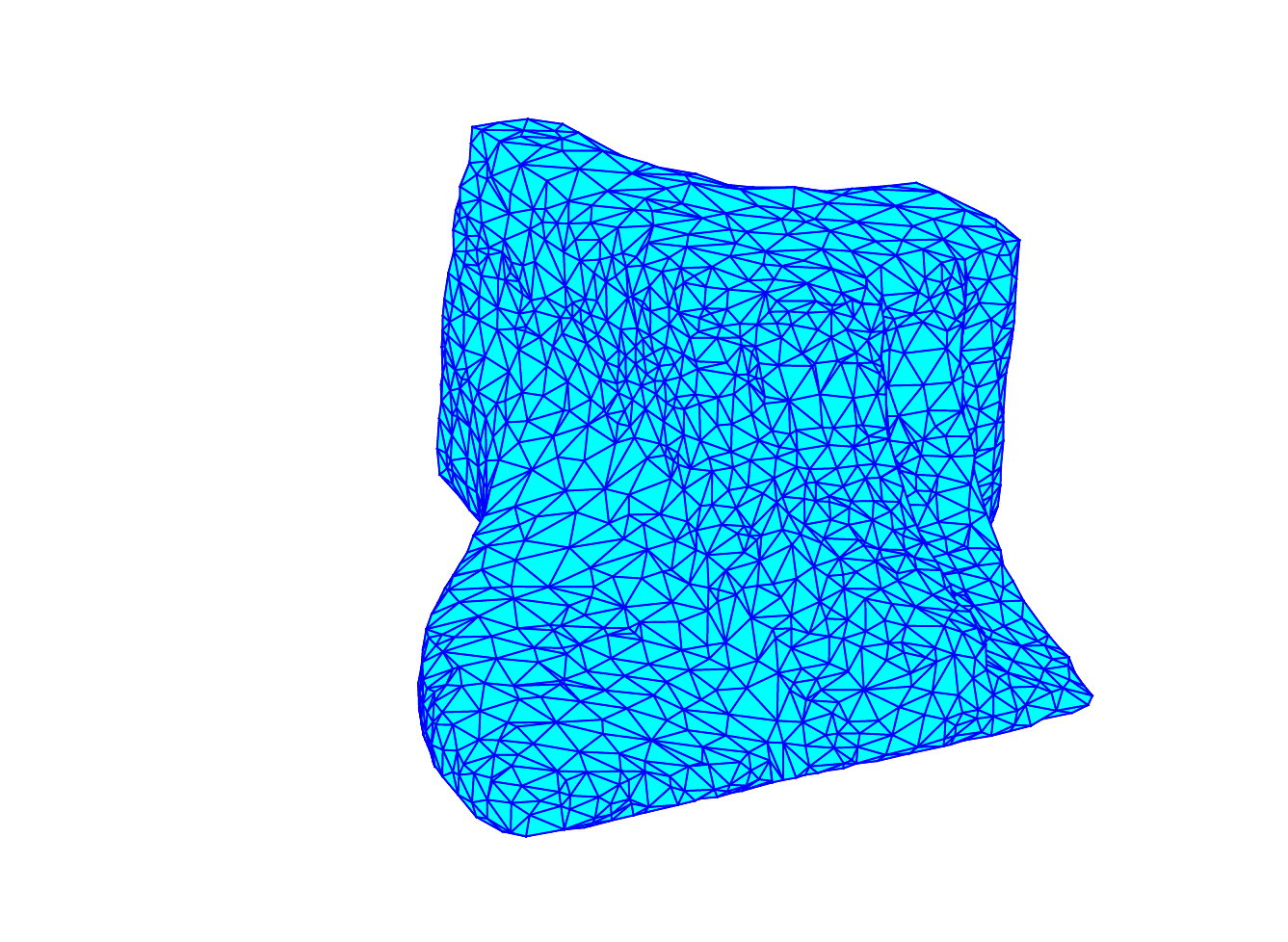}
        \subcaption{Mesh triangulation of the sparsified point cloud using isotropic $\ell_1$ regularization ($p=1, q=2$) for $\alpha = \beta = 0.065$}
        \label{fig:fandisk_iso_l1_blue}
    \end{subfigure}\hfill
    \begin{subfigure}[c]{0.47\textwidth}
        \includegraphics[height=4cm, trim=4cm 9cm 4cm 9cm, clip]{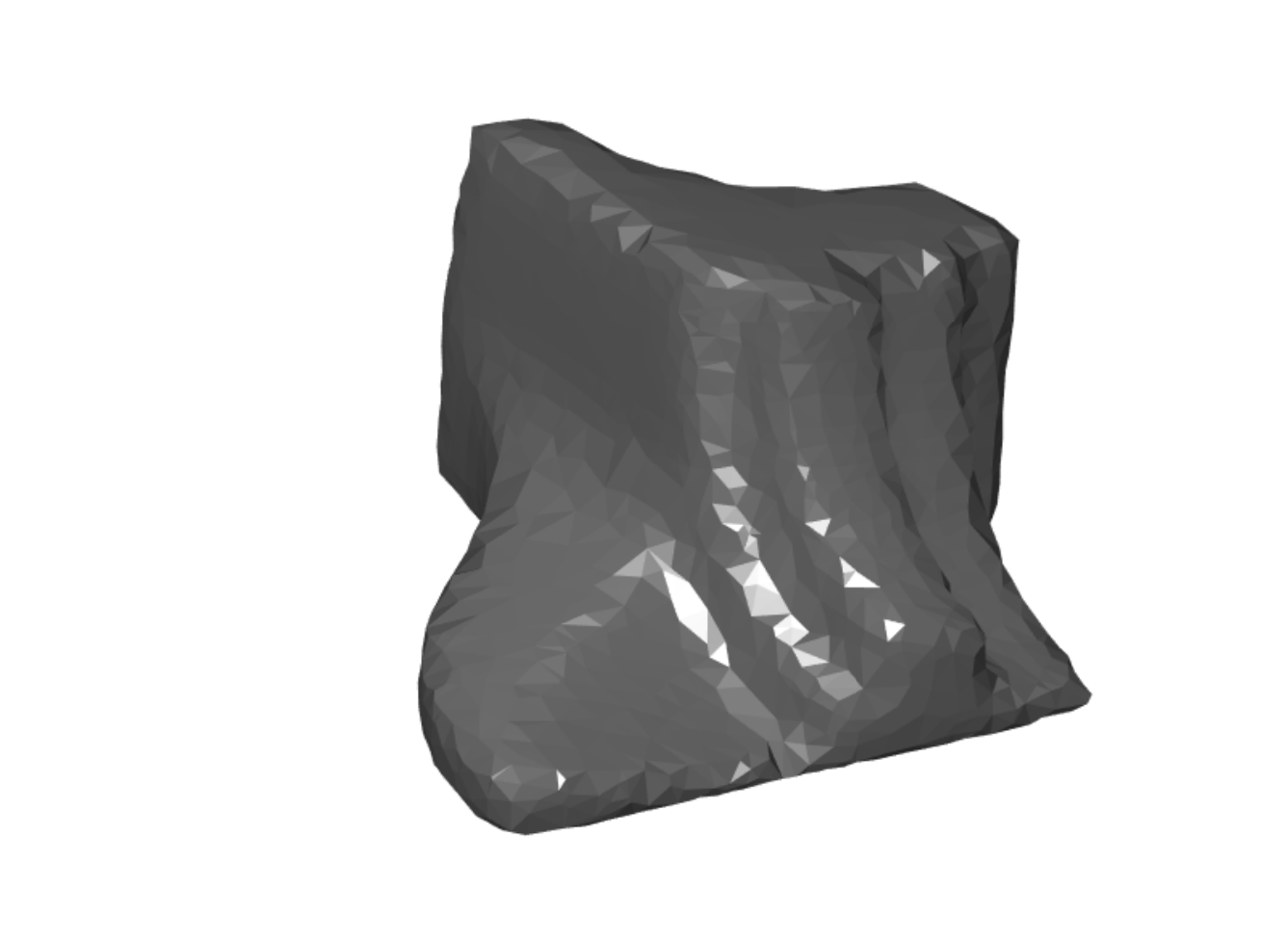}
        \subcaption{Surface rendering of the sparsified point cloud using anisotropic $\ell_1$ regularization ($p=1, q=2$) for $\alpha = \beta = 0.065$}
        \label{fig:fandisk_iso_l1}
    \end{subfigure}
    \\[0.4cm]
    \begin{subfigure}[c]{0.47\textwidth}
        \includegraphics[height=4cm, trim=4cm 9cm 4cm 9cm, clip]{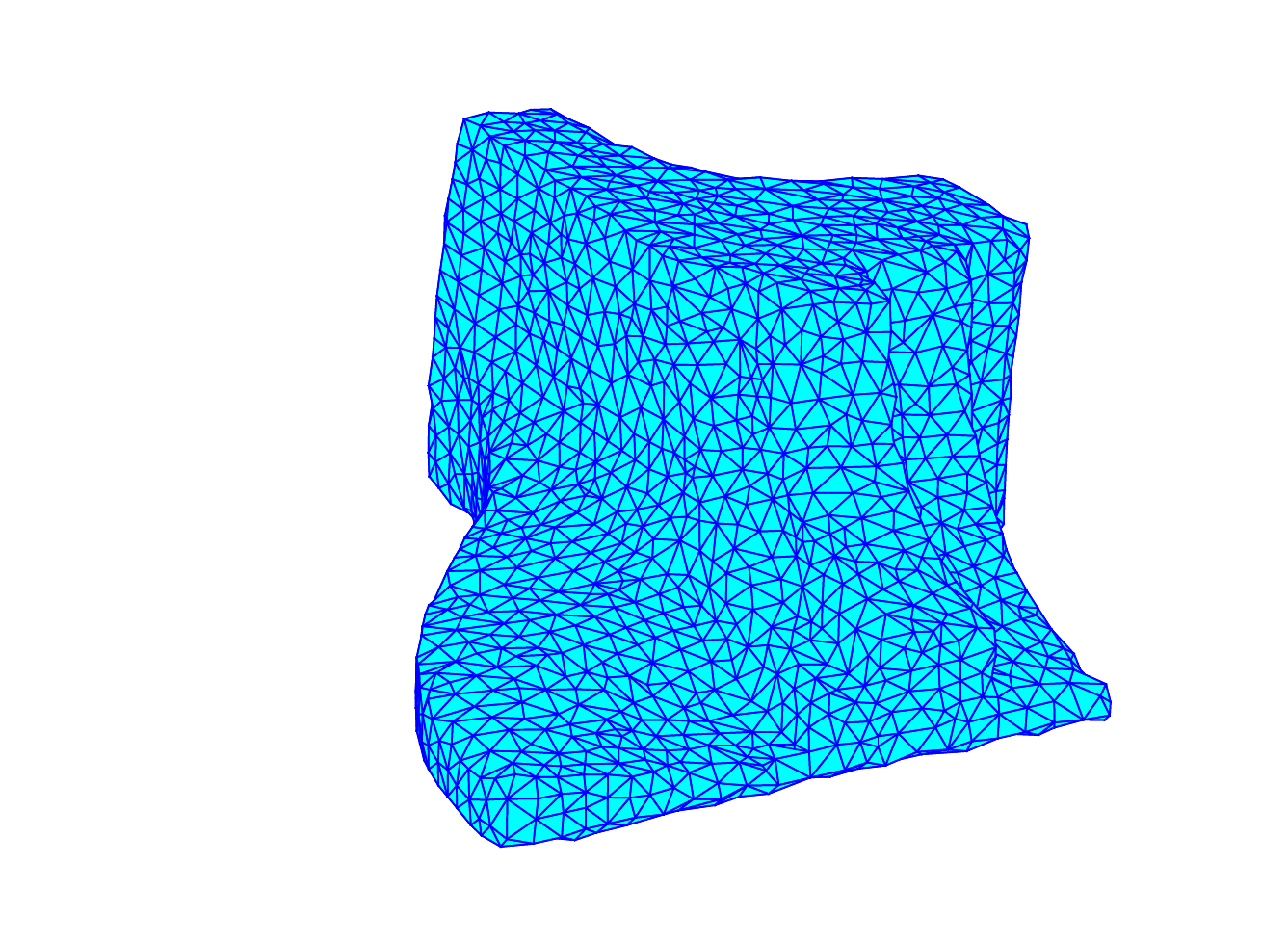}
        \subcaption{Mesh triangulation of the sparsified point cloud using the proposed $\ell_0$ regularization for $\alpha =  \beta = 0.05$}
        \label{fig:fandisk_iso_l0_blue}
    \end{subfigure}\hfill
    \begin{subfigure}[c]{0.47\textwidth}
        \includegraphics[height=4cm, trim=4cm 9cm 4cm 9cm, clip]{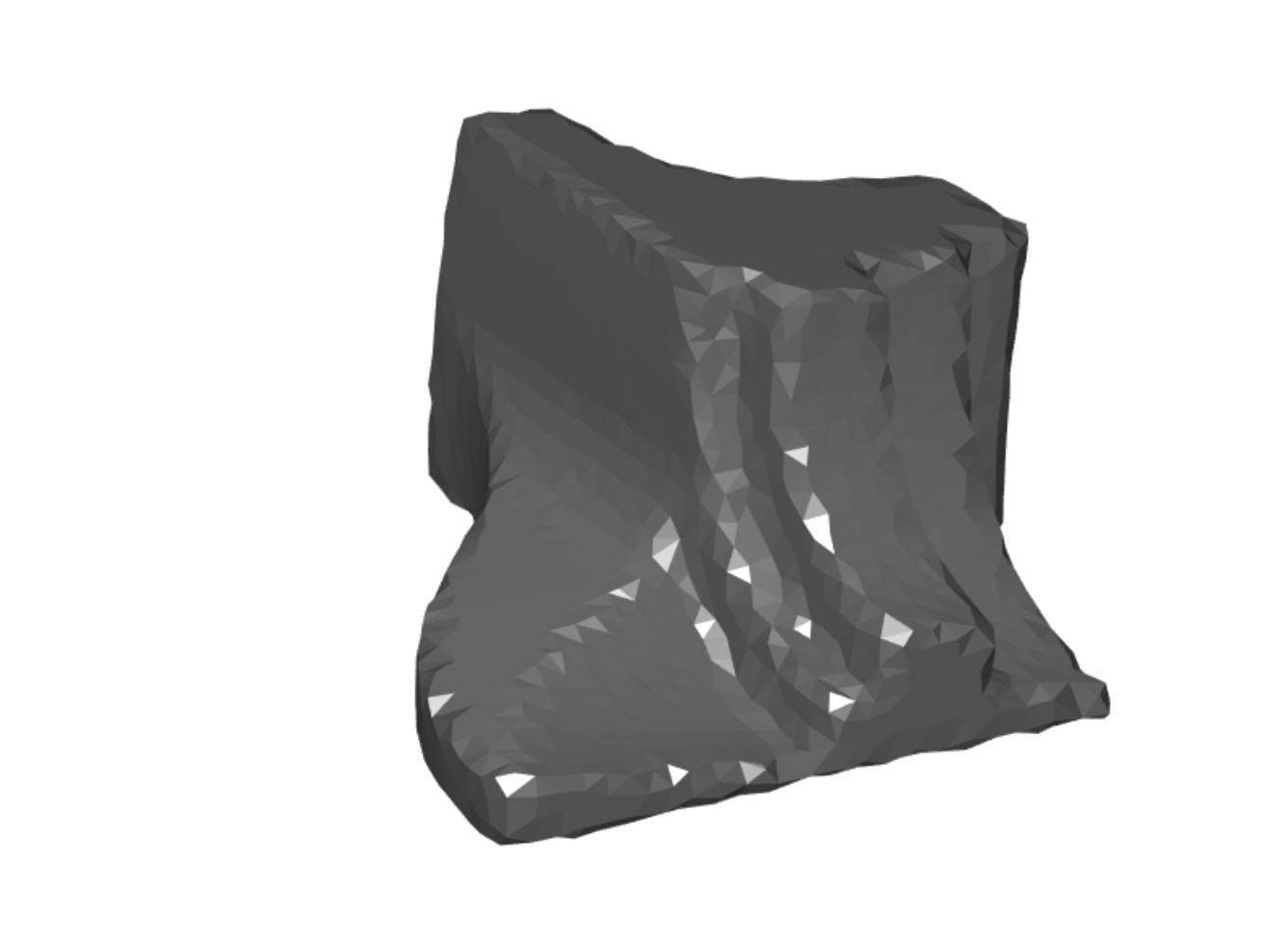}
        \subcaption{Surface rendering of the sparsified point cloud using the proposed $\ell_0$ regularization for $\alpha = \beta = 0.05$}
        \label{fig:fandisk_iso_l0}
    \end{subfigure}
    \caption{Comparison of point cloud sparsification results of the Fandisk model for anisotropic/isotropic $\ell_1$ and the proposed $\ell_0$ regularization. The regularization parameter $\alpha$ is chosen such that all compressed point clouds consist only of $17\%$ of the original point cloud.} 
    \label{fig:Fandisk} 
\end{figure}
}

\subsubsection{Different levels of sparsification using weighted $\ell_0$ regularization}
When looking at the proposed scheme in Algorithm \ref{alg:proposed_schemel0} we can observe that the partitioning problem only depends on the regularization parameter $\alpha$. The solution of the reduced problem is independent on the regularization and corresponds to the mean value of the data in each subset of the partition. Thus, $\alpha$ can be interpreted as a control parameter for the expected level-of-detail and thus of the resulting number of points as we demonstrate in Figure \ref{fig:BunnyApproximation} and Table \ref{t:reg_comparison}. Due to the fact that this approach leads to a sparsification result that is close to the original point cloud, there is no volume shrinkage effect and hence no need for an explicit debiasing step as discussed in Section \ref{ss:debiasing} below.
\begin{table}
	\begin{tabular}{|C{0.1\textwidth}|
			R{0.07\textwidth}L{0.075\textwidth}|
			R{0.07\textwidth}L{0.075\textwidth}|
			R{0.07\textwidth}L{0.075\textwidth}|
			R{0.07\textwidth}L{0.075\textwidth}|}
		\hline
		Data set & \multicolumn{2}{c|}{$\alpha = 0.1$} & \multicolumn{2}{c|}{$\alpha = 0.5$} & \multicolumn{2}{c|}{$\alpha = 1$} & \multicolumn{2}{c|}{$\alpha = 5$} \\\hline
		Bunny  & $35947$ & $(100\%)$& $22651$ & $(63\%)$ &$8034$ & $(22.3\%)$ & $1473$ & $(4\%)$    \\\hline
		Buddha & $543524$ & ($100\%)$ & $213901$ & $(39.4\%)$&$85719$ & $(15.8\%)$& $24555$ & $(4.5\%)$ \\\hline
		Dragon  & $435545$ & $(100\%)$ & $177448$ & $(40.7\%)$&$71040$ & $(16.3\%)$& $19844$ &  $(4.5\%)$ \\
		\hline
	\end{tabular}
	\caption{Comparison of sparsification rates of different regularization parameter selection for a successive graph cut approach with \eqref{eq:partitionproblem}. It shows the number of leftover points and the overall percentage. \label{t:reg_comparison}}
\end{table}
\begin{figure}[htb]
\begin{subfigure}[c]{0.47\textwidth}
	\includegraphics[height=4cm]{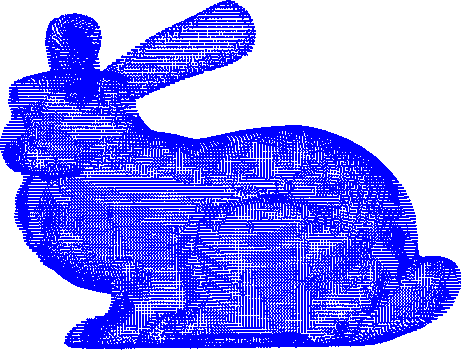}
	\subcaption{Full point cloud data of the \textit{Bunny} model\\($35,947$ points)}
\end{subfigure}\hfill
\begin{subfigure}[c]{0.47\textwidth}
	\includegraphics[trim={1.2cm 1.2cm 1.2cm 1.2cm},clip,height=4cm]{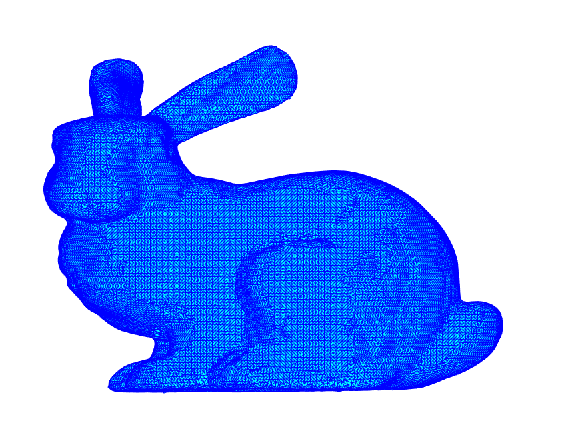}
	\subcaption{Triangulation of full point cloud data of the \textit{Bunny} model}
\end{subfigure}\\[0.4cm]
\begin{subfigure}[c]{0.47\textwidth}
	\includegraphics[height=4cm]{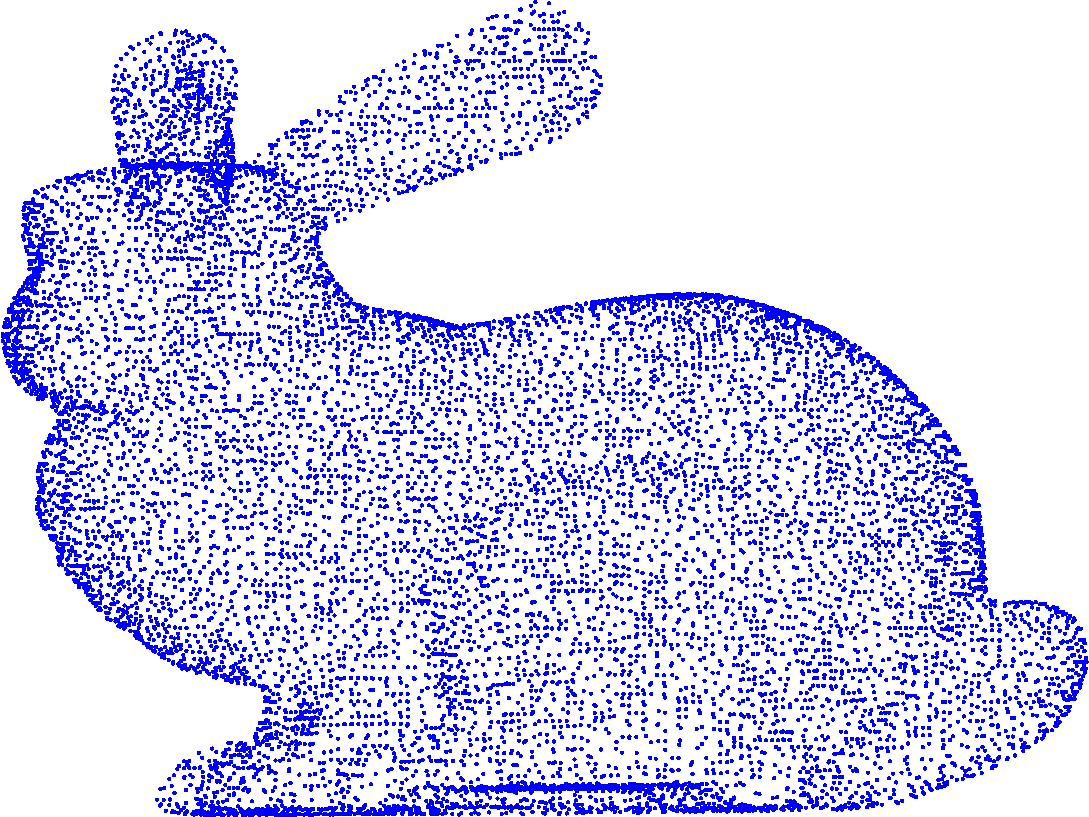}
	\subcaption{Point cloud sparsification for $\alpha=0.3$\\($12,105$ points)}
\end{subfigure}\hfill
\begin{subfigure}[c]{0.47\textwidth}
 	\includegraphics[trim={2.6cm 2.6cm 2.6cm 2.6cm},clip,height=4cm]{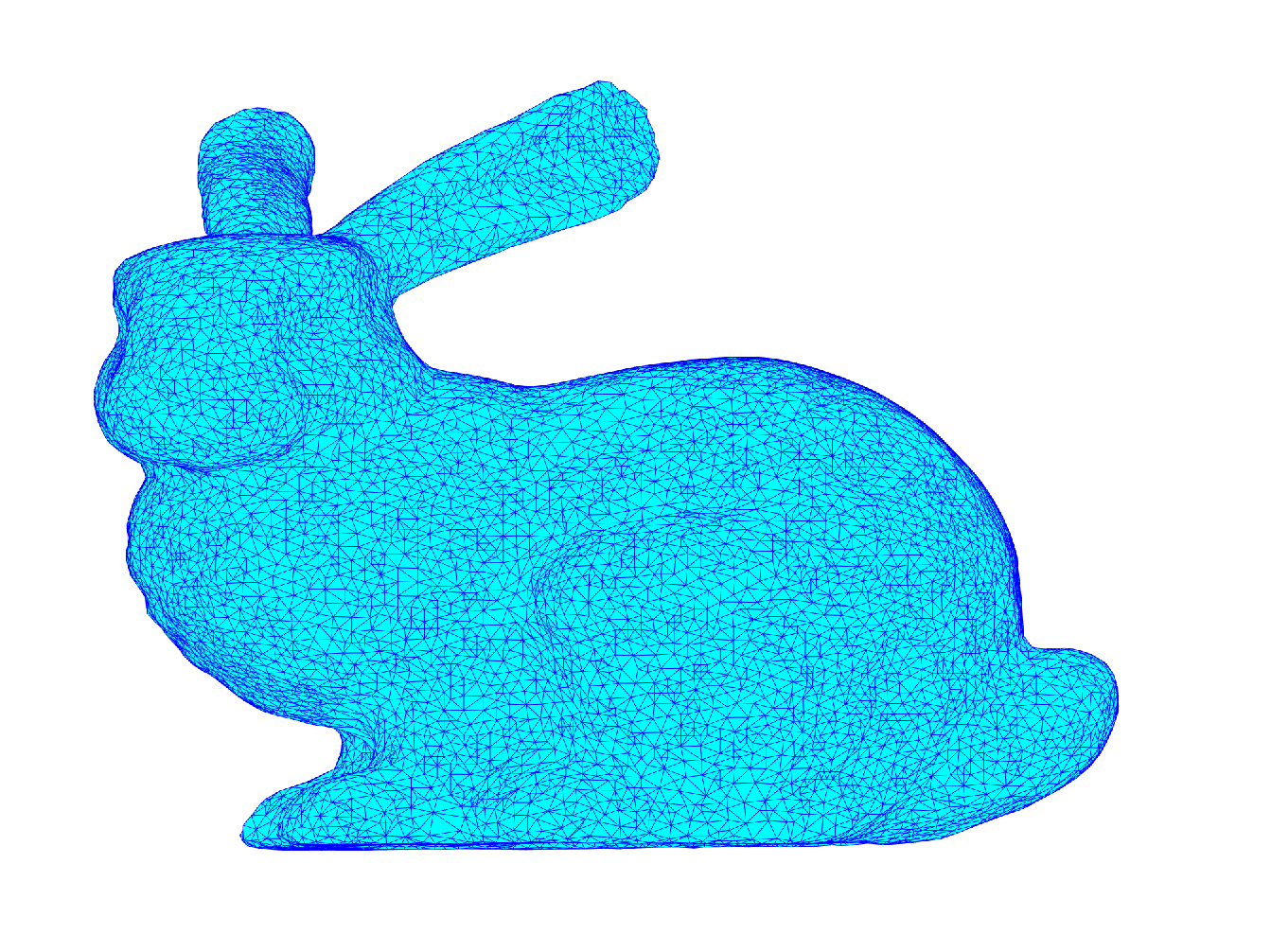}
	\subcaption{Triangulation of sparsified point cloud data for $\alpha=0.3$}
\end{subfigure}\\[0.4cm]
\begin{subfigure}[c]{0.47\textwidth}
	\includegraphics[height=4cm]{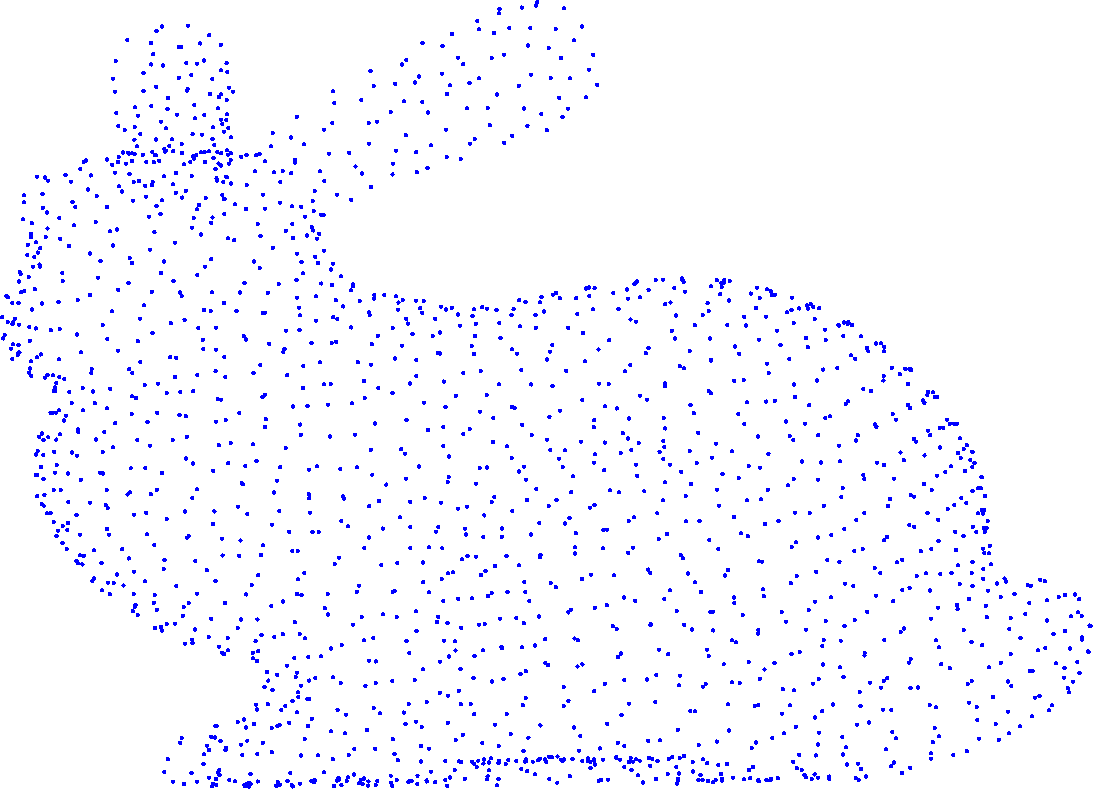}
	\subcaption{Point cloud sparsification for $\alpha=1$\\($1,912$ points)}
\end{subfigure}\hfill
\begin{subfigure}[c]{0.47\textwidth}
	\includegraphics[height=4cm]{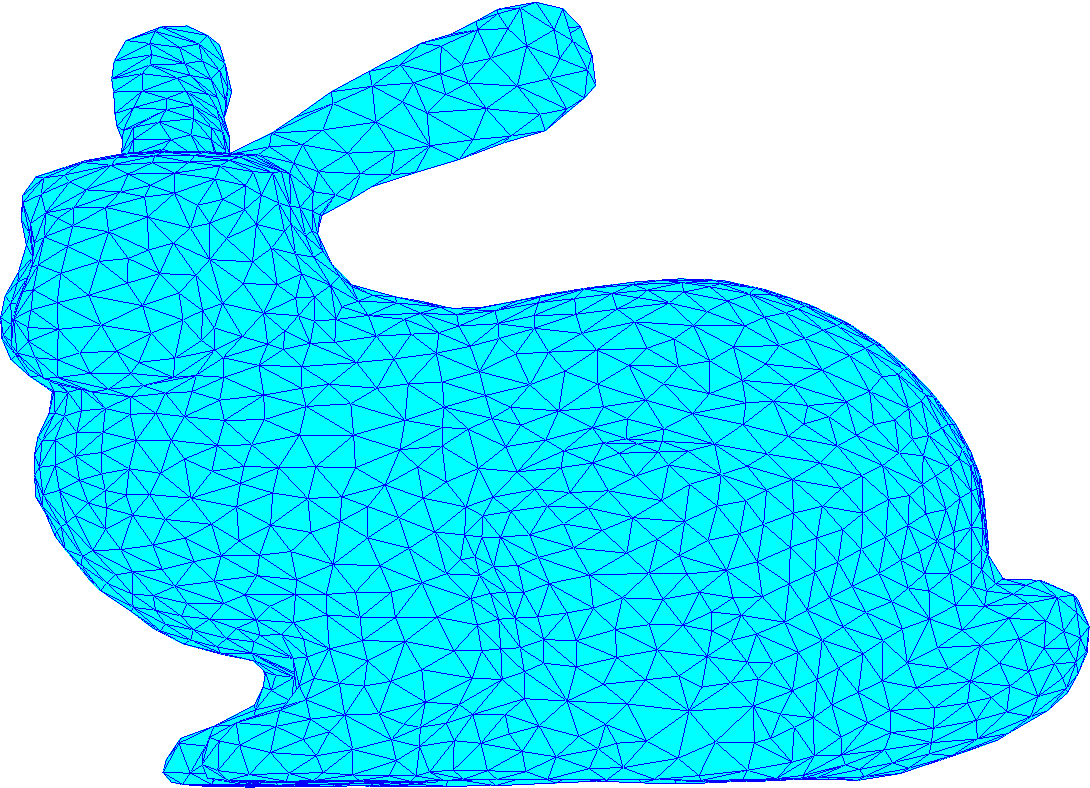}
	\subcaption{Triangulation of sparsified point cloud data for $\alpha=1$}
\end{subfigure}\\[0.4cm]
\begin{subfigure}[c]{0.47\textwidth}
	\includegraphics[height=4cm]{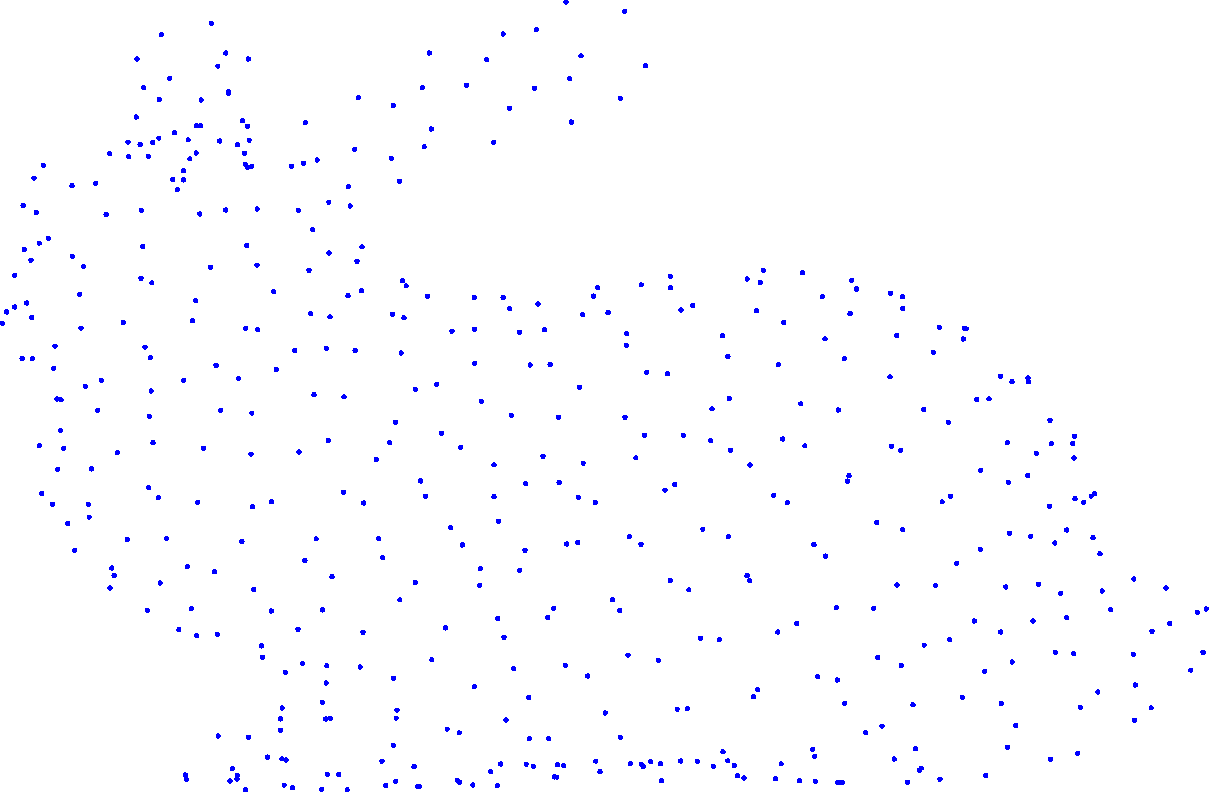}
	\subcaption{Point cloud sparsification for $\alpha=3.5$\\($498$ points)}
\end{subfigure}\hfill
\begin{subfigure}[c]{0.47\textwidth}
	\includegraphics[height=4cm]{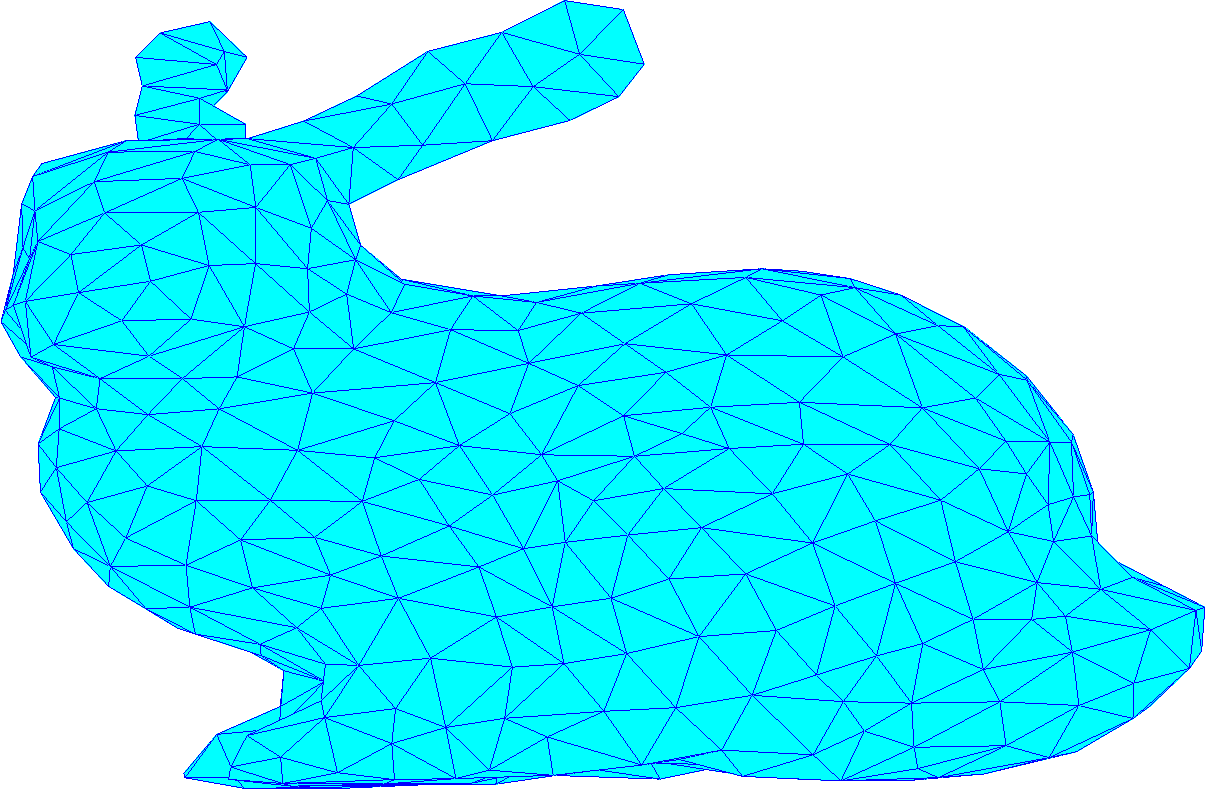}
	\subcaption{Triangulation of sparsified point cloud data for $\alpha=3.5$}
\end{subfigure}
\caption{Comparison of point cloud sparsification results using \dt{the proposed weighted $\ell_0$ regularization} with different values of $\alpha$ in \eqref{eq:partitionproblem}.} %
\label{fig:BunnyApproximation}       %
\end{figure}

\subsection{Point cloud sparsification in the presence of geometric noise}
\label{ss:noisy_data}
In contrast to the previous experiments in which we assumed the given point cloud data to be unperturbed, we focus in the following on data that is prone to geometric noise. In particular, we aim to study the behaviour of the proposed minimization scheme in Algorithm \ref{alg:proposed_scheme} when the given data is perturbed, which occurs in real world applications for cheap scanning hardware or far distances to the object-of-interest. We added a small noise perturbation to every point of the original point cloud following a Gaussian random distribution with mean $\mu=0$ and standard deviation $\sigma=0.003$.

In the left column of Figure \ref{fig:NoisyBunny} we show different point clouds for the \textit{Bunny} data set in a front view, while in the right column we changed the view angle by $90$ degrees to gain a side view of the model. In Figure \ref{fig:noisy_bunny_a}-\subref{fig:noisy_bunny_b} we illustrate the noisy point cloud to be sparsified. The data appears very fuzzy and there are many outliers, which make the task of point cloud sparsification very challenging. In Figure \ref{fig:noisy_bunny_c}-\subref{fig:noisy_bunny_d} one can observe the result of $3$ iterations of the octree approximation scheme discussed in Section \ref{ss:results_octree} above. As can be observed the resulting point cloud is sparse, but yet contains many noise artifacts and outliers, which makes it difficult to recognize the original surface of the model. In Figure \ref{fig:noisy_bunny_e}-\subref{fig:noisy_bunny_f} we demonstrate the result of the proposed minimization scheme in Algorithm \ref{alg:proposed_schemel0} for the weighted $\ell_0$ regularization and using isotropic cuts with a regularization parameter of $\alpha = 3$. As can be seen the distribution of points in the resulting point cloud is relatively sparse compared to the original data. Furthermore, the distribution of points appears much more uniform as compared to the octree approximation scheme in the previous experiment. Still, the impact of noise leads to perturbation artifacts and outliers when the minimization of the reduced problem \eqref{eq:reducedproblem} is skipped. This is not surprising as the reduced problem in the proposed minimization scheme is responsible for denoising the intermediate results of the partitioning scheme. Finally, we present the results of using weighted $\ell_0$ regularization for solving the partition problem and \dt{isotropic} $\ell_2$ regularization for the reduced problem in Figure \ref{fig:noisy_bunny_g}-\subref{fig:noisy_bunny_h}. We use the parameter settings $p=q=2, \alpha=3$ and the regularization parameter $\beta=40$. As can be observed the resulting point cloud is sparse and uniform, while the impact of noise is effectively suppressed. The shape of the original \textit{Bunny} model is well-reconstructed from the noisy input data.
\dt{This shows that there exists data for which it makes sense not to only use the proposed weighted $\ell_0$ regularization, but to incorporate a-priori knowledge about the expected solution in terms of the right regularization term.}
Note that we are able to denoise the raw point cloud data without the need of a mesh triangulation, which makes this approach usable in a wider range of applications.
\begin{figure}
\begin{subfigure}[c]{0.47\textwidth}
	\includegraphics[height=4cm]{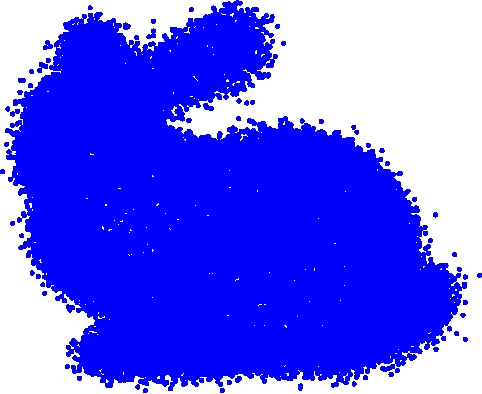}
	\subcaption{Noisy point cloud of \textit{Bunny} model (front view)}
	\label{fig:noisy_bunny_a}
\end{subfigure} \hfill
\begin{subfigure}[c]{0.47\textwidth}
	\includegraphics[height=4cm]{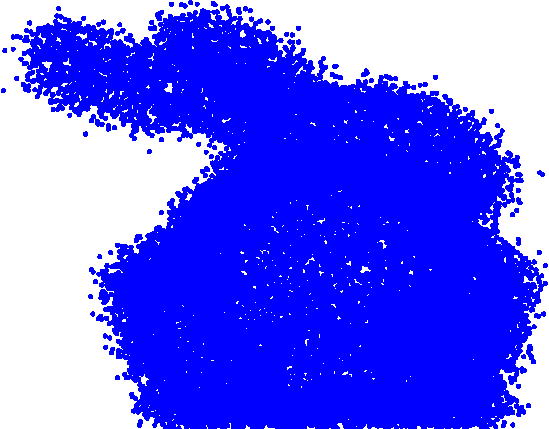}
	\subcaption{Noisy point cloud of \textit{Bunny} model (side view)}
	\label{fig:noisy_bunny_b}
\end{subfigure}\\[0.4cm]
\begin{subfigure}[c]{0.47\textwidth}
	\includegraphics[height=4cm]{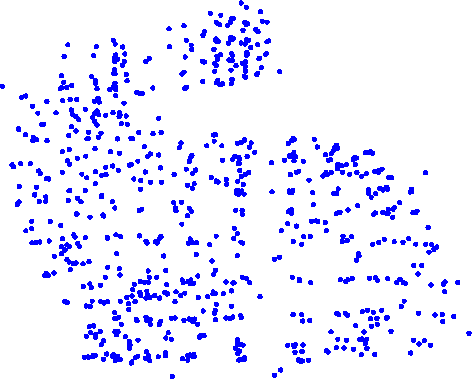}
	\subcaption{Point cloud using octree approximation (front view)}
	\label{fig:noisy_bunny_c}
\end{subfigure}\hfill
\begin{subfigure}[c]{0.47\textwidth}
	\includegraphics[height=4cm]{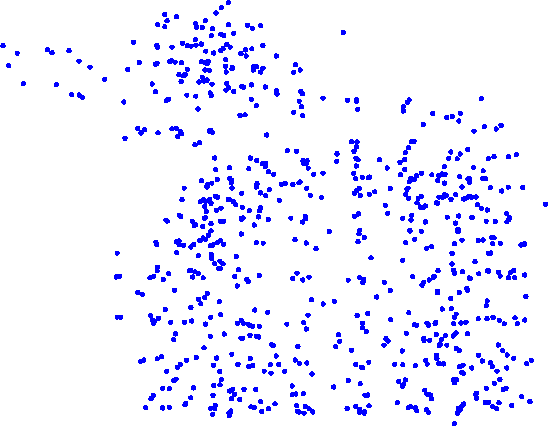}
	\subcaption{Point cloud using octree approximation (side view)}
	\label{fig:noisy_bunny_d}
\end{subfigure}\\[0.4cm]
\begin{subfigure}[c]{0.47\textwidth}
	\includegraphics[height=4cm]{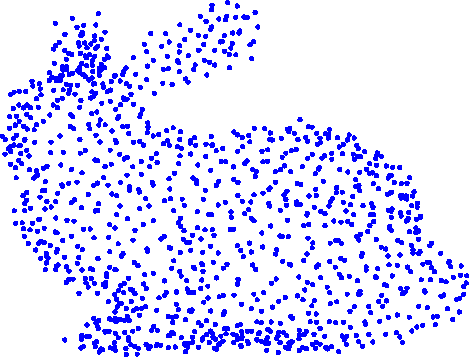}
	\subcaption{Point cloud using weighted $\ell_0$ \dt{regularization} (front view)}
	\label{fig:noisy_bunny_e}
\end{subfigure} \hfill
\begin{subfigure}[c]{0.47\textwidth}
	\includegraphics[height=4cm]{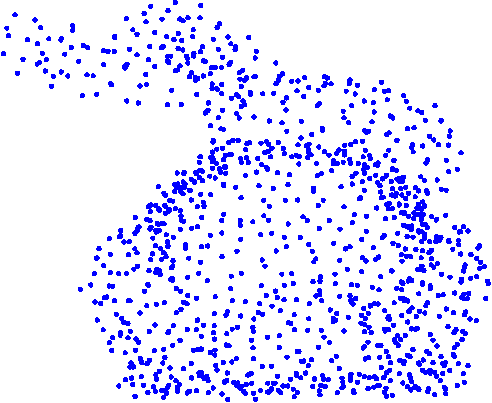}
	\subcaption{Point cloud using weighted $\ell_0$ \dt{regularization} (side view)}
	\label{fig:noisy_bunny_f}
\end{subfigure}\\[0.4cm]
\begin{subfigure}[c]{0.47\textwidth}
	\includegraphics[height=4cm]{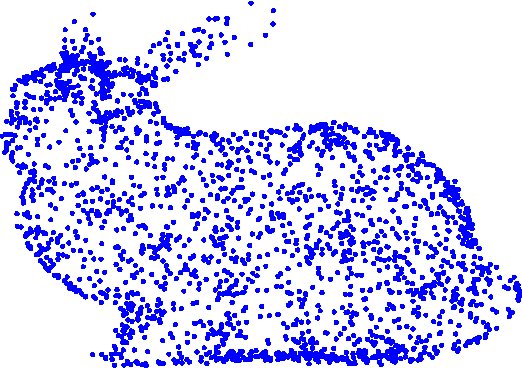}
	\subcaption{Point cloud using \dt{isotropic $\ell_2$ regularization} (front view)}
	\label{fig:noisy_bunny_g}
\end{subfigure} \hfill
\begin{subfigure}[c]{0.47\textwidth}
	\includegraphics[height=4cm]{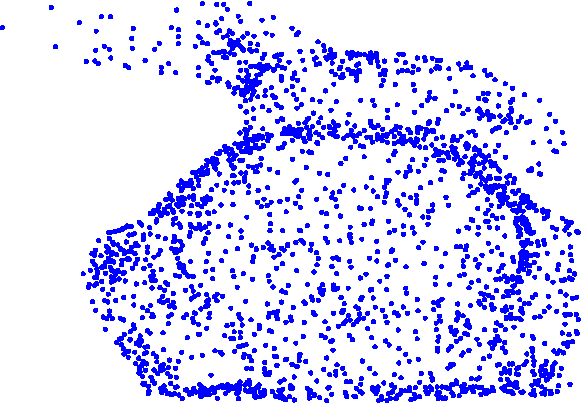}
	\subcaption{Point cloud using \dt{isotropic $\ell_2$ regularization} (side view)}
	\label{fig:noisy_bunny_h}
\end{subfigure}
\caption{Comparison of different point cloud sparsification methods for a noisy point cloud of the \textit{Bunny} data set.}
\label{fig:NoisyBunny} 
\end{figure}

\subsection{Debiasing}
\label{ss:debiasing}
One observation we made during our numerical experiments is that there is a loss of volume in the resulting sparse 3D point clouds when compared to the original point cloud \dt{in particular when using the anisotropic} $\ell_1$ regularization. This loss of volume is directly influenced by the choice of the regularization parameter $\beta$ in the reduced problem \eqref{eq:reducedproblem} of the proposed minimization scheme. In particular, the higher we choose the regularization parameter $\beta$ the more the resulting sparse point cloud shrinks. This effect is well-known in the image processing community as 'loss of contrast' or 'bias' and is typically associated with the application of total variation regularization. 

In order to overcome this problem we propose to perform a debiasing step as post-processing once the proposed minimization scheme in Algorithm \ref{alg:proposed_scheme} is converged to a minimizer. Note that the reduction of bias in variational regularization is a challenging task as can be seen in \cite{debias}. However, in our setting a debiasing step can be performed rather simple as we can adjust the value of whole vertex subsets $A_i \subset V$ by adjusting the optimal piece-wise constant functions on these subsets with respect to the original (possibly noisy) data. It turns out that the optimal piece-wise constant approximation on each subset is the mean value of the data being assigned to this subset by the partition $\Pi$. The debiasing step can easily be implemented by performing one final denoising step in \eqref{eq:reducedproblem} and setting the regularization parameter $\beta=0$ as proposed in \cite{debias}. In this case the minimizer $f_\Pi$ is adjusted according to the original data and thus correcting for the loss-of-volume effect. In our case this is a very cheap operation in terms of computational effort as only the mean value of the $k$ subsets $A_i \subset V$ induced by the partition $\Pi$ have to be computed.

In Figure \ref{fig:NoisyCube} we demonstrate the effect of the proposed debiasing step on a two-dimensional noisy point cloud. In Figure \ref{fig:debiasing_c} one can see the result of point cloud sparsification with the proposed minimization scheme for $p=q=1$ and a regularization parameter of $\beta = 10$. As can be seen the noise is effectively suppressed in the sparse point cloud. However, due to the strong regularization there is a significant loss-of-volume compared to the original data in Figure \ref{fig:debiasing_a}. After performing a subsequent debiasing step as discussed above one can observe the improved result in Figure \ref{fig:debiasing_d} in which the original dimensions are restored.
\begin{figure}
\begin{subfigure}[c]{0.47\textwidth}
	\includegraphics[height=4cm]{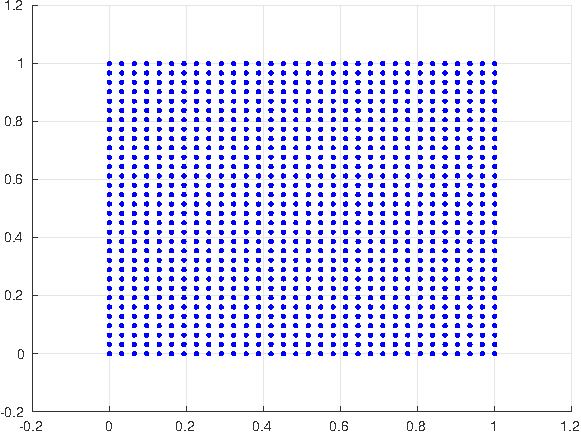}
	\subcaption{Unperturbed point cloud of 3D cube (front view)}
	\label{fig:debiasing_a}
\end{subfigure} \hfill
\begin{subfigure}[c]{0.47\textwidth}
	\includegraphics[height=4cm]{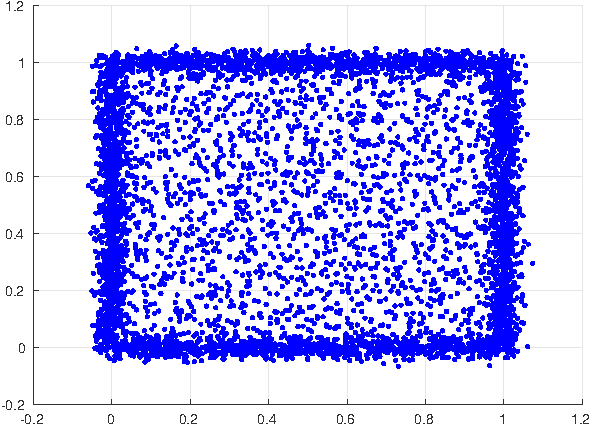}
	\subcaption{Noisy point cloud of 3D cube with $\sigma = 0.02$ (front view)}
\end{subfigure}\\[0.4cm]
\begin{subfigure}[c]{0.47\textwidth}
	\includegraphics[height=4cm]{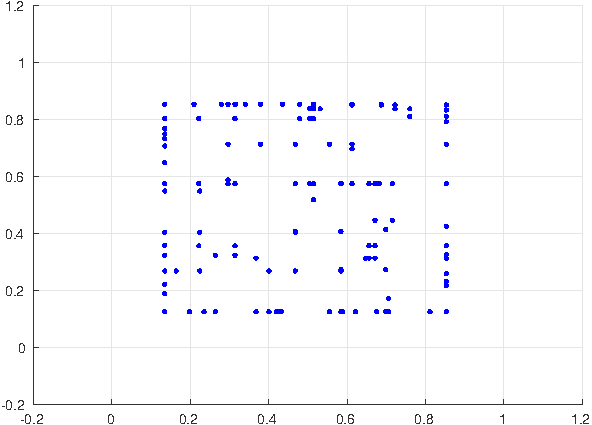}
	\subcaption{Result of point cloud sparsification without debiasing (front view)}
	\label{fig:debiasing_c}
\end{subfigure}\hfill
\begin{subfigure}[c]{0.47\textwidth}
	\includegraphics[height=4cm]{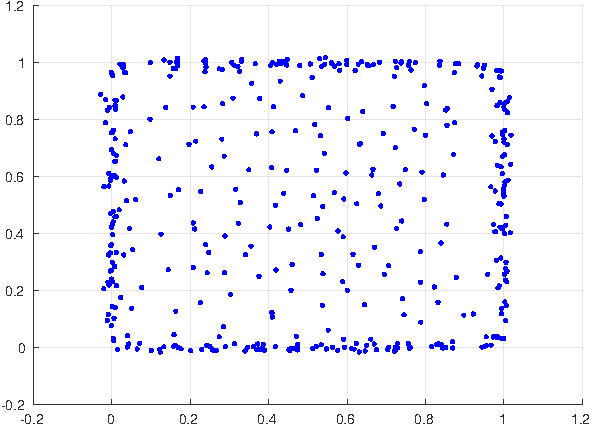}
	\subcaption{Result of point cloud sparsification with debiasing (front view)}
	\label{fig:debiasing_d}
\end{subfigure}\\[0.4cm]
\caption{Visualization of the impact of a subsequent debiasing step on the results of point cloud sparsification. As can be observed the original volume is restored by this post-processing step.}
\label{fig:NoisyCube}
\end{figure}

\section{Discussion}
\label{s:discussion}
In this paper we have proposed an efficient minimization strategy on finite weighted graphs for the task of point cloud sparsification, which is inspired by the recently proposed Cut Pursuit algorithm. We compared the numerical results of the proposed coarse-to-fine scheme to a fine-to-coarse strategy that has already been used for this application in the literature. As could be observed our method does not only preserve details of the underlying surface topology much better when using \dt{the proposed} weighted $\ell_0$ regularization, but also has a significantly lower computational effort. This renders variational methods for point cloud compression to be a real alternative to traditional methods, such as random sampling or octree compression.

As we discussed in this work, by deviating from the proposed Cut Pursuit scheme we gain additional flexibility for choosing different regularization functionals and hence controlling the appearance of the resulting sparse point clouds. On the other hand, we are currently not able to give strict convergence proofs for this method as we decoupled both minimization problems in the alternating scheme. Although, we expect the difference between a minimizer of the original variational problem and the approximation computed by our scheme to be relatively small, we aim to further analyze this discrepancy in future works. %

\dt{One aim for future work is to derive a quantitative measure for comparing a given 3D point cloud with the presented results of point cloud sparsification. This makes it possible to compare the effectiveness of preserving important geometrical features between the proposed regularization techniques.
So far we have tried two different schemes to get a quantitative measure of how far a compressed point cloud deviates from the original point cloud. The first idea consists of having a surface representation of the 3D point cloud based on level set functions. Measuring the distance between two level set functions is easy to perform, in particular if these are given by signed distance functions. However, we encountered problems for sparse point clouds in which the points were relatively far away from each other. If one does not impose a strong regularity on the level set segmentation method employed there may appear holes in the level set surface and the topology of the data is distorted. On the other hand, if one chooses a high regularization parameter for the level set method all sharp geometrical features can get lost. It turns out to be difficult to use level set methods to quantitatively measure the distance between two point clouds as one would have to optimize the regularity parameter by many trials for each pair of given point clouds. 
The second idea consists of measuring the Hausdorff distance between two mesh triangulations representing the surface sampled by the point cloud. Although there exist methods to measure this distance quantitatively between two meshes, we encountered similar problems as in the case of level set methods. When dealing with sparse point clouds many triangulation methods yield mesh representations with irregular triangle approximations or even holes and thus induce mistakes in the computation of the Hausdorff distance between the surfaces.
We currently work on an alternative way of measuring the distance between two point clouds directly using registration approaches, e.g., the iterative closest point (ICP) algorithm.
This will enable a comparison with state-of-the-art methods for point cloud sparsification, e.g. the optimal transport based scheme in \cite{OptimalTransportPointCloud}.

Furthermore, we} plan to analyze the effect of the graph construction and the choice of the weight function $w$ on the results of point cloud sparsification and plan to incorporate nonlocal relationships within the 3D point cloud data.
So far we did not \dt{exploit} any surface normal information, which could easily be estimated from performing a local principal component analysis on the point cloud. Using these normal information could help in reconstructing sparse point clouds without the loss-of-volume effect described during our numerical experiments. It also might \dt{further} improve the preservation of sharp features such as edges and corners.

\section*{Acknowledgment}
 This work was supported by the Bundesministerium f\"ur Bildung und Forschung under the project id 05M16PMB (MED4D) and the European Union’s Horizon 2020 research and innovation programme under the Marie Skłodowska-Curie grant agreement No. 777826 (NoMADS). 
 
The authors would like to thank Dana Sarah Hetland for realizing the illustrations in Figure \ref{fig:partition} and Figure \ref{fig:reduced_graph}, Julian Rasch for fruitful discussions on debiasing methods and primal-dual optimization, and Loic Landrieu for his inspiring talk during the NoMADS workshop at Politecnico di Milano in June 2018 and interesting discussions on weighted $\ell_0$ regularization.
\nocite{*}
\printbibliography

\appendix

\section{Appendix: The graph $p$-$q$-Laplace operator}\label{app:pqlap}
In this section we derive the weighted graph $p$-$q$-Laplace operator defined as in \eqref{eq:pq-laplace}. All computations done below are based on the definition of the divergence as the adjoint \eqref{eq:adjoint} of the gradient operator $\nabla_w$ in \eqref{eq:gradient}. Additionally, it is assumed that we have an undirected graph with $w(u,v) = w(v,u)$, and thus $\partial_uf = - \partial_vf$ as described in Section \ref{s:graphs}.
\begin{align*}
\Delta_{w,p;q} f(u) &=\frac{1}{2}\operatorname{div}_w\Big(\big\Vert \nabla_w f\big\Vert_q^{p-q} \left(\nabla_w f_j \big\vert \nabla_wf_j\big\vert^{q-2} \right)_{j=1}^d\Big)\\
&\stackrel{\eqref{eq:adjoint}}{=} -\frac{1}{2} \sum_{v\sim u} \sqrt{w(u,v)} \big\Vert \nabla_wf(u,v)\big\Vert_q^{p-q} \\
&\hspace{1cm}\Bigg(\nabla_w f(v,u)_j \big\vert \nabla_wf(v,u)_j\big\vert^{q-2} - \nabla_w f(u,v)_j \big\vert \nabla_wf(u,v)_j\big\vert^{q-2} \Bigg)_{j=1}^d\\
&= \sum_{v\sim u} \sqrt{w(u,v)} \big\Vert \nabla_wf(u,v)\big\Vert_q^{p-q} \Bigg(\nabla_w f(u,v)_j \big\vert \nabla_wf(u,v)_j\big\vert^{q-2} \Bigg)_{j=1}^d\\
&=\sum_{v\sim u} \sqrt{w(u,v)} \big\Vert w(u,v)^\frac{1}{2}(f(v)-f(u))\big\Vert_q^{p-q} \\
&\hspace{1cm}\Bigg(w(u,v)^\frac{1}{2} (f(v)_j -f(u)_j) \big\vert w(u,v)^\frac{1}{2} (f(v)_j -f(u)_j) \big\vert^{q-2} \Bigg)_{j=1}^d\\
&=\sum_{v\sim u} w(u,v)^\frac{p}{2} \big\Vert f(v)-f(u)\big\Vert_q^{p-q} \Bigg( (f(v)_j -f(u)_j) \big\vert  f(v)_j -f(u)_j \big\vert^{q-2} \Bigg)_{j=1}^d
\end{align*}

\fg{
\section{Appendix: Cut Pursuit derivation for $p-q$-norm}\label{app:deriviationCutPursuit}
In this section we want to derive an approach to compute the solution of the variational problem 
\begin{align}\label{eq:MinimizationProblemLNull}
    \argmin_{f\in \Hil(V)} \Big\{ J(f) = D(f,g) + \frac{\alpha}{2p} \left\Vert\nabla_w f \right\Vert_{p;q}^p\Big\}
\end{align}
by using a binary splitting algorithm. We will denote $R(f) = \frac{1}{p}\left\Vert\nabla_w f \right\Vert_{p;q}^p$. As we have discussed before the solutions of a minimization problem \eqref{eq:MinimizationProblemLNull} with a total variation regularizer consists of piecewise constant parts. 
Thus, an algorithm which starts with a very coarse partitioning and refines on the run until it has the same partition as the solution of \eqref{eq:MinimizationProblemLNull} is very practical. \par
To do so, we start by saying that we have a current partition $\Pi$ of $V$ with sets $A_i, A_j\in \Pi$ where $A_i \cap A_j = \emptyset$ and $V = \bigcup_{A_i \in \Pi} A_i$. We aim to find a set $B \subset V$ to refine $\Pi$ by computing $A\cap B$ and $A\cap B^c$ for every $A\in \Pi$. 
Let $c = (c_A)_{A\in \Pi}\in \mathcal{H}(\Pi) \cong \mathbb{R}^{\vert\Pi\vert \times d}$ be a vertex function with a constant vector $c_A \in \mathbb{R}^d$ for every partition $A\in \Pi$. Additionally, we will use the vertex function $$1_A = \begin{cases}
    1, \text{if } u\in A\\
    0, \text{else}
\end{cases} $$ which can be interpreted as a $N\times 1$ vector. Since, $c_A$ can be written as $1\times d$ vector we can compute the matrix $1_A c_A \in \mathbb{R}^{N\times d}$ for every $A\in \Pi$. 
Then we can find a vertex function $f\in \Hil(V)$ given by $$f = \sum_{A\in\Pi} 1_A c_A$$ being piecewise constant on the sets $A\in \Pi$. 
Hence, we can solve a minimization problem to find the best constants for a given partition $\Pi$ by solving 
\begin{align}\label{eq:reducedProblem}
    \argmin_{c\in \Hil(\Pi)} \ J(f) =  J\big(\sum_{A\in\Pi} 1_A c_A\big)
\end{align}

Let $c\in \Hil(\Pi)$ be a solution to \eqref{eq:reducedProblem} and $f = \sum_{A\in\Pi} 1_A c_A$ be the corresponding function on $V$. 
Assume now we have some binary partition of the set $V$ described by $B\in V$ and a corresponding piecewise constant function
$$\tilde{f}_B = \sum_{A\in\Pi} \left( 1_{A\cap B} d_{A\cap B} + 1_{A\cap B^c} d_{A\cap B^c} \right)$$
with some vectors $d_{A\cap B}, d_{A\cap B^c}\in \mathbb{R}^d$ for any combination of $A\in \Pi$ and $B\subset V$. Note that this function is piecewise constant on the \emph{new} sets $A\cap B \subset A$ and $ A\cap B^c \subset A$ for every $A\in \Pi$. The energy functional then becomes 
\begin{align}
    \begin{split}
        J(\tilde{f}_B) &= D\Big(\sum_{A\in\Pi} \left( 1_{A\cap B} d_{A\cap B} + 1_{A\cap B^c} d_{A\cap B^c} \right), g\Big) \\ & \phantom{=} + \frac{\alpha}{2}R \Big(\sum_{A\in\Pi} \left( 1_{A\cap B} d_{A\cap B} + 1_{A\cap B^c} d_{A\cap B^c} \right)\Big).  
    \end{split}
\end{align}

We now want to find a partitioning described by $B\subset V$ such that the energy of $J(\tilde{f}_B)$ decreases the most compared to the current energy $J(f)$. Thus, we optimize to find the set $B\subset V$ that minimizes 
\begin{align}\label{eq:energydiff}
    \begin{split}
        \argmin_{B\subset V} \ J(\tilde{f}_B) - J(f)\\
        \argmin_{B\subset V} \ D(\tilde{f}_B,g) - D(f,g) + \frac{\alpha}{2} \ \left( R( \tilde{f}_B) - R(f) \right)
    \end{split}
\end{align}
with fixed vectors for $d_{A\cap B}$ and $d_{A\cap B^c}$.\par
We will divide this in two steps and investigate the left-hand difference of \eqref{eq:energydiff} first. For simplicity we will set the values $d_{A\cap B} = c_A + \gamma_B^A$ and $d_{A \cap B^c} = c_A - \gamma_{B^c}^A$ with $\gamma_B^A, \gamma_{B^c}^A\in \mathbb{R}^d_+$ for every combination of $A\in \Pi$ and $B$. Then we can compute the difference - which we will denote as $\mathcal{D}$ from now on - as follows 
\begin{align}\label{eq:defDifference}
    \begin{split}
        \mathcal{D}(\tilde{f}_B, f, g) &=  D(\tilde{f}_B,g) - D(f,g)\\
        &= D(\sum_{A\in\Pi} \left( 1_{A\cap B} d_{A\cap B} + 1_{A\cap B^c} d_{A\cap B^c} \right), g) - D(\sum_{A\in\Pi} 1_A c_A,g)\\
        &= D(\sum_{A\in\Pi} \left(1_A c_A + 1_B \gamma_B - 1_{B^c} \gamma_{B^c}\right),g) - D(\sum_{A\in\Pi} 1_A c_A,g).
    \end{split}
\end{align}

The difference $\mathcal{D}$ of \eqref{eq:defDifference} can be approximated with Taylor expansion of $D$ since $D$ is assumed to be differentiable. Note that $\nabla D(f)\in \mathbb{R}^{N\times d}$ is a $N\times d$ matrix describing the derivative of $D$ for function $f$. For simplicity we will write $D(f)$ for $D(f,g)$ in the following. Therefore, we will evaluate at point $f$ and get an approximation of $D$ at some point $x\in \mathcal{H}(V)$ as 
\begin{align}
    D(x) \approx D(f) + \langle\nabla D(f), x - f\rangle.
\end{align}
Reformulating this and approximating $D$ for point $\tilde{f}_B$ we get the following evaluation 
\begin{align}
    D(\tilde{f}_B) - D(f) &\approx \langle\nabla D(f), \tilde{f}_B - f\rangle. 
\end{align}
We then obtain the approximation 
\begin{align} \label{eq:approxDataTermDiff}
D(\tilde{f}_B,g) - D(f) &\approx  \langle \nabla D(f), \sum_{A\in \Pi}1_{A\cap B} \gamma_B^A - 1_{A\cap B^c} \gamma_{B^c}^A \rangle.
\end{align}
using $$ \tilde{f}_B - f = \sum_{A\in \Pi}1_{A\cap B} \gamma_B^A - 1_{A\cap B^c} \gamma_{B^c}^A. $$

Note again the equivalence of 
 \begin{align}\label{eq:B_relation}
     1_V = 1_B + 1_{B^c} &\Leftrightarrow 1_{B^c} = 1_V - 1_B, \\
     1_A = 1_{A\cap B} + A_{A\cap B^c} & \Leftrightarrow 1_{A\cap B^c} = 1_A - 1_{A\cap B}.
 \end{align}
 Thus, we can rewrite the approximation again as 
 \begin{align}\label{eq:approxDataTerm}
    D(\tilde{f}_B) - D(f) &\approx  \langle \nabla D(f), \sum_{A\in \Pi} 1_{A\cap B} (\gamma_B^A + \gamma_{B^c}^A)\rangle + \kappa_D
 \end{align}
where $\kappa_D =  - \langle \nabla D(f), \sum_{A\in \Pi} 1_A  \gamma_{B^c}^A\rangle$ and finish this computation. From now on we will denote 
$$\vec{\gamma} = \sum_{A\in \Pi} 1_{A\cap B} (\gamma_B^A + \gamma_{B^c}^A).$$\par  

Let us now focus on the right-hand side of \eqref{eq:energydiff} that is given as 
\begin{align}\label{eq:tvdiff}
    \begin{split}
        R(\tilde{f}_B) - R(f) &= \frac{1}{p}\sum_{(u,v) \in E} w(u,v)^\frac{p}{2}\big(\Vert \tilde{f}_B(v) - \tilde{f}_B(u) \Vert_{p;q}^p - \Vert f(v) - f(u)\Vert_{q;p}^p\big) \\
        & = \frac{1}{p}\sum_{(u,v) \in E} w(u,v)^\frac{p}{2} d(u,v)
    \end{split}
\end{align}
with  $d(u,v) = \Vert \tilde{f}_B(v) - \tilde{f}_B(u)\Vert_{p;q}^p - \Vert f(v) - f (u)\Vert_{p;q}^p$. To simplify we will denote $\Vert \cdot \Vert_{p;q}^p$ just as $\Vert \cdot \Vert$.
In this case we have to investigate different edge types that can occur due to the splitting with $B$. \par

First we will study the case for edges $(u,v) \in S^c(f)$ where $R$ is non-differentiable. This is the case if $u, v\in A\in \Pi$. Then the value of $d$ can be computed as 
\begin{align*}
    d(u,v) &= \big\Vert c_A + 1_B(v)\gamma_B^A - 1_{B^c}(v) \gamma_{B^c}^A - \big(c_A + 1_B(u)\gamma_B^A - 1_{B^c}(u) \gamma_{B^c}^A\big)  \big\Vert\\
    & = \big\Vert \big(1_B(v)- 1_B(u)\big)\gamma_B^A + \big(1_{B^c}(u)- 1_{B^c}(v)\big) \gamma_{B^c}^A \big\Vert.
\end{align*} 
It is easy to see that for $u,v \in B$ or $u,v\in B^c$ we get $d(u,v) = 0$. For $u\in B, v\in B^c$ or $v\in B, u\in B^c$ respectively we get 
$$ d(u,v) = \big\Vert \gamma_B^A + \gamma_{B^c}^A \big\Vert.$$
Thus, the sum over these edges is given as 
\begin{align}\label{eq:}
    \begin{split}
        \frac{1}{2p}\sum_{(u,v) \in S^c} w(u,v)^\frac{p}{2} d(u,v) &= \frac{1}{2p} \sum_{A\in \Pi}\sum_{\substack{(u,v)\in E(A,A)\\ u\in B, v\in B^c}} w(u,v)^\frac{p}{2} \big\Vert \gamma_B^A + \gamma_{B^c}^A\big\Vert\\
        &= \frac{2}{2p} \sum_{A\in\Pi} w(A\cap B, A\cap B^c)\big\Vert \gamma_B^A + \gamma_{B^c}^A\big\Vert
    \end{split}
\end{align}

\par 
Now we investigate the edges $(u,v) \in S$ of the differentiable part of $R$. This is equivalent with $(u,v)\in E(A_i, A_j)$ for some $A_i, A_j \in \Pi$ with $i\neq j$ and $(A_i, A_j) \in E_r$. We again want to evaluate the value of $d$. This can be done as by first looking at
\begin{align}
    \begin{split}
        \big\Vert \tilde{f}_B(v) - \tilde{f}_B(u) \big\Vert &= \big\Vert c_{A_j} + 1_B(v)\gamma_B^{A_j} - 1_{B^c}(v) \gamma_{B^c}^{A_j} - \big(c_{A_i} + 1_B(u)\gamma_B^{A_i} - 1_{B^c}(u) \gamma_{B^c}^{A_i}\big)  \big\Vert.\\
    \end{split}
\end{align}
Let us write as a simplification
$$\Gamma_{j,i}(v,u) = 1_B(v)\gamma_B^{A_j} - 1_{B^c}(v) \gamma_{B^c}^{A_j} - 1_B(u)\gamma_B^{A_i} + 1_{B^c}(u) \gamma_{B^c}^{A_i}$$
and compute 
\begin{align}
    \begin{split}
        d(u,v) = \big\Vert c_{A_j} - c_{A_i} + \Gamma_{j,i}(v,u)\big\Vert - \big\Vert c_{A_j} - c_{A_i}  \big\Vert.
    \end{split}
\end{align}
As we stated before this part of the regularizer $R$ is differentiable thus we can again do a Taylor expansion and approximate 
\begin{align}
    \begin{split}
        d(u,v) \approx \langle \nabla \Vert\cdot\Vert (c_{A_j} - c_{A_i}), \Gamma_{j,i}(v,u)\rangle
    \end{split}
\end{align}

We remark that if we have $(u,v) \in E(A_i, A_j)$ there also
exists the edge $(v,u) \in E(A_i, A_j)$ since the graph is undirected and symmetric and $v\in A_j$ and $u\in A_i$. When we compute 
\begin{align}
    \begin{split}
        d(v,u) &=  \big\Vert c_{A_i} - c_{A_j} + \Gamma_{i,j}(u,v)\big\Vert - \big\Vert c_{A_i} - c_{A_j}  \big\Vert \\
        &=  \big\Vert c_{A_j} - c_{A_i} + \Gamma_{j,i}(v,u)\big\Vert - \big\Vert c_{A_j} - c_{A_i}  \big\Vert\\
        & = d(u,v). 
    \end{split}
\end{align}
Let $$\Gamma_i(u) = 1_B(u)\gamma_B^{A_i} - 1_{B^c}(u) \gamma_{B^c}^{A_i}$$ which implies that $\Gamma_{j,i}(v,u) = \Gamma_j(v) - \Gamma_i(u)$. Note that 
$$ \langle \nabla \Vert \cdot \Vert (c_{A_j} - c_{A_i}), -\Gamma_i(u)) \rangle = \langle \nabla \Vert \cdot \Vert (c_{A_i} - c_{A_j}), \Gamma_i(u)) \rangle.$$ With these properties in mind we can compute  
\begin{align}
    \begin{split}
        d(u,v) + d(v,u) &= 2d(u,v)\\
        &\approx 2 \langle \nabla \Vert\cdot\Vert (c_{A_j} - c_{A_i}), \Gamma_{j}(v) - \Gamma_{i}(u)\rangle\\
        &= 2 \langle \nabla \Vert\cdot\Vert (c_{A_j} - c_{A_i}), \Gamma_{j}(v) - \Gamma_{i}(u)\rangle \\
        &= 2 \langle \nabla \Vert\cdot\Vert (c_{A_j} - c_{A_i}), \Gamma_{j}(v)\rangle \\ & \quad+ 2 \langle \nabla \Vert\cdot\Vert (c_{A_i} - c_{A_j}), \Gamma_{i}(u) \rangle 
        \end{split}
\end{align}

Hence, we were able to split $d(u,v)+d(v,u)$ into the separate parts for $u$ and $v$. When we now collect these parts for every $\tilde{u}\in V$ we can compute the following 

\begin{align}\label{eq:GradientApproxOfR}
    \begin{split}
        &\sum_{(u,v)\in S} w(u,v)^\frac{p}{2} d(u,v) \\& \hspace{1.5cm}\approx 2 \sum_{\substack{\tilde{u}\in A \\ A\in \Pi}} \sum_{(\tilde{u}, v)\in S} w(\tilde{u},v)^\frac{p}{2} \langle \nabla \Vert \cdot \Vert (f(\tilde{u}) - f(v)), 1_B(\tilde{u})\gamma_B^{A} - 1_{B^c}\gamma_{B^c}^A(\tilde{u})\rangle 
    \end{split}
\end{align}

Note, that the gradient of the differentiable parts of $R$ is computed as $\nabla R_S(f)_u = \sum_{u\sim v \cap S} w(u,v)^\frac{p}{2} \nabla\Vert \cdot \Vert(f(u) - f(v))$. Then we can follow for \eqref{eq:GradientApproxOfR} 
\begin{align}\label{eq:TVapprox}
    \begin{split}
        \sum_{(u,v)\in S} w(u,v)^\frac{p}{2} d(u,v) &\approx \sum_{\substack{u\in A\\ A\in\Pi}} \langle \nabla R_S(f)_u, 1_B(u)\gamma_B^A - 1_{B^c}(u)\gamma_{B^c}^A \rangle \\ 
        &= 2\sum_{\substack{u\in A\\ A\in\Pi}} \langle \nabla R_S(f)_u, 1_B(u)(\gamma_B^A + \gamma_{B^c}^A) \rangle + \kappa_R \\ 
        &=  2\langle \nabla R_S(f), 1_{A\cap B}(\gamma_B^A + \gamma_{B^c}^A) \rangle + \kappa_R\\
        &=  2\langle \nabla R_S(f), \vec{\gamma} \rangle + \kappa_R
    \end{split}
\end{align}
with $\kappa_R = - 2\langle \nabla R_S(f),  \sum_{A\in\Pi} 1_A\gamma_{B^c}^A \rangle.$

Plugging \eqref{eq:approxDataTerm} and \eqref{eq:TVapprox} into \eqref{eq:energydiff} we get the approximation 
\begin{align} \label{eq:finalJDiff}
    \begin{split}
        &J(\tilde{f}_B)- J(f) \\
        & \hspace{2em}\approx \langle \nabla D(f,g) + \frac{\alpha}{p} \nabla R_S(f), \vec{\gamma}\rangle + \frac{\alpha}{p} w(A\cap B, A\cap B^c) \big\Vert \gamma_B^A + \gamma_{B^c}^A \big\Vert+ \kappa    \end{split}
\end{align}
with $\kappa = \kappa_D + \kappa_R$.\par 
As we used a Taylor expansion we can assume without loss of generality that the $\gamma_B^A, \gamma_{B^c}^A$ have all the same distance $\frac{\varepsilon}{2} > 0$ to their point $c_A$ for every $A\in \Pi$. Let us introduce 
\begin{align*}
    \bar{\gamma}_B^A = \frac{\gamma_B^A}{\Vert \gamma_B^A\Vert_2} \text{ and }  \bar{\gamma}_{B^c}^A = \frac{\gamma_{B^c}^A}{\Vert \gamma_{B^c}^A\Vert_2}
\end{align*}
such that we can write 
\begin{align*}
    \frac{\varepsilon}{2}\bar{\gamma}_{B}^A = \gamma_B^A \text{ and }\frac{\varepsilon}{2}\bar{\gamma}_{B^c}^A = \gamma_{B^c}^A.
\end{align*} 
    
Then follow that $\vec{\gamma} = \varepsilon\bar{\vec{\gamma}}$.
Now to make the approximation exact we could choose the scalars $\gamma_B = \frac{\varepsilon}{2}$ and $\gamma_{B^c} = \frac{\varepsilon}{2}$ with $\varepsilon > 0$. 
Thus, we can reformulate \eqref{eq:finalJDiff} as 
\begin{align} 
    \begin{split}
    \frac{J(\tilde{f}_B)- J(f)}{\varepsilon} &\approx \langle \nabla D(f,g) + \frac{\alpha}{p} \nabla R_S(f), \bar{\vec{\gamma}}\rangle + \frac{\alpha}{2p} w(A\cap B, A\cap B^c) \big\Vert \gamma_B^A + \gamma_{B^c}^A \big\Vert+ \kappa_\varepsilon\\ 
    \end{split}
\end{align}
with equality in the limit of $\varepsilon \rightarrow 0$. Note that $\kappa$ can be dropped for minimization since it is a constant.
With this fact and that
\begin{align}
    \argmin_{B\subset V} J(\tilde{f}_B) - J(f) = 
    \argmin_{B\subset V} \frac{J(\tilde{f}_B) - J(f)}{\varepsilon} 
\end{align}
for any $\varepsilon>0$ we can derive the final optimization problem  

\begin{align}\label{eq:finalMinimizationPoblem}
    \argmin_{B\subset V} J(\tilde{f}_B) - J(f) = \argmin_{B\subset V} \  \langle \nabla D(f,g) + \alpha \nabla R_S(f), \bar{\vec{\gamma}}\rangle + \frac{\alpha}{2p} w(A\cap B, A\cap B^c) \big\Vert \bar{\gamma}_B^A + \bar{\gamma}_{B^c}^A \big\Vert.
\end{align}

Interestingly the optimization is equivalent to minimize over the directional derivative $J'(f;\vec{\gamma})$ of $J$ at point $f$ in the direction $\vec{\gamma}$.
Thus, we indeed minimize the following 
\begin{align}\label{eq:derivedMinimizationProblem}
    \argmin_{B\in \mathcal{P}(V)} \ J'(f; \vec{\gamma})
\end{align} 
with $$ \bar{\vec{\gamma}} = \sum_{A\in \Pi} 1_{A\cap B} (\bar{\gamma}_B^A + \bar{\gamma}_{B^c}^A).$$

\section{Appendix: Minimum Partition Problem}\label{app:derivationMinimalPartition}
In this section we aim for finding minimum partitions for a given graph via a variational problem. For this special case we will use the $\ell_0$-total variation as a regularizer which is given as
\begin{align}
    \TV_0(f) = \sum_{(u,v) \in S(f)} w(u,v) = \sum_{(u,v)\in E} w(u,v)\1_{S_0}(f(v) - f(u))
\end{align}

with  
\begin{align}
    \1_{S_0}(x) = \begin{cases}
        1, & \text{if } \sum_{i=1}^d \vert x_j\vert > 0,\\
        0, & \text{else}.
    \end{cases}
\end{align}
Note, in literature this is sometimes also expressed as $\1_{S_0}(f(v) - f(u)) = \big[ f(v) \neq f(u)\big].$
The regularizer is non-differentiable and there also does not exist a directional derivative. But note that it is differentiable everywhere except in $0$. This means for an edge $(u,v)\in E$ with $f(u) \neq f(v)$ we can compute the derivative which is $0$ since it is constant $1$ everywhere but in $0$.
This minimization problem is non-convex, and thus, hard to solve and even if we find a solution it is likely to be not the global optimum. 

To solve this problem we again want to use an successive cut approach as the Cut Pursuit for $\Vert \cdot \Vert_{p;q}$. Therefore, we again want to find binary cuts to refine the set and then update the values. Let $\Pi$ be some partition and let $f = \sum_{A\in \Pi} 1_A c_A$ some function in $\mathcal{H}(V)$ that is constant on every $A\in \Pi$. Again we have some set $B\subset V$ and take some function 
$$\tilde{f}_B = \sum_{A\in\Pi}1_{A\cap B} d_{A\cap B} + 1_{A\cap B^c} d_{A\cap B^c}  $$
with vectors $d_{A\cap B} = c_A + \gamma_B^A$ and $d_{A\cap B^c} = c_A - \gamma_{B^c}^A$ for every $A\in \Pi$. We can also evaluate $\tilde{f}_B$ point-wise for $u\in A \in \Pi$ which is given as 
\begin{align*}
    \tilde{f}_B(u) = c_A + 1_B(u)\gamma_B^A - 1_{B^c}(u) \gamma_{B^c}^A.
\end{align*}
Once more we want to find a set $B \subset V$ that minimizes 
\begin{align}
    \argmin_{B\subset V} J(\tilde{f}_B) - J(f).
\end{align}

Therefore, we rewrite the difference. In the former section we already approximated it in \eqref{eq:approxDataTermDiff} as 

\begin{align}
    \mathcal{D}(\tilde{f}_B,f) = D(\tilde{f}_B,g) - D(f,g) \approx  \langle \nabla D(f,g),  \vec{\gamma} \rangle + \kappa
\end{align}
by Taylor expansion and selecting $d_{A\cap B} = c_A + \gamma_B$ and $d_{A\cap B^c} = c_A - \gamma_{B^c}$. 

Hence, we investigate the difference of the $\TV_0$ regularizer given as 
\begin{align}
    \TV_0(\tilde{f}_B) - \TV_0(f) = \sum_{(u,v) \in E} w(u,v) d(u,v)
\end{align}
with $d(u,v) = \1_{S_0} \big(\tilde{f}_B(v) - \tilde{f}_B(u)\big) - \1_{S_0} \big(f(v) -  f(u)\big)$.
\par First we concentrate on one particular partition $A\in\Pi$ and every edge $(u,v) \in E$ with $u,v\in A$. For these we have that $f(u) = f(v) = c_A$, and thus, they live in $S_0^c(f)$ and $\TV_0$ is non-differentiable at these edges. Here we get 
\begin{align*}
    d(u,v) &= \1_{S_0}\big(c_A + 1_B(v)\gamma_B^A - 1_{B^c}(v)\gamma_{B^c}^A - (c_A + 1_B(u)\gamma_B^A - 1_{B^c}(u) \gamma_{B^c}^A\big) - \1_{S_0}\big(c_A - c_A\big)\\
    &=\1_{S_0}\big( (1_B(v) -1_B(u))\gamma_B^A +  (1_{B^c}(u) - 1_{B^c}(v))\gamma_{B^c}^A\big)
\end{align*} 
We see that if $u,v \in B$ or $u,v\in B^c$ then $d(u,v) = 0$. Else if $u\in B$ and $v\in B^c$ or vice versa, then $d(u,v) = 1$. Thus, we can write 
\begin{align}
    \sum_{A\in \Pi}\sum_{(u,v) \in E(A,A)} w(u,v) d(u,v) &=  \sum_{(u,v) \in S^c} w(u,v) \vert 1_B(u) - 1_B(v)\vert\\
     &= 2w(B,B^c).
\end{align}
\par 
The second part is now to examine the leftover edges $(u,v) \in S(f)$ which is the set of edges between sets $A_i, A_j\in \Pi$. Let us consider some sets $A_i, A_j\in \Pi$ with $(A_i, A_j) \in E_r$ and fix some edges $(u,v)\in S$ with $u\in A_i$ and $v\in A_j$. Let us also denote
$$\vec{d_A}(u) = 1_B(u)\gamma_B^A - 1_{B^c}(u)\gamma_{B^c}^A$$ as a simplification. Then we get 
\begin{align*}
    d(u,v) = \1_{S_0}\big( c_{A_j} + \vec{d_{A_j}} - c_{A_i} - \vec{d_{A_i}} \big) - \1_{S_0}\big( c_{A_j} - c_{A_i} \big)
\end{align*}
which can again be approximated by Taylor expansion, since - as mentioned before - $\1_{S_0}$ is differentiable over these edges. But as we know $\nabla \1_{S_0} = 0$, consequently the approximation then yields 
\begin{align}
    d(u,v) \approx \langle \nabla \1_{S_0}(c_{A_j} - c_{A_i}), \vec{d_{A_j}} - \vec{d_{A_i}}\rangle = 0. 
\end{align}
Then we can directly deduce by the two cases that
\begin{align}
    \begin{split}
        \TV_0(\tilde{f}_B) - \TV_0(f) &=  2 w(B,B^c).
    \end{split}
\end{align}

Finally, with the same argumentation as in the former section we can deduce that the optimization problem we want to solve is 
\begin{align}\label{eq:appendixMinPartition}
    \argmin_{B\subset V} \ \langle \nabla D(f,g), \vec{\gamma}\rangle + \alpha w(B,B^c). 
\end{align}
}

\section{Appendix: Cut Pursuit}\label{app:cut}
To determine the derivative of the regularizer $R$ where it is differentiable for $q\geq p\geq 1$ we can calculate the derivative component-wise for each combination $u\in V$ and $j\in 1, \ldots, d$. Note that we have to distinguish between the cases for $q=p=1$ and $q\geq p \geq 1$ with $q>1$ due to the different differentiability properties.

Starting with $q=p=1$ we get
\begin{align*}
\begin{split}
\frac{\partial }{\partial f(u)_j}R_S(f) &= \frac{\partial }{\partial f(u)_j}\frac{1}{2} \sum_{\hat{u}\in V} \sum_{((\hat{u},v),\hat{\jmath})\in S} \sqrt{w(\hat{u},v)} \left|f(v)_{\hat{\jmath}}-f(u)_{\hat{\jmath}} \right|
\end{split}
\end{align*} 
First notice that we can drop all terms in $R$ where $u$ and $j$ are not contained. And since we work on undirected graphs $(u,v)\in S$ iff $(v,u)\in S$ and $w(u,v) = w(v,u)$. Due to the $q$-norm we thus have for each $(u,v)$ and $(v,u)$ the same term, such that we can add them up and sum up over all $(u,v)\in S$. This boils down to
\begin{align*}
\begin{split}
\frac{\partial }{\partial f(u)_j}R_S(f) &= \frac{\partial }{\partial f(u)_j}\frac{2}{2} \sum_{((u,v),j)\in S} \sqrt{w({u},v)} \left|f(v)_j-f(u)_j \right|\\
&=\sum_{ ((u,v),j) \in S} \sqrt{w(u,v)} \ \sign\big(f(u)_j-f(v)_j\big)
\end{split}
\end{align*} 

Now we consider the case $q\geq p\geq 1$ with $q>1$. 
\begin{align}
\begin{split}
\frac{\partial }{\partial f(u)_j}R_S(f) &= \frac{\partial }{\partial f(u)_j}\frac{1}{2p} \sum_{\hat{u}\in V} \sum_{(\hat{u},v)\in S} w(\hat{u},v)^\frac{p}{2} \left\Vert f(v)-f(\hat{u}) \right\Vert_q^p.
\end{split}
\end{align} 
 With the same ideas and properties from above we get to the simplified equation

\begin{align}
\begin{split}
\frac{\partial }{\partial f(u)_j}R_S(f) &= \frac{2}{2p}\frac{\partial }{\partial f(u)_j} \sum_{(u,v)\in S} w(u,v)^\frac{p}{2} \left\Vert f(v)-f(u) \right\Vert_q^p\\
&= \frac{1}{p}\frac{\partial }{\partial f(u)_j} \sum_{(u,v)\in S} w(u,v)^\frac{p}{2} \left\Vert f(v)-f(u) \right\Vert_q^p.
\end{split}
\end{align}
Notice that the derivative of the $q$-norm is calculated as
\begin{align}
\begin{split}
	\frac{\partial}{\partial f(u)_j} \left\Vert f(v)-f(u) \right\Vert_q &= \left(\frac{|f(v)_j - f(u)_j|}{\Vert f(v)-f(u)\Vert_q}\right)^{q-1} \frac{f(u)_j-f(v)_j}{\vert f(v)_j -f(u)_j \vert}.
\end{split}
\end{align}
By computing the inner and outer derivatives and use the derivative of the $q$-norm we can conclude
\begin{align}\label{eq:appderivative}
\begin{split}
\frac{\partial }{\partial f(u)_j} R_S(f) & =  \frac{p}{p}\sum_{(u,v)\in S} w(u,v)^\frac{p}{2}  \left\Vert f(v)-f(u) \right\Vert_q^{p-1} \frac{\partial }{\partial f(u)_j} \left\Vert f(v)-f(u) \right\Vert_q\\
& =\sum_{(u,v)\in S} w(u,v)^\frac{p}{2}  \left\Vert f(v)-f(u) \right\Vert_q^{p-1} \left(\frac{|f(v)_j - f(u)_j|}{\Vert f(v)-f(u)\Vert_q}\right)^{q-1} \frac{f(u)_j-f(v)_j}{\vert f(v)_j -f(u)_j \vert}\\
& = \sum_{(u,v)\in S} w(u,v)^\frac{p}{2}  \left\Vert f(v)-f(u) \right\Vert_q^{p-q} {|f(v)_j - f(u)_j|}^{q-2} (f(u)_j-f(v)_j).
\end{split}
\end{align}

Now let us consider three special cases and combine these results with the results of Appendix \ref{app:pqlap}. First we consider $p=q=1$ where 
\begin{align*}
\Delta_{w,1} f(u) &=\sum_{((u,v),j))\in S} \sqrt{w(u,v)} \Big( \sign\big(f(u)_j -f(v)_j\big) \Big)_{j=1}^d \\
&= \Bigg( \frac{\partial }{\partial f(u)_j} R_S(f) \Bigg)_{j=1}^d.
\end{align*}
Second we look at $p=q>1$
\begin{align*}
\Delta_{w,p} f(u) &=\sum_{(u,v)\in S} w(u,v)^\frac{p}{2} \Bigg( (f(v)_j -f(u)_j) \vert  (f(v)_j -f(u)_j) \vert^{p-2} \Bigg)_{j=1}^d \\
&= \Bigg( \frac{\partial }{\partial f(u)_j} R_S(f) \Bigg)_{j=1}^d.
\end{align*}

Finally, consider $q=2$ $p\geq 1$

\begin{align*}
\Delta_{w,p} f(u) &=\sum_{(u,v)\in S} w(u,v)^\frac{p}{2} \Vert f(v)-f(u)\Vert_2^{p-2} \Bigg( f(v)_j -f(u)_j  \Bigg)_{j=1}^d\\
& =\sum_{(u,v)\in S} w(u,v)^\frac{p}{2} \Vert f(v)-f(u)\Vert_2^{p-2} (f(v)-f(u)) = \frac{\partial }{\partial f(u)} R_S(f). 
\end{align*}

\section{Appendix: Analysis of directional derivatives in the non-smooth case}\label{app:C}
In the following we investigate the properties of the variational model \eqref{eq:generalproblem} for different choices of $p,q\geq 1$.  First, we discuss how we deduce an efficient optimization strategy for the latter model by describing the idea of Cut Pursuit in Section \ref{ssec:cutpursuit}. Subsequently, we show how to solve the two related subproblems, i.e., a minimum partition problem and a denoising problem, in Section \ref{ssec:graph_cuts} and \ref{ssec:red_problem}, respectively.
Hence, the only non-trivial cases to discuss in the following are for $q \geq p = 1$. For this let us denote the directional derivative of $R$ in direction $\vec{d} \in \mathcal{H}(V)$ by $R'(f;\vec{d}) := \langle \nabla R,\vec{d}\rangle$.
\paragraph{\bfseries{Case 1: $q=p=1$}}\mbox{}\par
In this case the regularization function in \eqref{eq:regularizer} simply becomes
$$ R(f) \ = \  \frac{1}{2}\sum_{(u,v)\in E}\sqrt{w(u,v)} \sum_{j=1}^{d}   \vert f(v)_j - f(u)_j \vert,$$ 
which is not differentiable along edges $(u,v) \in E$ where $f(u)_j = f(v)_j$ for some $j \in \lbrace 1,\ldots,d\rbrace$. 
In order to investigate the directional derivatives of the regularization function $R$ based on the choice $p$ and $q$ we introduce the following notation. Let us define by $S_1 := S_1(f) = \big\{((u,v),j) \in E \times \lbrace 1,\ldots,d\rbrace \ \big| \ f(u)_j \neq f(v)_j\big\}$ the set of points for which $R$ is differentiable. Then, we are able to partition our set of vertices $V = S_1 \cup S_1^c$ and thus restrict our discussion of the regularization functional $R$ to the nontrivial terms, i.e., the non-differentiable part $R_{S^c}$ with $R(f) = R_{S_1}(f) + R_{S_1^c}(f)$.
Computing the directional derivative for some direction $\vec{d} \in \mathcal{H}(V)$ can be done component-wise for every $((u,v),j)\in S_1^c$ and leads to 
\begin{align}\label{eq:derive_directional_derivative}
R_{S_1^c}'(f;\vec{d}) \ = \ \frac{1}{2}\sum_{((u,v),j)\in S_1^c} \sqrt{w(u,v)} \big\vert \vec{d}(v)_j -  \vec{d}(u)_j \big\vert.
\end{align}\par 

\paragraph{\bfseries{Case 2: $q>p=1$}}\mbox{}\par
Using the notation in Section \ref{ss:first_order_differential} the regularization functional in \eqref{eq:regularizer} can be written as
$$
R(f) \ = \ \frac{1}{2}\sum_{(u,v)\in E} \Vert \partial_vf(u)\Vert_q \ = \ \frac{1}{2}\sum_{(u,v)\in E} \left( \sum_{j=1}^d w(u,v)^\frac{q}{2} \vert f(v)_j - f(u)_j\vert^q\right)^\frac{1}{q}.
$$
It gets clear that this term is not differentiable iff $\Vert \partial_v f(u)\Vert_q = 0$, i.e., $f(u)_j-f(v)_j=0$ for every $j = 1,\ldots,d$, for some $(u,v) \in E$. In this case we can define $S_q(f) = \big\{(u,v)\in E \ \big| \ \Vert \partial_vf(u)\Vert_q \neq 0\big\}$ and thus the directional derivative can be computed for each edge $(u,v)\in S_q^c$ and is given by
\begin{align}\label{eq:dirregp}
R_{S_q^c}'(f; \vec{d})  = \frac{1}{2} \sum_{(u,v)\in S_q^c} \sqrt{w(u,v)} \left(\sum_{j=1}^{d} \vert  \vec{d}(v)_j -  \vec{d}(u)_j \vert^q\right)^\frac{1}{q}.
\end{align}

To summarize our observations above, we can deduce that for $q\geq p>1$ the regularizer is differentiable everywhere, and thus $S=\emptyset$. In this case the directional derivative of $J$ in direction $\vec{d}$ is simply given as 
\begin{align}\label{eq:dirderivdiff}
J'(f;\vec{d}) = \langle \nabla J, \vec{d} \rangle.
\end{align} 
and the gradient can be computed with \eqref{eq:derivative}.
For $q>p=1$ the functional $J$ is not differentiable in every vertex $v \in V$ but the directional derivative exists in every point.\par 
To conclude the discussion of the proposed denoising model we want to emphasize the relation of the derivative in \eqref{eq:derivative} to the graph $p$-Laplacian operators defined in Section \ref{s:graphs}.

\paragraph{\bfseries{Case 1: $p=q$}}\mbox{}\par 
In this case the derivative of the regularizer on the differentiable part $R_S$ is given for any $j \in \lbrace 1,\ldots,d\rbrace$ for $q=1$ as
\begin{align}\label{eq:derivativepq}
\begin{split}
\frac{\partial }{\partial f(u)_j}R_{S_1}(f) \ = \ \sum_{ ((u,v),j)\in S_1} \sqrt{w(u,v)} \frac{f(u)_j-f(v)_j}{|f(v)_j - f(u)_j|}
\end{split}
\end{align}
and for $q>1$ as
\begin{align}\label{eq:derivativepq2}
\begin{split}
\frac{\partial }{\partial f(u)_j}R_{S_q}(f) \ = \ \sum_{(u,v)\in S_q} w(u,v)^\frac{p}{2} {|f(v)_j - f(u)_j|}^{p-2} (f(u)_j-f(v)_j).
\end{split}
\end{align}
The above expression is exactly the definition of the \textit{anisotropic} graph $p$-Laplacian as introduced in \eqref{eq:pq-laplace_aniso_multi}.

\paragraph{\bfseries{Case 2: $q=2, p\geq 1$}}\mbox{}\par
In this case the derivative of the regularizer on the differentiable part $R_S$ is given for any $j \in \lbrace 1,\ldots,d\rbrace$
\begin{align*}
\frac{\partial }{\partial f(u)_j}R_{S_q}(f) 
&=  \sum_{(u,v)\in S_q} w(u,v)^\frac{p}{2}  \left\Vert f(v)-f(u) \right\Vert_2^{p-2}(f(u)_j-f(v)_j),
\end{align*}
which can formally be written for every $j=1,\ldots,d$ as the vector
\begin{align}\label{eq:derivativeq2}
\frac{\partial }{\partial f(u)}R_{S_q}(f) = \sum_{\substack{v \in V\\ (u,v) \in S_q}} w(u,v)^\frac{p}{2}  \left\Vert f(v)-f(u) \right\Vert_2^{p-2}(f(u)-f(v)).
\end{align}
This is exactly the \textit{isotropic} graph $p$-Laplacian as introduced in \eqref{eq:p-laplace_iso_multi}.

\section{Appendix: Projection onto $p^*,q^*$-balls}\label{app:pqball}

In the following we will see that for different $p$ and $q$ combination we will get different proximity operators. We will distinguish between three cases.

\paragraph{\bfseries{Case 1:} $q>1, p=1$}\mbox{}\par
For the special case $p=1$ we get $p^* = \infty$, and thus the dual norm becomes $\Vert y \Vert_{{\infty,q^*}} = \max_{y(u,v)} \big\{ \big\Vert y(u,v)\big\Vert_{q^*} \big\}$. Then $$
B_{\infty; q^*} = \Big\{ y \in X^* \big| \ \Vert y \Vert_{\infty,q^*} \leq \alpha\Big\} =  \Big\{ y \in X^* \big| \ \big\Vert y(u,v) \big\Vert_{q^*} \leq \alpha, \forall (u,v)\in E\Big\}
$$
The proximity operator for $p=1$, $q>1$ and  every $(u,v)\in E$ is just a projection of every $y(u,v)$ onto the ball $B_{q^*}(\alpha)$.

In conclusion we get the proximity operator of $F^*$ as
\begin{align}
\text{prox}_{\tau F^*}(z) &= \arg\min_{y\in X^*} \left\{ \frac{1}{2\tau}\Vert y-z \Vert_2^2 + F^* (y) \right\}\\
&= \text{proj}_{B_{\infty,q^*}(\alpha)}(z) \\
&=  \Big(\frac{\alpha \ z(u,v)}{\max(\alpha, \Vert z(u,v)\Vert_{q^*})}\Big)_{(u,v)\in E}.
\end{align}
\paragraph{\bfseries{Case 2:} $q=1, p=1$}\mbox{}\par
When $q=1$ and $p=1$, then $q^* = \infty$ and $p^* = \infty$. Then the ball becomes 
\begin{align*}
B_{\infty; \infty}(\alpha) &=  \big\{ y \in \mathcal{H}(E) \big| \ \Vert y(u,v) \Vert_{\infty} \leq \alpha, \forall (u,v)\in E\big\}\\ &=  \big\{ y \in  \mathcal{H}(E) \big| \ | y(u,v)_j | \leq \alpha, \forall (u,v)\in E, j\in[1,d]\big\},
\end{align*}
from which follows that the proximity operator becomes
\begin{align*}
\text{prox}_{\tau F^*}(z)&= \proj_{B_{\infty;\infty}(\alpha)}(z) \\ 
&=  \Bigg(\frac{\alpha z(u,v)_j}{\max(\alpha, | z(u,v)_j|)}\Bigg)_{(u,v,j)\in E\times [1,d]}.
\end{align*}

\dt{
\paragraph{\bfseries{Case 3:} $q \geq 1,1 < p < \infty$}\mbox{}\par
In this special case one has to compute the projection onto the $p^*,q^*$-ball numerically as there does not exist any known closed-form solution, except for the case of $p=2$.
Since the constraints are smooth one is able to use a standard Newton method for computing this projection.
Note that usually this case is not relevant in most applications from imaging or machine learning in contrast to cases 1 and 2 above.
}

\dt{
\paragraph{\bfseries{Case 4:} $q\geq 1, p = \infty$}\mbox{}\par
In this case one has to project onto $1,q*$-balls. To compute these projections there exist efficient numerical algorithms, e.g., see \cite{duchi2008efficient,sra2011fast}.
}
\section{Appendix: Regularity of $J'(f; \1_B)$}\label{app:submod}

To show \dt{the regularity of $J'(f; \1_B)$} as described in \cite{kolmogorov} we have to investigate the property for directional derivative of the non-differentiable part of the regularizer given as 
$$R_S'(f,\1_B) = \sum_{ ((u,v),j) \in S_1^c} \sqrt{w(u,v)}|\1_B(u)_j - \1_B(v)_j|.$$ 
This can be translated into the notation of \cite{kolmogorov} with $$E(\1_B(u)_j,\1_B(v)_j) = \sqrt{w(u,v)}|\1_B(u)_j - \1_B(v)_j|$$
for every $((u,v),j)\in S_1^c$.
Now we have to show that $E(0,0) + E(1,1) \leq E(1,0) + E(0,1)$ which is satisfied since
\begin{align*}
	E(1,1) = E(0,0) = 0,\\
	E(0,1) = E(1,0) = \sqrt{w(u,v)}
\end{align*}
and $w(u,v)\geq 0$. Thus, $J'(f;\1_B)$ is regular, respectively submodular.

\end{document}